\documentclass[12pt]{article}
\usepackage[]{amsmath,amssymb}
\usepackage[]{amsthm,enumerate}
\usepackage[]{comment}
\usepackage{url}

\newtheorem{theorem}{Theorem}[section]
\newtheorem{proposition}[theorem]{Proposition}
\newtheorem{lemma}[theorem]{Lemma}
\newtheorem{corollary}[theorem]{Corollary}
\theoremstyle{definition}

\theoremstyle{remark}
\newtheorem{remark}[theorem]{Remark}

\newcommand{\QQ}{\mathbb{Q}}
\newcommand{\ZZ}{\mathbb{Z}}
\newcommand{\RR}{\mathbb{R}}

\newcommand{\cC}{\mathcal{C}}

\newcommand{\cP}{\mathcal{P}}
\DeclareMathOperator{\wt}{wt}
\DeclareMathOperator{\supp}{supp}
\DeclareMathOperator{\rank}{rank}
\begin{document}

%\title{Quasi-unbiased Hadamard matrices and weakly unbiased
%Hadamard matrices: a coding-theoretic approach
%(Version 17.13)}

\title{Quasi-unbiased Hadamard matrices and weakly unbiased
Hadamard matrices: a coding-theoretic approach}
\date{\today}

\author{
Makoto Araya\thanks{Department of Computer Science,
Shizuoka University,
Hamamatsu 432--8011, Japan.
email: araya@inf.shizuoka.ac.jp},
Masaaki Harada\thanks{
Research Center for Pure and Applied Mathematics,
Graduate School of Information Sciences,
Tohoku University, Sendai 980--8579, Japan.
email: mharada@m.tohoku.ac.jp.}
and
Sho Suda\thanks{Department of Mathematics Education, Aichi University of Education,
1 Hirosawa, Igaya-cho, Kariya 448--8542, Japan.
email: suda@auecc.aichi-edu.ac.jp}
}

\maketitle

%\tableofcontents

\vspace*{-1cm}
\begin{center}
{\bf Dedicated to Professor Satoshi Yoshiara on his 60th birthday}
\end{center}

\begin{abstract}
This paper is concerned with quasi-unbiased Hadamard matrices
and weakly unbiased Hadamard matrices,
which are generalizations of unbiased Hadamard matrices, equivalently
unbiased bases.
These matrices are studied from the viewpoint of coding theory.
As a consequence of a coding-theoretic approach, 
we provide upper bounds on the number of mutually
quasi-unbiased Hadamard matrices.
We give classifications of a certain class of self-complementary 
codes for modest lengths.  
These codes give quasi-unbiased Hadamard matrices
and weakly unbiased Hadamard matrices.
Some modification of the notion of
weakly unbiased Hadamard matrices
is also provided.
\end{abstract}

%%%%%%%%%%%%%%%%%%%%%%%%%%%%%%%%%%%%%%%%%%%%%
\section{Introduction}

Two Hadamard matrices $H,K$ of order $n$ are said to be 
{\em unbiased} if $(1/\sqrt{n})HK^{T}$ is also 
a  Hadamard matrix of order $n$, where
$K^T$ denotes the transpose of $K$.
This means that the absolute value of any entry of $HK^{T}$
is $\sqrt{n}$.
The notion of unbiased Hadamard matrices is 
essentially the same as that of unbiased bases in $\RR^n$.
It is a fundamental problem to determine the 
maximum size among sets of mutually unbiased Hadamard matrices.
Much work has been done concerning this fundamental problem
(see~\cite{BK10}, \cite{BSTW}, 
\cite{CCKS}, \cite{CS73}, \cite{DGS2},
\cite{HKO}, \cite{KSS}, \cite{LMO}, \cite{WB}).

%For instance, from the view point of algebraic combinatorics, it is known that the existence of real mutually unbiased bases (mutually unbiased Bush-type Hadamard matrices, resp.) is equivalent to that of a certain association scheme of class $4$~\cite{LMO} (of class $5$~\cite{KSS}, resp.), and real mutually unbiased bases are related to spin models by restricting~\cite[Corollary 8.3]{GR} to the matrices with real entries. 
%One reason for the interest in them
%is to have applications to other areas like
%the determination of the state~\cite{I}, the reconstruction of the quantum
%state~\cite{WF} and the entanglement detection~\cite{SHBAH}.  
%{\bf (To Suda:  This part is from our previous paper.  Please
%check that this is for real (not complex))}
%%%%%%%%%%%%%%

% (Here comes from~\cite{LMO})
%Mutually unbiased bases in $\CC^n$ play an important role in quantum
%information theory, and much work has been done concerning
%mutually unbiased bases in $\CC^n$.
%Recently, there has been interest in unbiased bases in $\RR^n$.
%One reason for the interest in them is useful for understanding 
%mutually unbiased bases in $\CC^n$.
% (Here comes from~\cite{LMO})

%%%%%%%%%%%%%%
Recently, the notion of unbiased Hadamard matrices has been 
generalized in~\cite{BK10}, \cite{HKO} and~\cite{NS} 
(see also Section~\ref{sec:Hmat} for the motivation).
Two weighing matrices $W_1,W_2$ of order $n$ and weight $k$ are 
{\em unbiased} if 
$(1/\sqrt{k}) W_1 W_2^T$ is a weighing matrix
of order $n$ and weight $k$~\cite{HKO}.
As a natural generalization, quasi-unbiased weighing matrices
are defined in~\cite{NS} as follows:
$W_1,W_2$ are {\em quasi-unbiased for parameters $(n,k,l,a)$} if 
$(1/\sqrt{a}) W_1 W_2^T$ is a weighing matrix
of weight $l$.
In this paper, we restrict our investigation to the case where
$W_1,W_2$ are Hadamard, in order to adopt a coding-theoretic approach. 
% to study Hadamard matrices by considering the corresponding Hadamard codes. 
% This approach was already taken for binary linear codes~\cite{HS} and
% thus we extend this method to binary codes including non-linear codes.
% This approach was already taken for some binary linear 
% codes containing the first order binary Reed--Muller code 
% as a subcode~\cite{HS} and~\cite{NS}.
% In this paper, the approach is done extensively by considering
% a class of  binary self-complementary codes.
We say that Hadamard matrices $H,K$ are 
{\em quasi-unbiased} Hadamard matrices with parameters $(l,a)$
if $(1/\sqrt{a})HK^{T}$ is a weighing matrix of weight $l$.
Note that the absolute value of any entry of $HK^{T}$
is $0$ or $\sqrt{a}$.
Two Hadamard matrices $H,K$ are {\em weakly unbiased} if
$a_{ij} \equiv 2 \pmod 4$ for $i,j \in \{1,2,\ldots,n\}$ and
$|\{ |a_{ij}| \mid i,j \in \{1,2,\ldots,n\}\}| \le 2$, where
$a_{ij}$ denotes the $(i,j)$-entry of $HK^T$~\cite{BK10}.
%%%%%%%%%
Hadamard matrices $H_1,H_2,\ldots,H_f$ are said to be 
{\em mutually} unbiased (resp.\ quasi-unbiased and weakly unbiased) Hadamard matrices
if any pair of two distinct  of them is unbiased 
 (resp.\ quasi-unbiased and weakly unbiased) 
Hadamard matrices.
In this paper, by adopting a coding-theoretic approach,
we study the maximum size among sets of mutually 
quasi-unbiased Hadamard matrices and 
weakly unbiased Hadamard matrices.
% Similar to mutually unbiased Hadamard matrices, 
% this is a fundamental problem.

%All feasible parameter sets for quasi-unbiased 
%(weakly unbiased) Hadamard matrices are examined for orders up to $48$.
%We give a coding-theoretic approach to the construction of
%a set of mutually quasi-unbiased (weakly unbiased)
%Hadamard matrices.
%As an example, we investigate  
%the maximum size among sets of mutually quasi-unbiased 
%(weakly unbiased) Hadamard  matrices.
%Mutually quasi-unbiased (weakly unbiased)
%Hadamard matrices are constructed from some 
%binary self-complementary codes.
%Finally, some modification of the notion of
%weakly unbiased Hadamard matrices is given.
%We derive some results which are analogy to those of quasi-unbiased 
%(weakly unbiased) Hadamard matrices.

%%%%%%%%%%%%%%%%%%%%%%%%%%%%%%
This paper is organized as follows.
In Section~\ref{sec:Pre}, we give definitions and some known results
of Hadamard matrices, codes and association schemes
used in this paper.
In Section~\ref{sec:B}, we give two upper bounds on the number of
codewords of binary self-complementary codes 
(Theorems~\ref{thm:SCbound} and~\ref{thm:SCLP}).
% using Delsarte's theory~\cite{D}.
% These bounds are used  
% % give upper bounds on the size of sets of mutually quasi-unbiased 
% % Hadamard matrices 
% % in Corollary~\ref{cor:bound} and Corollary~\ref{cor:weakIIbound}.
% in Sections~\ref{sec:F2} and~\ref{sec:weakF2}.
In Sections~\ref{sec:basic} and~\ref{sec:F2},
we study the existence of mutually quasi-unbiased Hadamard matrices.
% Some constrcution methods of quasi-unbiased Hadamard matrices are 
% given in Section~\ref{sec:basic}.
% Restrictions on parameters for quasi-unbiased Hadamard matrices
% are provided.
% We give a certain observation, which is useful for constructing 
% quasi-unbiased Hadamard matrices by a computer search.
% All feasible parameter sets for quasi-unbiased 
% Hadamard matrices are listed in Table~\ref{Tab:Par}
% for orders up to $48$.
In Section~\ref{sec:F2},
% we give a coding-theoretic approach to the existence of
% a set of mutually quasi-unbiased Hadamard matrices.
% More precisely, 
we characterize binary self-complementary
$(n,2fn)$ codes whose existence is equivalent to that of 
a set of $f$ mutually quasi-unbiased 
Hadamard matrices of order $n$ (Theorem~\ref{thm:F2}).
By Theorems~\ref{thm:SCbound} and~\ref{thm:SCLP},
this characterization derives 
upper bounds on the size of sets of mutually quasi-unbiased 
Hadamard matrices (Theorem~\ref{thm:bound}).
For modest lengths, we also give classifications of some binary 
self-complementary codes satisfying the conditions 
in Theorem~\ref{thm:F2}, in order to construct 
mutually quasi-unbiased Hadamard matrices.
In analogy to the case of quasi-unbiased Hadamard matrices,
Sections~\ref{sec:weakly} and~\ref{sec:weakF2}
study the existence of weakly unbiased Hadamard matrices.
% All feasible parameter sets for weakly unbiased 
% Hadamard matrices are listed in Table~\ref{Tab:weakP}
% for orders up to $48$.
Theorem~\ref{thm:weakUB} shows that the size of a set of 
mutually weakly unbiased Hadamard matrices 
is at most $2$.
Similar to Theorem~\ref{thm:F2},
we characterize binary self-complementary
codes whose existence is equivalent to that of a pair of weakly unbiased 
Hadamard matrices of order $n$ (Theorem~\ref{thm:weakF2}).
For modest lengths, we also give classifications of some binary 
self-complementary codes satisfying the conditions 
in Theorem~\ref{thm:weakF2}, in order to construct 
weakly unbiased Hadamard matrices.
Finally, in Section~\ref{sec:II}, 
as a modification of the notion of
weakly unbiased Hadamard matrices, 
we introduce the notion of Type~II weakly unbiased Hadamard matrices.
We establish results which are analogue to those of quasi-unbiased 
Hadamard matrices and weakly unbiased Hadamard matrices.

All computer calculations in this paper were
done by programs in the algebra software
{\sc Magma}~\cite{Magma} and
programs in the language C.

%%%%%%%%%%%%%%%%%%%%%%%%%%%%%%%%%%%%%%%%%%%%%%%%%%%%%%
\section{Preliminaries}\label{sec:Pre}
In this section, we give definitions and some known results
of Hadamard matrices, codes and association schemes used in this paper.

%%%%%%%%%
\subsection{Hadamard matrices}\label{sec:Hmat}

A {\em Hadamard matrix} of order $n$ is an $n \times n$ $(1,-1)$-matrix
%whose entries are from $\{ 1,-1 \}$ 
$H$ such that $H H^T = nI_n$,  
where $I_n$ is the identity matrix of order $n$.
It is well known that the order $n$ is necessarily $1,2$, or a multiple of $4$.
Throughout this paper,
we assume that $n \ge 2$ unless otherwise specified.
A {\em weighing matrix} of order $n$ and weight $k$ is an
$n \times n$ $(1,-1,0)$-matrix $W$ such that
$W W^T=kI_n$.
Of course, a weighing matrix of order $n$ and weight $n$ is a Hadamard matrix.
The two distinct rows $r_i,r_j$ $(i \ne j)$
of a weighing matrix $W$ of order $n$ and weight $k$
are % mutually 
orthogonal under the standard inner product
$r_i \cdot r_j$
and $W$ contains exactly $k$ nonzero entries in each row and each column. 
%%%%%%%%%%%%%%%
Two Hadamard matrices $H,K$ are said to be {\em equivalent}
if there exist $(1,-1,0)$-monomial matrices $P, Q$ with $K = PHQ$.
All Hadamard matrices of orders up to $32$ have been classified
(see~\cite[Chap.~7]{HSS-OA} for orders up to $28$ and 
\cite{KT13} for order $32$, see also~\cite{Hadamard}).
The numbers of inequivalent Hadamard matrices
of orders $4,8,12,16,20,24,28,32$ are 
$1, 1, 1, 5, 3, 60, 487, 13710027$, respectively.

Two Hadamard matrices $H,K$ of order $n$ are said to be 
{\em unbiased} if $(1/\sqrt{n})HK^{T}$ is also 
a  Hadamard matrix of order $n$, where
$K^T$ denotes the transpose of $K$.
This means that the absolute value of any entry of $HK^{T}$
is $\sqrt{n}$.
Hadamard matrices are said to be 
{\em mutually} unbiased Hadamard matrices
if any pair of two distinct of them is unbiased 
Hadamard matrices.
The existence of $f$ mutually unbiased Hadamard matrices of order $n$
is equivalent
to that of $f+1$ mutually unbiased bases in 
$\RR^n$~\cite[Observation~2.1]{BSTW}.
It is a fundamental problem to determine the 
maximum size among sets of mutually unbiased Hadamard  matrices of
order $n$.
% For example, it follows from~\cite[Obervation~2.1]{BSTW} 
% and~\cite[Theorem~5.2]{DGS2} that the maximum size is 
% at most $n/2$.
%%%%%%%%%%%%%%%%%%%%%%%
%%%%  do not delete here  %%%%%%%%% 
%For example, the following upper bound on the maximum size is known.
%The $(f+1)n$ row vectors of the
%identity matrix of order $n$ and $f$ mutually unbiased Hadamard matrices
%of order $n$ form a set of $(f+1)n$ lines of $\RR^n$
%through the origin with angles $0,1/\sqrt{n}$.
%Upper bounds on the number of lines through the origin with a few angles 
%are given in~\cite[Table~1]{DGS2}.
%By setting $\alpha=0$ and $\beta=1/n$ in the second case
%of~\cite[Table~1]{DGS2}, we obtain $(f+1)n\leq n(n/2+1)$,
%which gives $f\leq n/2$.
%%%%  do not delete here  %%%%%%%%% 
%%%%%%%%%%%%%%%%%%%%%%%
%%%
%For example, combined with~\cite[Table~1]{DGS2},
%it is established that the number of  mutually unbiased Hadamard matrices
%of order $n$ is at most $n/2$ from~\cite[Observation~2.1]{BSTW}.
For example, it follows from~\cite[Observation~2.1]{BSTW}
and~\cite[Table~1]{DGS2} that $f \le n/2$.

Recently, generalizations of unbiased Hadamard matrices have
been presented~\cite{BK10}, \cite{HKO} and~\cite{NS}.
Two weighing matrices $W_1,W_2$ of order $n$ and weight $k$ are 
{\em unbiased} if 
$(1/\sqrt{k}) W_1 W_2^T$ is a weighing matrix
of weight $k$~\cite{HKO}.
As a natural generalization, quasi-unbiased weighing matrices
are defined in~\cite{NS} as follows:
$W_1,W_2$ are  {\em quasi-unbiased for parameters $(n,k,l,a)$} if 
$(1/\sqrt{a}) W_1 W_2^T$ is a weighing matrix 
of order $n$ and weight $l$.
%This notion was introduced to give an answer to the problem
%posted by Best, Kharaghani and Ramp~\cite{BKR}, namely,
%do there exist Hadamard matrices $H_1,H_2,\ldots,H_{2^{2t+1}}$ of order
%$2^{2t+1}$ such that
%the entries of $H_i H_j^T$ ($i\neq j$) have absolute values 
%$0$ or $2^{t+1}$?
This notion was introduced to show that Conjecture~32 in~\cite{BKR}
is true.
In addition, a set of $f$ mutually quasi-unbiased
weighing matrices 
for parameters $(n,k,l,a)$ implies a set of $f-1$ mutually unbiased
weighing matrices of order $n$ and weight $l$.
%%%%%%% Quasi
In this paper, we restrict our investigation to the case where
$W_1,W_2$ are Hadamard in the definition of quasi-unbiased weighing
matrices, in order to adopt a coding-theoretic approach. 
Our restriction is also natural for a consideration of a certain
generalization of the situation in~\cite[Conjecture~32]{BKR}.
We say that Hadamard matrices $H,K$ of order $n$ are
{\em quasi-unbiased} Hadamard matrices with parameters $(l,a)$
if $(1/\sqrt{a})HK^{T}$ is a weighing matrix of weight $l$.
Equivalently,  the absolute value of any entry of $HK^{T}$
is $0$ or $\sqrt{a}$.
%%%%%%% Weakly
Two Hadamard matrices $H,K$ are {\em weakly unbiased} if
$a_{ij} \equiv 2 \pmod 4$ for $i,j \in \{1,2,\ldots,n\}$ and
$|\{ |a_{ij}| \mid i,j \in \{1,2,\ldots, n\}\}| \le 2$, where
$a_{ij}$ denotes the $(i,j)$-entry of $HK^T$~\cite{BK10}.
A pair of weakly unbiased Hadamard matrices is constructed
from that of unbiased quaternary complex Hadamard 
matrices satisfying a certain condition~\cite[Theorem~14]{BK10}.

Throughout this paper,
in the presentation of Hadamard matrices, 
we use $+,-$ to denote $1,-1$, respectively.

%%%%%%%%%%%%%%%%%%%%%%%%%%%%
\subsection{Binary codes and $\ZZ_4$-codes}\label{sec:2.2}

Let $\ZZ_{2k}\ (=\{0,1,\ldots,2k-1\})$ denote the ring of integers
modulo $2k$.
A {\em $\ZZ_{2k}$-code} $C$ of length $n$ 
is a subset of $\ZZ_{2k}^n$.
A $\ZZ_{2k}$-code $C$ is called {\em linear}
if $C$ is a $\ZZ_{2k}$-submodule of $\ZZ_{2k}^n$.
Usually $\ZZ_2$-codes are called {\em binary}.
In this paper, we deal with binary codes and $\ZZ_4$-codes.
In addition, codes mean binary codes unless otherwise specified. 

The (Hamming) {\em distance} $d(x,y)$ between
two vectors $x$ and $y$ of $\ZZ_{2k}^n$ is the number of 
components in which they differ.
Let $C$ be a $\ZZ_{2k}$-code of length $n$.
A vector of $C$ is called a {\em codeword} of $C$.
The {\em minimum} (Hamming) {\em distance} $d_H(C)$ of $C$ is the smallest (Hamming)
distance among all pairs of two distinct codewords of $C$.
% Two $\ZZ_{2k}$-codes are {\em equivalent} if one can be obtained from the
% other by permuting the coordinates and (if necessary) changing
% the signs of certain coordinates.
A {\em generator matrix} of a linear $\ZZ_{2k}$-code 
is a matrix such that the rows generate the code and no
proper subset of the rows of the matrix generates the code.
For a linear $\ZZ_{2k}$-code $C$
of length $n$ and vectors $x_1,x_2,\ldots,x_s \in\ZZ_{2k}^n$,
we denote by $\langle C, x_1,x_2,\ldots,x_s \rangle$ 
the linear $\ZZ_{2k}$-code generated by the
codewords of $C$ and $x_1,x_2,\ldots,x_s$.
Let $S_n$ denote the symmetric group of degree $n$.
For $x \in \ZZ_{2k}^n$ and $\sigma \in S_n$, 
let $\sigma(x)$ denote the vector obtained from $x$,
by the permutation $\sigma$ of the coordinates.
For $j \in \{1,2,\ldots,n\}$,
let $\tau_j(x)$ denote the vector obtained from $x$,
by changing the sign of the $j$-th coordinate.
In addition, set 
$\sigma(C)=\{\sigma(c) \mid c \in C\}$ and
$\tau_j(C)=\{\tau_j(c) \mid c \in C\}$.

%% binary 
A binary $(n,M)$ code is a binary code of length $n$
with $M$ codewords.
A binary $(n,M,d)$ code is a binary $(n,M)$ code with minimum distance $d$.
A binary  $[n,k]$ code means a binary linear code of length $n$
with $2^k$ codewords.
A binary  $[n,k,d]$ code means a binary $[n,k]$ code
with minimum distance $d$.
The {\em distance distribution} of a binary code $C$ of length $n$
is defined as $(A_0(C),A_1(C),\ldots,A_n(C))$,
where
\[
A_i(C)=\frac{1}{|C|} |\{(x,x') \mid x,x' \in C, d(x,x')=i \}|
\quad (i=0,1,\ldots,n).
\]
A binary code $C$ is called {\em self-complementary} if
$x+\mathbf{1}\in C$ for any $x \in C$, 
where $\mathbf{1}$ denotes the all-one vector.
Two binary $(n,M,d)$ codes $C,D$ are {\em equivalent}
if there exist a permutation $\sigma \in S_n$ and a vector $x \in
\ZZ_2^n$ such that $D=x+\sigma(C)$.

%% We describe the notation of the codes!  
A Hadamard matrix is {\em normalized} if 
all entries in the first row and the first column are $1$.
Let $H$ be a normalized Hadamard matrix of order $n$.
Throughout this paper,
we denote by $C(H)$ 
the binary $(n,2n)$ code consisting of the $2n$ row vectors of 
$(1,0)$-matrices 
$(H+J_{n})/2$ and $(-H+J_{n})/2$, 
where $J_n$ denotes the $n \times n$ all-one matrix.
The code $C(H)$ is often called a Hadamard code.
It is trivial that
$C(H)$ is a self-complementary code with distance
distribution $(A_0(C),A_{n/2}(C),A_n(C))=(1,2n-2,1)$.

%% Z4-codes
The {\em Lee weight} $\wt_L(x)$
of a vector $x=(x_1,x_2,\ldots,x_n)$ of $\ZZ_4^n$ is
$n_1(x)+2n_2(x)+n_3(x)$, 
where $n_{\alpha}(x)$ denotes 
the number of components $i$ with $x_i=\alpha$ $(\alpha=0,1,2,3)$.
The {\em Lee distance} $d_L(x,y)$ between
two vectors $x$ and $y$ of $\ZZ_{4}^n$ is $\wt_L(x-y)$.
The {\em minimum Lee distance} $d_L(\cC)$
of a $\ZZ_4$-code $\cC$ 
is the smallest Lee distance among all pairs of two distinct 
codewords of $\cC$.
The {\em Gray map} $\phi$ is defined as a map from
$\ZZ_4^n$ to $\ZZ_2^{2n}$
mapping $(x_1,x_2,\ldots,x_n)$ to
$(\phi(x_1),\phi(x_2),\ldots,\phi(x_n))$,
where $\phi(0)=(0,0)$, $\phi(1)=(0,1)$, $\phi(2)=(1,1)$ and
$\phi(3)=(1,0)$.
% The Gray image $\phi(\cC)$ of a $\ZZ_4$-code $\cC$ need not be
% linear even if $\cC$ is linear.
If $\cC$ is a $\ZZ_4$-code of length $n$
and minimum Lee distance $d_L(\cC)$, 
then the Gray image $\phi(\cC)$ is a
binary $(2n,|\cC|,d_L(\cC))$ code.
The {\em Lee distance distribution} of 
a $\ZZ_4$-code $\cC$ of length $n$ is defined as
$(A_0(\cC),A_1(\cC),\ldots,A_{2n}(\cC))$,
where
\[
A_i(\cC)=\frac{1}{|\cC|} |\{(x,x') \mid x,x' \in \cC, d_L(x,x')=i \}|
\quad (i=0,1,\ldots,2n).
\]
%%%%
% Two linear $\ZZ_4$-codes are {\em equivalent} if one can 
% be obtained from the other by permuting the coordinates and 
% (if necessary) changing the signs of certain coordinates.
Two linear $\ZZ_4$-codes $\cC,\cC'$ of length $n$ are {\em equivalent} if 
there exist $\sigma \in S_n$ and $j_1,j_2,\ldots,j_k \in \{1,2,\ldots,n\}$
such that 
$\cC=\tau_{j_1}\tau_{j_2}\cdots\tau_{j_k}\sigma(\cC')$.
%%%%
Let $G(1,m)$ denote a generator matrix of the first order 
binary Reed--Muller code $RM(1,m)$ of length $2^m$.
The first order {\em Reed--Muller $\ZZ_4$-code $ZRM(1,m)$} 
is defined as the linear $\ZZ_4$-code of length $2^m$, which is
generated by the rows of the matrix
%$\displaystyle{
%\left( \begin{array}{ccc}
%11 \cdots  \cdots  11 \\
%  2G(1,m) 
%\end{array}
%\right)
%}$,
$
\left( \begin{smallmatrix}
11 &\cdots &  11 \\
  &2G(1,m) &
\end{smallmatrix}
\right)
$,
where we regard $2G(1,m)$ as a $\ZZ_4$-matrix~\cite{Z4-HKCSS}.

%%%%%%%%%%%%%
\subsection{Association schemes}\label{sec:AS}

Let $X$ be a finite set and $\{R_0,R_1,\ldots,R_n\}$ be a set of non-empty subsets of $X\times X$.
Let $A_i$ denote the adjacency matrix of the digraph 
with vertex set $X$ and arc set $R_i$ for $i=0,1,\ldots, n$.
The pair $(X,\{R_i\}_{i=0}^n)$ is called a {\em symmetric association
scheme} of class $n$ if the following conditions hold:
\begin{itemize}
\item $A_0 =I_{|X|}$,
\item $\sum_{i=0}^nA_i=J_{|X|}$,
\item $A_i^T=A_i$ for $i \in \{1,2,\ldots,n\}$,
%\item $A_iA_j$ is a linear combination of $A_0,A_1,\ldots,A_n$ for $0\leq i,j\leq n$.
\item $A_iA_j=\sum_{k=0}^n p_{i,j}^kA_k$,
where $p_{i,j}^k$ are nonnegative integers  ($i,j \in \{0,1,\ldots,n\}$).
\end{itemize}
The vector space $\mathcal{A}$ over $\mathbb{R}$ spanned by the matrices
$A_i$ forms an algebra.
Since $\mathcal{A}$ is commutative and semisimple, 
$\mathcal{A}$ has a unique basis of
primitive idempotents $E_0=\frac{1}{|X|}J_{|X|},E_1,\ldots,E_n$.
The algebra $\mathcal{A}$ is closed under the ordinary multiplication
and entry-wise multiplication denoted by $\circ$. 
We define 
% the intersection numbers $p_{i,j}^k$ and 
the Krein numbers 
$q_{i,j}^k$ for $i,j,k\in \{0,1,\ldots,n\}$ as 
%A_iA_j=\sum_{k=0}^n p_{i,j}^kA_k,\quad 
$E_i\circ E_j=\frac{1}{|X|}\sum_{k=0}^n q_{i,j}^kE_k$.
It is known that the Krein numbers are nonnegative real numbers 
(see~\cite[Lemma~2.4]{D}). 
%Since both $\{A_0,A_1,\ldots,A_n\}$ and $\{E_0,E_1,\ldots,E_n\}$ form 
%bases of $\mathcal{A}$, there exists a matrix $Q=(q_{ij})$ 
Since $\{A_0,A_1,\ldots,A_n\}$ forms 
a basis of $\mathcal{A}$, there exists a matrix $Q=(q_{ij})$ 
% satisfying the following:
with
% A_i=\sum_{j=0}^n p_{ji}E_j,\quad 
$E_i=\frac{1}{|X|}\sum_{j=0}^n q_{ji}A_j$.
% The matrices $P$ and $Q$ are called the first and second eigenmatrices 
% of $(X,\{R_i\}_{i=0}^n)$, respectively.
%The symmetric association scheme $(X,\{R_i\}_{i=0}^n)$ is said to be $P$-polynomial 
%if for each $i\in\{0,1,\ldots,n\}$, there exists a polynomial $v_i(x)$ with degree $i$ such that $P_{ji}=v_i(P_{j1})$ for all $j\in\{0,1,\ldots,n\}$.
%Dually, 
A symmetric association scheme $(X,\{R_i\}_{i=0}^n)$ is said to be 
{\em $Q$-polynomial} 
if for each $i\in\{0,1,\ldots,n\}$, there exists a polynomial $v_i(z)$ of degree $i$ such that $q_{ji}=v_i(q_{j1})$ for all $j\in\{0,1,\ldots,n\}$. 
%It is also known that a symmetric association scheme is a $Q$-polynomial if and only if the matrix $(q_{1,j}^k)_{j,k=0}^n$ is a tridiagonal matrix with nonzero superdiagonal and subdiagonals~\cite[p.193]{BI}. 
%For a $Q$-polynomial association scheme, set $a_i^*=q_{1,i}^i$, $b_i^*=q_{1,i+1}^i$, and $c_i^*=q_{1,i-1}^i$. 
We say that a $Q$-polynomial association scheme is {\em $Q$-bipartite} if $q_{i,j}^k=0$ for all $i,j,k\in\{1,2,\ldots,n\}$ such that $i+j+k$ is odd. 

%There exists a $|X| \times |X|$ matrix $S=(S_0\ S_1\ \cdots\ S_n)$ 
There exists a matrix $S=(S_0\ S_1\ \cdots\ S_n)$ 
whose rows and columns are indexed by $X$, 
satisfying that $SS^T=|X|I_{|X|}$ and $S$ diagonalizes the adjacency matrices, 
where $E_i=\frac{1}{|X|}S_iS_i^T$ for $i\in \{0,1,\ldots,n\}$~\cite[p.~11]{D}.
We then define the $i$-th {\em characteristic matrix} $G_i$ of a subset $C$ of $X$ as the submatrix of $S_i$ that lies in the rows indexed by $C$.

Suppose that 
$X=\mathbb{Z}_2^n$ and $R_i=\{(x,y)\mid x,y\in X, d(x,y)=i\}$ 
for $i=0,1,\ldots,n$.  
Then the pair $(X,\{R_i\}_{i=0}^n)$ is a symmetric association scheme,
which is called the {\em binary Hamming} association scheme. 
% For a suitable numbering of the primitive idempotents,
% it holds that $p_{i,j}^k=q_{i,j}^k$ for $i,j,k\in \{0,1,\ldots,n\}$.  
The binary Hamming association scheme is a $Q$-bipartite $Q$-polynomial 
association scheme with the polynomials $v_i(z)=K_i(n-2z)$, 
where $K_i(z)$ is the Krawtchouk polynomial of degree $i$ defined as 
$
K_i(z)=\sum_{j=0}^i(-1)^j \binom{z}{j}\binom{n-z}{i-j}.
$
By~\cite[Theorem~2.5]{D73},
the Krawtchouk polynomials satisfy the following recursion:
\begin{align}\label{eq:kp}
K_1(z)K_i(z)=(n-i+1)K_{i-1}(z)+(i+1)K_{i+1}(z),
\end{align}
for $i=0,1,\ldots,n-1$, where $K_{-1}(z)$ is defined as $0$.
% (see~\cite[Theorem 2.5]{D73},~\cite[p.~208, (2.3)]{BI}).

% From the viewpoint of algebraic combinatorics, 
% it seems to be important that the existence of mutually Hadamard matrices
% is equivalent to that of a certain association scheme of class 
% $4$~\cite{LMO}.

Recently,  by generalizing the result in~\cite{ABS},
it has been shown in~\cite{LMO}
that there exists a set of $f$ mutually unbiased Hadamard matrices 
of order $n$ if and only if there exists
a $Q$-polynomial association scheme of class $4$ which is both
$Q$-antipodal and $Q$-bipartite with $f$ $Q$-antipodal classes
(see~\cite{LMO} for undefined terms).
%%%%%%%%%%%%%%%%%%%%%%%%%%%%%%%%%%%%%%%%%%%%%%%%%
\section{Bounds for self-complementary codes}\label{sec:B}

For a code $C$ of length $n$, set $S(C)=\{i \in \{1,2,\ldots,n\}
\mid A_i(C) \ne 0\}$. 
The size of $S(C)$ is said to be the {\em degree} of $C$.  
The {\em annihilator polynomial} of $C$ is defined as follows:
\begin{align*}
\alpha_C(z)&=|C|\prod_{i\in S(C)}\left(1-\frac{z}{i}\right).
\end{align*}
By considering annihilator polynomials,
in this section, we give two upper bounds on the number of
codewords of binary self-complementary codes.
The two bounds are used to 
give upper bounds on the size of sets of mutually quasi-unbiased 
(resp.\ Type~II weakly unbiased) Hadamard matrices 
in Theorem~\ref{thm:bound} (resp.\ Theorem~\ref{thm:weakIIbound}).
We also consider the condition of equality of the first bound.

% First we prove the following lemma.
\begin{lemma}\label{lem:eq}
Let $S$ be a subset of $\{1,2,\ldots,n\}$ such that $|S|=s$,
$n\in S$ and if $a\in S \setminus\{n\}$ then $n-a\in S$.
%Set $\overline{\alpha}(z)=\prod_{i\in S\setminus \{n\}}(1-\frac{z}{i})$. 
Then 
$\overline{\alpha}(z)=\prod_{i\in S\setminus \{n\}}(1-\frac{z}{i})$ 
has the following expansion by the Krawtchouk polynomials:  
\begin{align}\label{eq:alpha}
\overline{\alpha}(z)
=\sum_{\substack{i=0,1,\ldots, s-1\\ i\equiv s-1\pmod{2}}}\alpha_iK_i(z),
\end{align}
where $\alpha_i \in \QQ$.
\end{lemma}
\begin{proof}
When $s$ is odd, we may write
 $S=\{a_1,a_2,\ldots,a_{(s-1)/2},n-a_1,n-a_2,
\ldots,n-a_{(s-1)/2},n\}$, where
 $0<a_1 < a_2 < \cdots < a_{(s-1)/2}<n/2$.
Then we have
\begin{align*}
\overline{\alpha}(z)
%&=\prod_{i\in S}\left(1-\frac{z}{i}\right)\\
&=\prod_{i=1}^{(s-1)/2}
\left(\left(1-\frac{z}{a_i}\right)\left(1-\frac{z}{n-a_i}\right)\right)\\
%&=\prod_{i=1}^{(s-1)/2}\frac{1}{a_i(n-a_i)}\prod_{i=1}^{(s-1)/2}\left(a_i-z\right)\left(n-a_i-z\right)\\
%&=\prod_{i=1}^{(s-1)/2}\frac{1}{a_i(n-a_i)}\prod_{i=1}^{(s-1)/2}\left(a_i-n/2-x/2\right)\left(-a_i+n/2-x/2\right)\\
%&=\prod_{i=1}^{(s-1)/2}\frac{1}{a_i(n-a_i)}\prod_{i=1}^{(s-1)/2}\left(-(a_i-n/2)^2+x^2/4\right)
&=\prod_{i=1}^{(s-1)/2}\frac{1}{a_i(n-a_i)}\prod_{i=1}^{(s-1)/2}\left(-\Big(a_i-\frac{n}{2}\Big)^2+\frac{x^2}{4}\right),
\end{align*}
where $x=n-2z$. 
Thus, $\overline{\alpha}(z)=\overline{\alpha}(n/2-x/2)$ is 
an even polynomial in variable $x$.

When $s$ is even, we may write $S=\{a_1,a_2,\ldots,a_{s/2-1},n/2,
n-a_1,n-a_2,\ldots,n-a_{s/2-1},n\}$, where $0<a_1 <a_2 < \cdots < a_{s/2-1}<n/2$. 
Similar to the case where $s$ is odd, we have 
\begin{align*}
\overline{\alpha}(z)
%&=\prod_{i\in S}\left(1-\frac{z}{i}\right)\\
%&=\prod_{i=1}^{(s-1)/2}\left(1-\frac{z}{a_i}\right)\left(1-\frac{z}{n-a_i}\right)\\
%&=\prod_{i=1}^{(s-1)/2}\frac{1}{a_i(n-a_i)}\prod_{i=1}^{(s-1)/2}\left(a_i-z\right)\left(n-a_i-z\right)\\
%&=\prod_{i=1}^{(s-1)/2}\frac{1}{a_i(n-a_i)}\prod_{i=1}^{(s-1)/2}\left(a_i-n/2-x/2\right)\left(-a_i+n/2-x/2\right)\\
%&=\prod_{i=1}^{s/2-1}\frac{1}{a_i(n-a_i)}\frac{1}{n/2}\prod_{i=1}^{s/2-1}\left(-(a_i-n/2)^2+x^2/4\right)(x/2).
&=\left(\prod_{i=1}^{s/2-1}\frac{1}{a_i(n-a_i)}\right)\frac{1}{\frac{n}{2}}
\left(\prod_{i=1}^{s/2-1}
\left(-\Big(a_i-\frac{n}{2}\Big)^2+\frac{x^2}{4}\right)\right)\frac{x}{2},
\end{align*}
where $x=n-2z$. 
Thus, $\overline{\alpha}(z)=\overline{\alpha}(n/2-x/2)$ is 
an odd polynomial in variable $x$.

It can be shown that $K_i(z)=K_i(n/2-x/2)$ is an even (resp.\ odd) 
polynomial of degree $i$ in variable $x$ if $i$ is an even (resp.\ odd),  
from which  the expansion of $\overline{\alpha}(z)$ by the 
Krawtchouk polynomials has the desired form~\eqref{eq:alpha}.
\end{proof}

\begin{theorem}\label{thm:SCbound}
Let $C$ be a self-complementary code of length $n$ and degree $s$.
Then 
\begin{align*}
|C|\leq 2\sum_{\substack{i=0,1,\ldots, s-1\\ i\equiv s-1\pmod{2}}}\binom{n}{i}.
\end{align*}
\end{theorem}
\begin{proof}
We consider a subcode $C'$ of $C$ such that $C=C'\cup (C'+\boldsymbol{1})$,
$C'\cap (C'+\boldsymbol{1})=\emptyset$.
% and $\boldsymbol{1}\not\in C'$.
Then $|C|=2|C'|$ and $C'$ satisfies that $S(C')\subset S(C)\setminus\{n\}$.
% The annihilator polynomial of $C'$ is defined as follows:
% \begin{align*}
% \alpha'(z)&=|C'|\prod_{i\in S(C')}\left(1-\frac{z}{i}\right).
% \end{align*}
Since $C$ is self-complementary, 
the annihilator polynomial $\alpha_{C'}(z)$ of $C'$ 
has the following expansion by Lemma~\ref{lem:eq}: 
\begin{align*}
\alpha_{C'}(z)
=\sum_{\substack{i=0,1,\ldots, s-1\\ i\equiv s-1\pmod{2}}}\alpha_iK_i(z).
\end{align*}

Set $K=\begin{pmatrix}G_0&G_2&\cdots &G_{s-1}  \end{pmatrix}$ 
%$\left(\begin{smallmatrix}G_0&G_2&\cdots &G_{s-1}  \end{smallmatrix}\right)$
if
$s$ is odd and  
$K=\begin{pmatrix}G_1&G_3&\cdots &G_{s-1}\end{pmatrix}$
%$\left(\begin{smallmatrix}G_1&G_3&\cdots &G_{s-1}  \end{smallmatrix}\right)$
  if $s$ is even, 
where $G_i$ is the $i$-th characteristic matrix of $C$,
and set
\begin{align*}
\Gamma=\bigoplus_{\substack{i=0,1,\ldots, s-1\\ i\equiv s-1\pmod{2}}}\alpha_i I_{K_i(0)}.
\end{align*} 
By~\cite[Theorem~3.13]{D}, we have 
$K\Gamma K^T=|C'|I_{|C'|}$. 
Taking the rank of the above equation yields that 
\begin{align*}
|C'|=\rank(K\Gamma K^T)\leq\rank(K)\leq 
\min\left\{|C'|,\sum_{\substack{i=0,1,\ldots, s-1\\ i\equiv
s-1\pmod{2}}}K_i(0) \right\}, 
\end{align*}
as desired.
\end{proof}

\begin{remark}
The above upper bound depends on the degrees.
An upper bound, which depends on the minimum distances, 
can be found in~\cite{L93}.
\end{remark}
% \begin{remark}
% The upper bound for a code $C$ of length $n$ and degree $s$ was 
% obtained in \cite[Theorem 4.2]{D73}: $|C|\leq \sum_{i=0}^s
% \binom{n}{i}$. 
% The bound in Theorem~\ref{thm:SCbound} improves this bound, since 
% $ 2\sum_{\substack{i=0,1,\ldots, s-1\\ i\equiv s-1\pmod{2}}}
% \binom{n}{i}-\sum_{i=0}^s \binom{n}{i}=\binom{n-1}{s}$.   
% \end{remark}

%If equality holds,
If $|C| = 2\sum_{\substack{i=0,1,\ldots, s-1\\ i\equiv s-1\pmod{2}}}\binom{n}{i}$,
then the matrix $K$ is a square matrix and invertible.  
% Thus, $\sum_{\substack{0\leq i\leq s\\ i\equiv s\pmod{2}}}\alpha_i K_i(z)$ is a scalar multiple of the identity matrix, which implies that  $\alpha_i$ are all equal. 
Thus, 
$\bigoplus_{\substack{i=0,1,\ldots, s-1\\ i\equiv s-1\pmod{2}}}\alpha_iI_{K_i(0)}$
is a scalar multiple of the identity matrix, which implies that 
$\alpha_i$ are all equal.

Let $C$ be a self-complementary code of length $n$ and degree $s$. 
By Lemma~\ref{lem:eq}, we may suppose that
the expansion of 
$\overline{\alpha}_C(z)=\prod_{i\in S(C)\setminus \{n\}}(1-\frac{z}{i})$
by the Krawtchouk 
polynomials is as follows:
\begin{align}\label{eq:alpha1}
\overline{\alpha}_C(z)
=\sum_{\substack{i=0,1,\ldots, s-1\\ i\equiv s-1\pmod{2}}}\alpha_iK_i(z).
\end{align}

\begin{theorem}\label{thm:SCLP}
Suppose that $\alpha_{\delta}=\alpha_0$ 
if $s$ is odd and $\alpha_{\delta}=\alpha_1$ 
if $s$ is even.
If $\alpha_i$ in~\eqref{eq:alpha1}
are all nonnegative and $\alpha_{\delta}$ is positive,
then
\[
|C|\leq \left\lfloor  \frac{2}{\alpha_{\delta}} \right\rfloor.    
\]
\end{theorem}
\begin{proof}
% By the definition of $\overline{\alpha}(z)$, 
The annihilator polynomial of $C$ is written as
$\alpha_C(z)=|C|\left(1-\frac{z}{n}\right)\overline{\alpha}_C(z)$.

By $K_1(z)=n-2z$ and~\eqref{eq:kp}, 
% the annihilator polynomial above is 
\begin{align*}
\alpha_C(z)&=\frac{|C|}{2}\left(1+\frac{1}{n}K_1(z)\right)
\overline{\alpha}_C(z) 
\\
&=\frac{|C|}{2}\sum_{\substack{i=0,1,\ldots, s-1\\ i\equiv s-1\pmod{2}}}
\left(
\alpha_i K_i(z)+ \frac{\alpha_i K_1(z)K_i(z)}{n}\right)\\
&=\frac{|C|}{2}\sum_{\substack{i=0,1,\ldots, s-1\\ i\equiv s-1\pmod{2}}}
\left(
\alpha_i K_i(z)+ \frac{\alpha_i ((n-i+1)K_{i-1}(z)+(i+1)K_{i+1}(z))}{n}\right),
\end{align*} 
where $K_{-1}(z)=0$.
Hence, the coefficient of $K_0(z)$ is $|C|\alpha_{\delta}/2$.
%By the assumption on $\alpha_i$, the linear programming 
%bound~\cite[Theorem~5.23 (ii)]{D} can be applied. 
%Since the coefficient of $K_0(z)$ is $|C|\alpha_{\delta}/2$,  
%the desired bound follows. 
By the assumption on $\alpha_i$, the linear programming 
bound~\cite[Theorem~5.23 (ii)]{D} shows that the coefficient of $K_0(z)$ is at most $1$. 
Therefore, the desired bound follows. 
%{\bf (Can you say one more????)}
\end{proof}

The above two bounds are referred to as the absolute bounds and
the linear programming bounds, respectively.
%%%%%
As a consequence, upper bounds on 
the maximum size 
among sets of mutually quasi-unbiased 
(resp.\ Type~II weakly unbiased)
Hadamard matrices
are given in Section~\ref{sec:F2} (resp.\ Section~\ref{sec:II}).

%%%%%%%%%%%%%%%%%%%%%%%%%%%%%%%%%%%%%%%%%%%%%%
\section{Quasi-unbiased Hadamard matrices}
\label{sec:basic}

In this section, we study quasi-unbiased Hadamard matrices.
All feasible parameter sets for quasi-unbiased 
Hadamard matrices are examined for orders up to $48$.

%%%%%%%%%%%%%%%%%%%%%%%%
\subsection{Basic properties and feasible parameters}

\begin{proposition}
If there exists a pair of quasi-unbiased Hadamard matrices 
of order $n$ with parameters $(l,a)$,
% then $l=(n/2\alpha)^2$ and $a=4\alpha^2$ for some
% positive integer $\alpha$ with $n \le 4 \alpha^2$.
then 
\begin{equation}\label{eq:par}
l=\Big(\frac{n}{2\alpha}\Big)^2,\quad a=4\alpha^2
\end{equation}
for some
positive integer $\alpha$ satisfying that
$n \equiv 0 \pmod{2\alpha}$ and 
$n \le 4 \alpha^2$.
\end{proposition}
\begin{proof}
Let $(H_1,H_2)$ be a pair of quasi-unbiased Hadamard matrices of order
$n$ with parameters $(l,a)$.
From the definition, $a$ must be a square, say, $a=b^2$, where $b$ is a
positive integer.
Let $h_1$ (resp.\ $h_2$) be a row of $H_1$ (resp.\ $H_2$).
Let $n_\pm(h_1,h_2)$ denote the number of 
components which are different in $h_1$ and $h_2$.
Then 
$2n_\pm(h_1,h_2)=n-b$ and $n+b$ if 
$h_1 \cdot h_2  =b$ and $-b$, respectively.
Since $n=2$ or $n \equiv 0 \pmod 4$,
$b$ is even.
Therefore, $a=4\alpha^2$ for some
positive integer $\alpha$, then
$l=(n/2\alpha)^2$.
Since $(1/\sqrt{a})H_1H_2^{T}$ is a weighing matrix of weight $l$,
it is trivial that $l \le n$.
Hence,  $n \le 4 \alpha^2$.
\end{proof}

From now on, we assume that $\alpha$ is a positive integer
for parameters $((n/2\alpha)^2,4\alpha^2)$.
We say that parameters $(l,a)$ satisfying~\eqref{eq:par}
are {\em feasible}.  % for each order $n$.
Since $(l,a)=(1,n^2)$ satisfies~\eqref{eq:par},
the parameters $(1,n^2)$ are feasible for each order $n$.

% Now we consider mutually
% quasi-unbiased Hadamard matrices of order $n$ 
% with trivial parameters  $(1,n^2)$.

\begin{proposition}
If there exists a Hadamard matrix of order $n$,
then there exists a set of $2^n n!$ mutually quasi-unbiased 
Hadamard matrices with parameters $(1,n^2)$, 
where $2^n n!$ is the maximum size among sets of such matrices.
\end{proposition}
\begin{proof}
Let $H,K$ be Hadamard matrices of order $n$.
It is easy to see that $(H,K)$ is a pair of quasi-unbiased 
Hadamard matrices with parameters  $(1,n^2)$
if and only if 
there exists a monomial $(1,-1,0)$-matrix $P$ such that $K=P H$.
In addition, 
for any monomial $(1,-1,0)$-matrices $P$ and $Q$,
$(PH,QH)$ is a pair of quasi-unbiased 
Hadamard matrices with parameters  $(1,n^2)$.
\end{proof}

For $n=4,8,\ldots,48$, we give in Table~\ref{Tab:Par}
feasible parameters $(l,a)$ and 
our present state of knowledge about 
the maximum size $f_{max}$
among sets of mutually quasi-unbiased Hadamard matrices of order $n$
with parameters $(l,a)$ except $(1,n^2)$.
In the third column of the table, ``-'' means that there exists no pair of
quasi-unbiased Hadamard matrices.
The last two columns provide references for
the lower and upper bounds on $f_{max}$.

\begin{table}[thb]
\caption{Quasi-unbiased Hadamard matrices $(n=4,8,\ldots,48)$}
\label{Tab:Par}
\begin{center}
%{\small
{\footnotesize
%{\scriptsize
%\begin{tabular}{c|c|l|ccccccc}
\begin{tabular}{c|c|c|l|l}
\noalign{\hrule height0.8pt}
%$n$ &$(l,a)$ & \shortstack{Upper and lower bounds\\ on the maximal value $f$} & \multicolumn{1}{c}{Reference}\\
$n$ &$(l,a)$ & $f_{max}$ & \multicolumn{2}{c}{Reference}\\
\hline
  4& $( 4,  4)$&  $2$  &\cite[Proposition~6]{CS73}& \cite[Table~1]{DGS2}\\
\hline
  8& $( 4,  16)$& $8$  &\cite[Theorem~4.4]{NS} &\cite[Theorem~4.1]{NS}\\
\hline
12  & $( 4,  36)$&  -  & &Corollary~\ref{cor:none}\\
    & $( 9,  16)$&  $2$ & Section~\ref{Subsec:Const}& Section~\ref{Subsec:Const}\\
\hline
 16 & $( 4,  64)$&  $8-35$ &\cite[Section~3]{HS} & Table~\ref{Tab:UB} \\
    & $(16,  16)$&  $8$ &\cite[Proposition~6]{CS73}& \cite[Table~1]{DGS2}\\
\hline
 20& $( 4, 100)$&  - &  &Corollary~\ref{cor:noneP}\\
\hline
 24  & $( 4, 144)$&  $2-85$ & Section~\ref{Subsec:Const}& Table~\ref{Tab:UB}\\
%   & $( 9 , 64)$& $16-85$ & Proposition~\ref{prop:tensor}, Section~\ref{Subsec:B24}& Proposition~\ref{prop:bound} (ii)\\
   & $( 9 , 64)$& $16-85$ & Section~\ref{Subsec:B}& Table~\ref{Tab:UB}\\
 & $(16,  36)$& - & &Proposition~\ref{prop:none}\\
\hline
 28& $( 4, 196)$&   - & & Corollary~\ref{cor:noneP}\\
\hline
32 & $( 4, 256)$& $8-155$ & Proposition~\ref{prop:tensor}  & Table~\ref{Tab:UB}\\
% & $(16,  64)$& $32$ & \cite{NS} &  Proposition~\ref{prop:bound} (ii)\\
 & $(16,  64)$& $32$  & \cite[Theorem~4.4]{NS} &\cite[Theorem~4.1]{NS}\\
\hline
36 & $( 4, 324)$ & - &  & Corollary~\ref{cor:none}\\
     & $( 9, 144)$ & $\le 199$ & &Table~\ref{Tab:UB}\\
     & $(36, 36 )$ & $2$ & \cite[Theorem~1.5]{HKO}& \cite[Lemma~3.3]{BSTW}\\
\hline
40 & $(4, 400 )$ & $\le 247$& & Table~\ref{Tab:UB}\\
     & $(16, 100)$ & - &  &Proposition~\ref{prop:none} \\
     & $(25, 64 )$ & $\le 28$ & & Table~\ref{Tab:UB}\\
\hline
44 &$(4, 484 )$ & - &  & Corollary~\ref{cor:noneP}\\
\hline
48 &$(4, 576 )$ & $2-361$ & Proposition~\ref{prop:tensor}& Table~\ref{Tab:UB}\\
    &$(9, 256 )$ & $16-361$ & Proposition~\ref{prop:tensor}& Table~\ref{Tab:UB}\\
    &$(16, 144 )$ & $\le 361$&  & Table~\ref{Tab:UB}\\
    &$(36, 64 )$ & $2-28$& Proposition~\ref{prop:tensor} & Table~\ref{Tab:UB}\\
\noalign{\hrule height0.8pt}
\end{tabular}
}
\end{center}
\end{table}

%%%%%%%%%%%%%%%%%%%%%%%%%%%%%%%%%%%%%%%%%%%%%%%%

\begin{proposition}\label{prop:none}
Suppose that there exists a pair of quasi-unbiased 
Hadamard matrices of order $n$ 
with parameters $((n/2\alpha)^2,4\alpha^2)$.
If $n \ne 4\alpha^2$, then $\alpha$ must be even.
\end{proposition}
\begin{proof}
Let $H$ be a Hadamard matrix of order $n$
and let $h_i$ be the $i$-th row of $H$.
Let $x$ be a vector of $\{1,-1\}^{n}$.
Then it is easy to see that
$h_i\cdot x  \equiv  h_j\cdot x \pmod 4$ for $i,j\in \{1,2,\ldots,n\}$.

Let $(H,K)$ be a pair of quasi-unbiased 
Hadamard matrices with parameters $((n/2\alpha)^2,4\alpha^2)$.
Since $(1/2\alpha) HK^T$ is a weighing matrix of weight 
$(n/2\alpha)^2$,
any row $x$ of $K$ satisfies
that
$h_i \cdot x \in \{0,\pm 2\alpha\}$ for $i=1,2,\ldots,n$.
Hence, 
if $(1/2\alpha)HK^T$ is not Hadamard,
equivalently $n \ne 4\alpha^2$, then
$\alpha$ must be even.
\end{proof}

\begin{corollary}\label{cor:none}
Suppose that $n \equiv 4 \pmod 8$ and $n \ge 12$.
Then there exists no pair of quasi-unbiased 
Hadamard matrices of order $n$ with parameters $(4,(n/2)^2)$.
\end{corollary}
\begin{proof}
Follows from Proposition~\ref{prop:none} by
considering the case $\alpha = n/4$.
\end{proof}

\begin{corollary}\label{cor:noneP}
Suppose that $n =4p$, where $p$ is an odd prime with $p \ge 5$.
Then there exists no pair of quasi-unbiased 
Hadamard matrices of order $n$ with parameters 
$(l,a) \ne (1,n^2)$.
\end{corollary}
\begin{proof}
From $p \ge 5$,
the only feasible parameters are $(4,4p^2)$ and $(1,16p^2)$.
By Proposition~\ref{prop:none},
there exists no pair of quasi-unbiased 
Hadamard matrices with parameters $(4,4p^2)$.
\end{proof}

%There are many ways to construct Hadamard matrices of larger order from those of smaller order (see~\cite{CK}). 
%In the following propositions, we show that the constructions given in~\cite[Proposition~1.27, Theorem~1.29]{CK} preserve the property of the quasi-unbiasedness.

\begin{proposition}\label{prop:tensor}
Let $\{H_1,H_2,\ldots,H_f\}$ (resp.\ $\{K_1,K_2,\ldots,K_f\}$)
be a set of $f$ mutually quasi-unbiased
Hadamard matrices of order $n$ (resp.\ $n'$) with parameters 
$(l,a)$ (resp.\ $(l',a')$).
Then $\{H_1\otimes K_1,H_2\otimes K_2,\ldots,H_f\otimes K_f\}$ is a set of $f$ mutually quasi-unbiased Hadamard matrices of order $nn'$
with parameters $(ll\rq{},aa\rq{})$.
\end{proposition}
\begin{proof}
It is sufficient to give a proof for the case $f=2$.
Using some $(1,-1,0)$-matrices $L$ and $L'$, 
the matrices
$H_1H_2^T$ and $K_1K_2^T$ are written as
$\sqrt{a} L$ and $\sqrt{a'} L'$, respectively.
Then $(H_1\otimes K_1)(H_2\otimes K_2)^T = \sqrt{aa'}L \otimes L'$.
The result follows.
\end{proof}
%\begin{proof}
%It is sufficient to give a proof for the case $f=2$.
%Let $H,K$ be a pair of 
%quasi-unbiased Hadamard matrices of order $n$
%with parameters $(l,a)$ and 
%let $H\rq{},K\rq{}$ be a pair of 
%quasi-unbiased Hadamard matrices of order $n'$ with parameters 
%$(l\rq{},a\rq{})$.
%It is easy to see that
%$H \otimes H\rq{}, K \otimes K\rq{}$ are a pair of 
%quasi-unbiased Hadamard matrices of order $nn'$
%with parameters $(ll\rq{},aa\rq{})$.
%\end{proof}

Let $(H,K)$ be a pair of quasi-unbiased Hadamard matrices of order $n$
with parameters $(l,a)$.
We denote the unique Hadamard matrix of order $2$ by $H_2$.
There exists a pair $(H_4,K_4)$
of unbiased Hadamard matrices of order $4$~\cite[Proposition 6]{CS73}.
By the above proposition,
$(H \otimes H_2, K \otimes H_2)$
is a pair of quasi-unbiased Hadamard matrices of order $2n$
with parameters $(l,4a)$, and
$(H \otimes H_4, K \otimes K_4)$
is a pair of quasi-unbiased Hadamard matrices of order $4n$
with parameters $(4l,4a)$.

If there exist
Hadamard matrices of orders $4m$ and $4n$, then there exists
a Hadamard matrix of order $8mn$~\cite[Statement 4.10]{A85}
(see also~\cite[Theorem~1]{C92} and \cite[Theorem~4.2.5]{IoninShrikhande}).
The explicit construction given in~\cite[Theorem~1]{C92} 
and~\cite[Theorem~4.2.5]{IoninShrikhande}
is as follows.
Let $H$ be a Hadamard matrix of order $4m$
and $K$ be a Hadamard matrix of order $4n$.
Let $H_{i}$ ($i=1,2$) be the $4m\times 2m$ matrices and
$K_i$ ($i=1,2$) be the $2n\times 4n$ matrices such that
$H=\begin{pmatrix}H_{1}&H_{2}\end{pmatrix}$, 
%$H=\left(\begin{smallmatrix}H_{1}&H_{2}\end{smallmatrix}\right)$,
$K=\begin{pmatrix}K_1 \\K_2\end{pmatrix}$.
%$K=\left(\begin{smallmatrix}K_1 \\K_2\end{smallmatrix}\right)$.
The following matrix
\begin{align*}
M(H,K)=\frac{1}{2}(H_{1}+H_{2})\otimes
K_1+\frac{1}{2}(H_{1}-H_{2})\otimes K_2
\end{align*}
is a Hadamard matrix of order $8mn$.
\begin{proposition}
Let $\{H_1,H_2,\ldots,H_f\}$ be a set of $f$ mutually
quasi-unbiased Hadamard matrices of order $4m$ with
parameters $(l,a)$ and $K$ be a Hadamard matrix of order $4n$.
Then
$\{M(H_1,K),M(H_2,K),\ldots,M(H_f,K)\}$ is a set of $f$ mutually
quasi-unbiased Hadamard matrices of order $8mn$ with
parameters $(l,4a n^2)$.
\end{proposition}
%\begin{proof}
%The tedious but straightforward proof is omitted.
%\end{proof}
\begin{proof}
Similar to that of the above proposition.
The tedious but straightforward proof is omitted.
\end{proof}

\subsection{Observations by straightforward construction}
\label{Subsec:Const}

From the definition of quasi-unbiased Hadamard matrices,
we immediately have the following observation.

\begin{proposition}\label{prop:reduction}
Let $P,Q,R$ be $n \times n$ $(1,-1,0)$-monomial matrices.
Then $(H,K)$ is a pair of
quasi-unbiased Hadamard matrices of order $n$ with
parameters $((n/2\alpha)^2,4\alpha^2)$
if and only if
$(PHQ,RKQ)$ is a pair of
quasi-unbiased Hadamard matrices of order $n$ with
parameters $((n/2\alpha)^2,4\alpha^2)$.
\end{proposition}

Suppose that $n \ge 4$.
For a given $(n,\alpha)$, 
when attempting to determine whether there exists a pair of
quasi-unbiased Hadamard matrices $H,K$ of order $n$ 
with parameters $((n/2\alpha)^2,4\alpha^2)$,
it is sufficient to consider only the inequivalent Hadamard matrices of
order $n$ as possible choices for $H$ 
and only the Hadamard matrices $\overline{K}$ of order $n$ 
as possible choices for $K$,
where the first three columns $c_1,c_2,c_3$ of 
$\overline{K}$ satisfy the following:
% \begin{equation}\label{eq:3col}
% \begin{array}{cccccc}
% &(++ \cdots ++ & ++ \cdots ++ & ++ \cdots ++ & ++ \cdots ++ ),\\
% &(++ \cdots ++ & ++ \cdots ++ & -- \cdots -- & -- \cdots -- ),\\
% &(
% \underbrace{++ \cdots ++}_{n/4 \text{ columns}} &
% \underbrace{-- \cdots --}_{n/4 \text{ columns}} &
% \underbrace{++ \cdots ++}_{n/4 \text{ columns}} &
% \underbrace{-- \cdots --}_{n/4 \text{ columns}}).
% \end{array}
% \end{equation}
\begin{equation}\label{eq:3col}
\begin{array}{lccccc}
c_1^T=(&+ \cdots + & + \cdots + & + \cdots + & + \cdots + &),\\
c_2^T=(&+ \cdots + & + \cdots + & - \cdots - & - \cdots - &),\\
c_3^T=(&
\underbrace{+ \cdots +}_{\frac{n}{4} \text{ rows}} &
\underbrace{- \cdots -}_{\frac{n}{4} \text{ rows}} &
\underbrace{+ \cdots +}_{\frac{n}{4} \text{ rows}} &
\underbrace{- \cdots -}_{\frac{n}{4} \text{ rows}}&).
\end{array}
\end{equation}
This substantially reduces the number of pairs of Hadamard matrices 
to be checked as possible pairs $(H,K)$.

Let $H_{12}$ be the Hadamard matrix of order $12$ having 
the following form:
\begin{equation}\label{eq:R}
\left(
\begin{array}{cccc}
+ & +  & \cdots & + \\
+ & {}     & {}     &{} \\
\vdots & {}     & R      &{} \\
+ & {}     &{}      &{} \\
\end{array}
\right),
\end{equation}
where $R$ is the $11 \times 11$
circulant matrix with first row:
\[
%(-1,1,-1,1,1,1,-1,-1,-1,1,-1).
(-+-+++---+-).
\]
We determined the maximum size $f$
among sets of mutually quasi-unbiased Hadamard matrices 
$H_{12,1},H_{12,2},\ldots,H_{12,f}$ of order $12$
with parameters $(9,16)$ as follows.
By Proposition~\ref{prop:reduction},
without loss of generality, we may assume that $H_{12,1}=H_{12}$.
Our exhaustive computer search
under the above condition~\eqref{eq:3col} on $K$
found $1485$ distinct Hadamard matrices $\overline{K_{12,i}}$
$(i=1,2,\ldots,1485)$ such that 
$(H_{12}, \overline{K_{12,i}})$ is a pair of
quasi-unbiased Hadamard matrices with the parameters.
In addition, our exhaustive computer search verified that
there exists no pair $(\overline{K_{12,i}},\overline{K_{12,j}})$ $(i \ne j)$
such that $\{H_{12},\overline{K_{12,i}},\overline{K_{12,j}}\}$ is
a set of $3$ mutually 
quasi-unbiased Hadamard matrices.
This means that $f=2$.
In Figure~\ref{Fig:12}, we list $\overline{K_{12}}$, which
is one of the $1485$ Hadamard matrices.

%%%%%%%%%%%%%%%%%  Fig  %%%%%%%%%%%%%%%%%
\begin{figure}[thbp]
\centering
{\scriptsize
\[
\overline{K_{12}}=
\left( \begin{array}{c}
+++--++++-+-\\
+++-+-------\\
++++-+-+-+-+\\
++--+++--+++\\
++-++--++++-\\
++-+--+-+--+\\
+-+++++-++--\\
+-+-----++++\\
+-+++-++--++\\
+--+-+----+-\\
+-----++-+--\\
+---++-++--+\\
\end{array}
\right)
\]
}
\caption{The matrix $\overline{K_{12}}$}
\label{Fig:12}
\end{figure}
%%%%%%%%%%%%%%%%%  Fig  %%%%%%%%%%%%%%%%%

% Up to equivalence, there exist $60$ Hadamard matrices of order
% $24$ (see~\cite{Hadamard}).
% We denote by $H_{24,1}$ \verb+had.24.1+ in~\cite{Hadamard}.
%R We denote by $H_{24,1}$ \verb+had.24.1+ in~\cite{Hadamard},
%R which is a Hadamard matrix of order $24$.
Our computer search under the condition~\eqref{eq:3col} on $K$
found a Hadamard matrix $\overline{K_{24,1}}$ of order $24$ such that
$(H_{24,1}, \overline{K_{24,1}})$ is a pair of
quasi-unbiased Hadamard matrices of order $24$ 
with parameters $(4,144)$,
where $H_{24,1}$ is \verb+had.24.1+ in~\cite{Hadamard}.
The matrix $\overline{K_{24,1}}$ is listed in Figure~\ref{Fig:24}.

%%%%%%%%%%%%%%%%%  Fig  %%%%%%%%%%%%%%%%%
\begin{figure}[thb]
\centering
{\scriptsize
\[
\overline{K_{24,1}}=
\left( \begin{array}{c}
+++++-+++-+--+-------+-+\\
++++-+----+--+++-++-+---\\
+++---++-+--++-+---++-+-\\
+++-+++++-+-+-+-+++++++-\\
+++----++--+---++++---++\\
++++++-+-+-++-+-+-------\\
++-+++++-+-+-+-+++++++-+\\
++-+--+-+--++-++--++-+--\\
++-++----++-+----+++--++\\
++--+---++-+-++--+--+++-\\
++---+--+++-+--++---++-+\\
++---++---++-++-+--+--++\\
+-+-+++-++++---+-+-+----\\
+-+++-+-+++++++++-+-+-++\\
+-+----+-+++--+---++++-+\\
+-++-++----++----+--++++\\
+-+-+-------++++++-+-+-+\\
+-++-+--++---+--+-++-++-\\
+----+++++--+++--++----+\\
+---++-+--++++-+--+--++-\\
+--+---++-++++--++-++---\\
+---+-+---------+-+-+---\\
+--+++-++-----++---++-++\\
+--+--++-++---++++---++-\\
\end{array}
\right)
\]
}
\caption{The matrix $\overline{K_{24,1}}$}
\label{Fig:24}
\end{figure}
%%%%%%%%%%%%%%%%%  Fig  %%%%%%%%%%%%%%%%%

%%%%%%%%%%%%%%%%%%%%%%%%%%%%%%%%%%%%%%%%%%%%%%
\section{A coding-theoretic approach to quasi-unbiased Hadamard matrices}
\label{sec:F2}

In this section, we give a coding-theoretic approach to 
 mutually quasi-unbiased Hadamard matrices.
As an application, upper bounds on the size of sets of mutually quasi-unbiased 
Hadamard matrices are derived.  % using Delsarte's theory~\cite{D}.
For modest lengths, we also give classifications of some binary 
self-complementary codes, in order to construct 
mutually quasi-unbiased Hadamard matrices.

%%%%%%%%%%%%
\subsection{Binary codes and quasi-unbiased Hadamard matrices}

% We give a coding-theoretic approach to the existence of
% a set of $f$ mutually quasi-unbiased 
% Hadamard matrices of order $n$ with parameters $((n/2\alpha)^2,4\alpha^2)$.

\begin{theorem}\label{thm:F2}
% Let $f$ be a positive integer with $f \ge 2$.
Let $\alpha$ be an integer with $0< \alpha < n/2$.
There exists a self-complementary
$(n,2fn)$ code $C$ satisfying the following conditions:
\begin{align}
\label{eq:B1}
&\{i \in \{0,1,\ldots,n\}\mid A_i(C) \ne 0\}=\{0,n/2\pm \alpha,n/2,n\}, \\
\label{eq:B2}
&C=C_1 \cup C_2 \cup \cdots \cup C_f,
\end{align}
where 
% $C_i$ has the same distance distribution as $RM(1,m)$.
each $C_i$ has distance distribution
$(A_0(C_i),A_{n/2}(C_i),A_n(C_i))=(1,2n-2,1)$
if and only if
there exists 
a set of $f$ mutually quasi-unbiased 
Hadamard matrices of order $n$
with parameters $((n/2\alpha)^2,4\alpha^2)$.
\end{theorem}
\begin{proof}
Suppose that there exists an $(n,2fn)$ code $C$
satisfying~\eqref{eq:B1} and~\eqref{eq:B2}.
Define $\psi$ as a map from $\ZZ_2^n$ to 
$\{1,-1\}^n$ $(\subset \ZZ^n)$ by
%  $\psi((x_i)_{i=1}^n)=(x'_i)_{i=1}^n$, 
 $\psi((x_1,x_2,\ldots,x_n))=(x'_1,x'_2,\ldots,x'_n)$, 
where $x'_i=-1$ if $x_i=1$ and  $x'_i=1$ if $x_i=0$.
% It follows from \eqref{eq:B1} that $C_i+\mathbf{1}=C_i$ 
It follows from the distance distribution of $C_i$ that $C_i+\mathbf{1}=C_i$ 
for $i=1,2,\ldots,f$.
Thus, 
%% Since $C_i+\mathbf{1}=C_i$, 
%Since $\{i \mid A_i \ne 0\}=\{0,n/2,n\}$ for $u_i+RM(1,m)$,
$\psi(C_i)$ is antipodal, that is, $-\psi(C_i)=\psi(C_i)$ for  $i=1,2,\ldots,f$.
%$\psi(u_i+RM(1,m))$ is antipodal, that is, $-\psi(u_i+RM(1,m))=\psi(u_i+RM(1,m))$. 
Hence, there exists a subset $X_i$ of $\psi(C_i)$ 
%Hence, there exists a subset $X_i$ of $\psi(u_i+RM(1,m))$
such that $X_i\cup(-X_i)=\psi(C_i)$ and  $X_i\cap(-X_i)=\emptyset$.
%such that $X_i\cup(-X_i)=\psi(u_i+RM(1,m))$ and  $X_i\cap(-X_i)=\emptyset$.
Note that
$\psi(x)\cdot \psi(y) = n-2d(x,y)$ for $x,y\in \ZZ_2^n$.
% Since $C_i$ has the same distance distribution as $RM(1,m)$,
The distance distribution of $C_i$ implies that
$d(x,y) \in \{0,n/2,n\}$ for $x,y \in C_i$.
Thus, $\psi(x)\cdot \psi(y) \in \{-n,0,n\}$ for
$x,y \in C_i$.
This means that
any two different vectors of $X_i$ are orthogonal 
for  $i=1,2,\ldots,f$.
Hence, one may define
a Hadamard matrix $H_i$ of order $n$ whose rows
are the vectors of $X_i$ for $i=1,2,\ldots,f$.

Let $v_i$ be a vector of $X_i$ for  $i=1,2,\ldots,f$.
The assumption of~\eqref{eq:B1}
implies that $d(\psi^{-1}(v_i),\psi^{-1}(v_j))=n/2,n/2\pm \alpha$ $(i \ne j)$,
namely,
$v_i \cdot v_j$ $(i \ne j)$ is  $0,\mp 2\alpha$ respectively, where $\alpha$ is the integer given in~\eqref{eq:B1}.
This shows that for any distinct $i,j\in \{1,2,\ldots, f\}$,
 $(1/2\alpha)H_i H_j^T$ is a $(1,-1,0)$-matrix, and thus it is a
 weighing matrix of weight $(n/2\alpha)^2$. 
Therefore, $\{H_1,H_2,\ldots,H_{f}\}$ is a set of $f$
mutually quasi-unbiased Hadamard matrices of order $n$
with parameters $((n/2\alpha)^2,4\alpha^2)$.

%%%%%%%%%%%%%%%%%%%%%%%
% Reversing the above argument shows the converse assertion.
The converse assertion follows by reversing the above argument.
% The converse assertion is obvious by reversing the above argument.
\end{proof}

\begin{remark}
The ``only if\rq{}\rq{} part in 
the above proposition was proved in~\cite{NS} 
% for linear codes $C$ satisfying \eqref{eq:B1} and \eqref{eq:B2}.
for a specific case, namely,
$C$ is a linear code of length $n=2^m$ satisfying~\eqref{eq:B1} and 
containing $RM(1,m)$ as a subcode.
\end{remark}

Now, as the case $s=4$ of Theorems~\ref{thm:SCbound} and~\ref{thm:SCLP},
we have two upper bounds
on the number of the codewords of 
self-complementary codes satisfying~\eqref{eq:B1}.
% We also examine the case of equality.

\begin{lemma}\label{lem:bound}
Let $C$ be a self-complementary code of length $n$ 
satisfying~\eqref{eq:B1}.
Then
\begin{enumerate}[\rm (i)]
\item $|C|\leq \frac{n(n^2-3n+8)}{3}$. If equality holds, then $4\alpha^2=3n-8$. 
\item If $3n-4\alpha^2 -2 >0$, then 
$|C|\leq \lfloor \frac{2n(n^2-4\alpha^2)}{3n-4\alpha^2-2} \rfloor$. 
If $|C|= \frac{2n(n^2-4\alpha^2)}{3n-4\alpha^2-2}$, 
then a pair $(C,\{R_i\}_{i=0}^4)$ is a
$Q$-polynomial
association scheme, where $R_i=\{(x,y)\mid x,y\in C, d(x,y)=\beta_i\}$
and $\{i \in \{0,1,\ldots,n\}\mid A_i(C)\neq0\}=\{\beta_0,\beta_1,\ldots,\beta_4\}$ with $0=\beta_0<\beta_1<\cdots<\beta_4$.
\end{enumerate}
\end{lemma}
\begin{proof}
(i)
% Since  $C$ is a self-complementary code of length $n$ and degree $4$,
% the upper bound follows from Theorem~\ref{thm:SCbound} directly.
The upper bound is the case $s=4$ of Theorem~\ref{thm:SCbound}.

Suppose that equality holds.
From the observation after Theorem~\ref{thm:SCbound},
$\alpha_{C'}(z)= \beta(K_1(z)+K_3(z))$ for some $\beta$.
Since $n/2\pm \alpha$ are roots of $K_1(z)+K_3(z)$, we have
$4\alpha^2=3n-8$.

(ii)
Expanding by the Krawtchouk polynomials, we have 
\begin{align*}
\overline{\alpha}_C(z)=&\Big(1-\frac{2z}{2\alpha+n}\Big)
\Big(1-\frac{2z}{n}\Big)\Big(1-\frac{2z}{-2\alpha+n}\Big)\\
=&\frac{3n-4\alpha^2-2}{n(n^2-4\alpha^2)}K_1(z)+\frac{6}{n(n^2-4\alpha^2)}K_3(z).
% &=\alpha_1K_1(z)+\alpha_3K_3(z)\quad \text{(say)}. 
\end{align*}
% By the assumption on $\alpha$ and $n$, $\alpha_1$ is positive and $\alpha_3$ is nonnegative. 
% Thus, Theorem~\ref{thm:SCLP} implies the desired bound.
By the assumption on $\alpha$ and $n$, 
both $\frac{3n-4\alpha^2-2}{n(n^2-4\alpha^2)}$ and 
$\frac{6}{n(n^2-4\alpha^2)}$ are positive.
Thus, Theorem~\ref{thm:SCLP} implies the desired bound.

Suppose that $|C|=\frac{2n(n^2-4\alpha^2)}{3n-2-4\alpha^2}$.
By following the same line as in the proof 
of~\cite[Theorems~1.1, 1.2~(5)]{BB}, we may prove that
$(C,\{R_i\}_{i=0}^4)$ is a $Q$-polynomial association scheme.
A detailed proof is given in~Appendix~A.
\end{proof}

By Theorem~\ref{thm:F2}, 
we immediately have the following two upper bounds on 
the maximum size among sets of mutually quasi-unbiased Hadamard matrices,
one of which
depends only on $n$, and
the other depends on $n,\alpha$.  
% The bounds are referred to as the absolute bounds and
% the linear programming bounds, respectively.
This is one of the main results of this paper.

\begin{theorem}\label{thm:bound}
Suppose that there exists a set of $f$ mutually
quasi-unbiased Hadamard matrices of order $n$
with parameters $((n/2\alpha)^2,4\alpha^2)$.
Then
\begin{enumerate}[\rm (i)]
\item $f \leq \lfloor \frac{n^2-3n+8}{6}\rfloor $. 
If $f =\frac{n^2-3n+8}{6}$, then $4\alpha^2 = 3n-8$.
\item If $3n-4\alpha^2 -2 >0$, then 
$f \leq \lfloor\frac{n^2-4\alpha^2}{3n-4\alpha^2-2}\rfloor$. 
\end{enumerate}
\end{theorem}

\begin{remark}
It is known that 
$f \le n/2$ if $n=4\alpha^2$ and $\alpha$ is even~\cite[Table~1]{DGS2},
$f \le 2$ if $n=4\alpha^2$ and $\alpha$ is odd~\cite[Lemma~3.3]{BSTW},
and
$f \le n$ if $2n=4\alpha^2$~\cite[Theorem~4.1]{NS}.
For the first and third cases,
the bounds are the same as (ii).
\end{remark}

% \begin{remark}
% We may apply the absolute and linear programming bounds for spherical codes~\cite{DGS}, 
% but Corollary~\ref{cor:bound} provides stronger results.   
% \end{remark}

% \begin{question}
% Improve the upper bounds given in Corollary~\ref{cor:bound}.
% \end{question}

\begin{table}[thb]
\caption{Absolute and linear programming bounds in Theorem~\ref{thm:bound}}
\label{Tab:UB}
\begin{center}
%{\small
{\footnotesize
%{\scriptsize
%\begin{tabular}{c|c|l|ccccccc}
\begin{tabular}{c|c|c|c}
\noalign{\hrule height0.8pt}
$n$ &$(l,a)$ & Absolute bound  & Linear programming bound \\
\hline
  4& $( 4,  4)$&$2$& $2$\\
\hline
  8& $( 4,  16)$&$8$& $8$\\
\hline
%12  & $( 4,  36)$&$58/3\approx19.33$& $*$\\
12    & $( 9,  16)$&$\lfloor58/3\rfloor=19$ &$\lfloor64/9\rfloor=7$\\
\hline
 16 & $( 4,  64)$&$35$&$*$ \\
    & $(16,  16)$&$36$&$8$ \\
\hline
% 20& $( 4, 100)$& $58$ & $*$\\
%\hline
 24  & $( 4, 144)$ &$\lfloor256/3\rfloor=85$& $*$\\
   & $( 9 , 64)$& $85$ &$\lfloor256/3\rfloor=85$  \\
% & $(16,  36)$&  &$270/17\approx15.88$ \\
\hline
% 28& $( 4, 196)$&$118$& $*$\\
%\hline
32 & $( 4, 256)$ & $155$ &$*$\\
 & $(16,  64)$   & $156$ &$32$\\
\hline
%36 & $( 4, 324)$  &$598/3\approx199.33$& $*$\\
36     & $( 9, 144)$ & $\lfloor598/3\rfloor=199$&$*$\\
     & $(36, 36 )$ & $199$ &$18$ \\
\hline
40 & $(4, 400 )$ &$247$& $*$ \\
%     & $(16, 100)$ & &$250/3\approx83.33$ \\  
     & $(25, 64 )$ & $248$ &$\lfloor256/9\rfloor=28$ \\
\hline
%44 &$(4, 484 )$ & $302$ & $*$ \\
%\hline
48 &$(4, 576 )$ & $\lfloor1084/3\rfloor=361$ & $*$ \\
    &$(9, 256 )$ & $361$& $*$\\
    &$(16, 144 )$& $361$& $*$\\
    &$(36, 64 )$ & $361$&$\lfloor1120/39\rfloor=28$ \\
\noalign{\hrule height0.8pt}
\end{tabular}
}
\end{center}
%* means that the assumption of  Proposition~\ref{prop:bound} is not satisfied. 
%Absolute bounds depend only on $n$.
\end{table}

For the feasible parameters given in Table~\ref{Tab:Par}, 
we list in Table~\ref{Tab:UB}
the maximum possible sizes among sets of mutually quasi-unbiased 
Hadamard  matrices, which are obtained by the
two upper bounds.
We do no list the maximum possible sizes when
there exists no pair of quasi-unbiased Hadamard matrices.
In the table, ``$*$'' means that the assumption of 
Theorem~\ref{thm:bound} (ii) is not satisfied. 
%%%%
By Theorem~\ref{thm:bound} (i),
if $4\alpha^2 \ne 3n-8$, then $f < \frac{n^2-3n+8}{6}$. 
Suppose that $n=4\alpha$.
Then $4\alpha^2 = 3n-8$ if and only if $\alpha=1,2$.
As an example, for the cases 
$(n,l,a)=(16,4,64)$, $(32,4,256)$, $(40,4,400)$ in 
Table~\ref{Tab:UB}, 
the upper bound can decrease
from that of Theorem~\ref{thm:bound} (i) by $1$.

The following proposition was proved in~\cite{HS} %for linear codes $C$.
for a specific case, namely,
$C$ is a linear code of length $n=2^m$ satisfying~\eqref{eq:B1} and 
containing $RM(1,m)$ as a subcode.
Although the proof can be easily applied to all codes
satisfying~\eqref{eq:B1} and~\eqref{eq:B2},
we give a proof for the sake of completeness.

\begin{proposition} \label{prop:DisD}
Let $C$ be an $(n,2fn)$ code 
satisfying~\eqref{eq:B1} and~\eqref{eq:B2}.
Then the distance distribution of $C$ is given by:
\begin{multline*}
(A_0(C),A_{n/2-\alpha}(C),A_{n/2}(C),A_{n/2+\alpha}(C),A_n(C))\\
=(1,(f-1)l,2n-2+(f-1)(2n-2l),(f-1)l,1),
%&=(1,(M/2^{m}-1)l,2n-2+(M/2^{m}-1)(2n-2l),(M/2^{m}-1)l,1),
\end{multline*}
where $l=(n/2\alpha)^2$.
%where $n=2^m$, $M=f2^{m+1}$ and $l=(n/2a)^2$.
\end{proposition}
\begin{proof}
Let $H_i$ be the Hadamard matrix and let $C_i$ be the code
as in the proof of Theorem~\ref{thm:F2} for $i=1,2,\ldots,f$.
Let $x_i$ be a codeword of $C_i$ for $i=1,2,\ldots,f$.
%Since $C_i$ has the same distance distribution as $RM(1,m)$,
The distance distribution of $C_i$ implies that
there exist $2n-2$ codewords $y$ of $C_i$ such that
$d(x_i,y)=n/2$.
Now, suppose that $i,j \in \{1,2,\ldots,f\}$ with $i \ne j$.
Since $(1/2\alpha)H_i H_j^T$ is a weighing matrix of weight $l$,
the number of $0$'s in each row of $(1/2\alpha)H_i H_j^T$ 
is $n-l$.
That is, for a fixed row $r_i$ of $H_i$, there exist $n-l$ rows $r$ of $H_j$
such that $r_i \cdot r =0$.
Hence, 
since $C$ is self-complementary,
there exist $2(n-l)$ codewords $y \in C_j$ such that
$d(x_i,y)=n/2$.
Therefore, we have
\begin{align*}
A_{n/2}(C)=&(2fn(2n-2)+f(f-1)2n(2n-2l))/|C| \\
       =&(2n-2)+(f-1)(2n-2l).
\end{align*}
Since $C$ is self-complementary, we have the desired 
distance distribution.
\end{proof}

\begin{remark}
The minimum distance of $C$ implies the distance distribution
of $C$.
\end{remark}

%%%%%%%%%%%%
\subsection{Binary codes satisfying~\eqref{eq:B1} and~\eqref{eq:B2}}
\label{Subsec:B}

%%%%%%%%%%%%
%R \subsubsection{Method for classifications}
%R \label{Subsec:BM}

For some $(n,2n)$ codes $C_1$ ($n=8,12,16,20,24$),
we give a classification of  $(n,2fn)$ codes 
of the following form:
\begin{equation}\label{eq:form}
C_1 \cup (u_2 +C_1) \cup (u_3 +C_1) \cup \cdots \cup (u_{f} +C_1),
\end{equation}
satisfying~\eqref{eq:B1} and~\eqref{eq:B2}.
Although our method for the classifications
is straightforward,
we describe it for the sake of completeness.
%R 
% We describe how to complete the classification.
% This was done step by step as follows.
Let $C$ be an $(n,2(f-1)n)$ code 
of the form~\eqref{eq:form} satisfying~\eqref{eq:B1} and~\eqref{eq:B2}.
Every $(n,2fn)$ code $\overline{C}$ of the form~\eqref{eq:form} 
satisfying~\eqref{eq:B1} and~\eqref{eq:B2} and that $\overline{C} \supset C$, 
can be constructed as 
$C \cup (u_f +C_1)$, where $u_f \in \ZZ_2^n$.
By considering all vectors of $\ZZ_2^n \setminus C$, 
all $(n,2fn)$ codes $\overline{C}$ 
of the form~\eqref{eq:form} satisfying~\eqref{eq:B1},~\eqref{eq:B2}
and that $\overline{C} \supset C$,
can be obtained.
In addition, by considering all inequivalent $(n,2(f-1)n)$ codes $C$ 
of the form~\eqref{eq:form} satisfying~\eqref{eq:B1} and~\eqref{eq:B2},
all $(n,2fn)$ codes $\overline{C}$
of the form~\eqref{eq:form} satisfying~\eqref{eq:B1} and~\eqref{eq:B2},
which must be checked further for equivalences, can be obtained.
By checking equivalences among these codes, 
one can complete the classification of codes
of the form~\eqref{eq:form} satisfying~\eqref{eq:B1} and~\eqref{eq:B2}
for a fixed $C_1$.

Let $C,D$ be two binary $(n,M)$ codes containing the zero vector 
$\mathbf 0$.
Two codes $C,D$ are equivalent if and only 
if there exist a permutation $\sigma \in S_n$ and a vector $x \in
C$ such that $D=\{\sigma(c+x) \mid c \in C\}$.
For an $(n,M)$ code $C$,
we have an $M \times n$ $(1,0)$-matrix $m(C)$ with rows
composed of the codewords of $C$.
To test equivalence,
we checked whether there exists a vector $x \in C$
such that
the incidence structures with incidence matrices 
$m(D),m(\{\sigma(c+x) \mid c \in C\})$ are isomorphic.
%This was calculated by the {\sc Magma} function 
%{\tt IsIsomorphic}. 
The {\sc Magma} function {\tt IsIsomorphic} was used to
find out whether the incidence structures are isomorphic.

% To test equivalence of two $(n,M)$ codes $C,C'$,
% we checked whether their incidence structures with incidence matrices 
% $M(C),M(C')$ are isomorphic.
% This was calculated by the {\sc Magma} function 
% {\tt IsIsomorphic}. 

In this way, for some $(n,2n)$ codes $C_1$ ($n=8,12,16,20,24$),
by a computer calculation,
we completed the classification of codes of the form~\eqref{eq:form} 
satisfying~\eqref{eq:B1} and~\eqref{eq:B2}.
We list the number $N_2(C_1,2fn)$ of the inequivalent $(n,2fn)$ codes 
of the form~\eqref{eq:form} 
satisfying~\eqref{eq:B1} and~\eqref{eq:B2}.
%%%%
We mention that a classification of 
linear codes of length $2^m$ satisfying~\eqref{eq:B1} 
and containing $RM(1,m)$ as a subcode has been recently done in~\cite{HS}
under the equivalence of linear codes for $m=3,4,5$.

%%%%%%%%%%%%
%R \subsubsection{Binary codes of length 8
%R satisfying~\eqref{eq:B1} and~\eqref{eq:B2}}\label{Subsec:B8}

% For length $8$,
% we classify codes satisfying~\eqref{eq:B1} and~\eqref{eq:B2},
% where $C_1=RM(1,3)$ and
% $C_i$ are nontrivial cosets of $RM(1,3)$ in $\ZZ_2^8$ $(i=2,3,\ldots,f)$.
% We give the numbers $N_2(RM(1,3),16f)$.

\begin{proposition}
$N_2(RM(1,3), 16f)=1$ $(f=2,3,5,6,7,8)$,
$N_2(RM(1,3), 16f)=2$ $(f=4)$,
and
$N_2(RM(1,3),16f)=0$ $(f=9)$.
\end{proposition}

%%%%%%%%%%%%%%%%%%%%%%%%%%%
\begin{table}[thb]
\caption{Complete representatives of $\ZZ_2^8/RM(1,3)$}
\label{Tab:B8}
\begin{center}
%{\small
{\footnotesize
%{\scriptsize
%\begin{tabular}{c|c|l|ccccccc}
\begin{tabular}{c|l||c|l||c|l||c|l}
\noalign{\hrule height0.8pt}
$i$  & \multicolumn{1}{c||}{$\supp(x_i)$}&
$i$  & \multicolumn{1}{c||}{$\supp(x_i)$}&
$i$  & \multicolumn{1}{c||}{$\supp(x_i)$}&
$i$  & \multicolumn{1}{c}{$\supp(x_i)$}\\
\hline
% $1$&$(0, 0, 0, 0, 0, 0, 0, 0)$&$ 9$&$(0, 0, 0, 1, 0, 0, 0, 0)$\\
% $2$&$(0, 0, 0, 0, 0, 0, 0, 1)$&$10$&$(0, 0, 0, 1, 0, 0, 0, 1)$\\
% $3$&$(0, 0, 0, 0, 0, 0, 1, 0)$&$11$&$(0, 0, 0, 1, 0, 0, 1, 0)$\\
% $4$&$(0, 0, 0, 0, 0, 0, 1, 1)$&$12$&$(0, 0, 0, 1, 0, 0, 1, 1)$\\
% $5$&$(0, 0, 0, 0, 0, 1, 0, 0)$&$13$&$(0, 0, 0, 1, 0, 1, 0, 0)$\\
% $6$&$(0, 0, 0, 0, 0, 1, 0, 1)$&$14$&$(0, 0, 0, 1, 0, 1, 0, 1)$\\
% $7$&$(0, 0, 0, 0, 0, 1, 1, 0)$&$15$&$(0, 0, 0, 1, 0, 1, 1, 0)$\\
% $8$&$(0, 0, 0, 0, 0, 1, 1, 1)$&$16$&$(0, 0, 0, 1, 0, 1, 1, 1)$\\
 1& $\emptyset$ & 5& $\{6\}$      & 9& $\{4\}$      &13& $\{4, 6\}$\\
 2& $\{8\}$    & 6& $\{6, 8\}$   &10& $\{4, 8\}$   &14& $\{4, 6, 8\}$\\
 3& $\{7\}$    & 7& $\{6, 7\}$   &11& $\{4, 7\}$   &15& $\{4, 6, 7\}$\\
 4& $\{7, 8\}$ & 8& $\{6, 7, 8\}$&12& $\{4, 7, 8\}$&16& $\{4, 6, 7, 8\}$\\
\noalign{\hrule height0.8pt}
\end{tabular}
}
\end{center}
\end{table}

To list the result of the classification,
we fix the generator matrix of $RM(1,3)$ as %follows:
$
\left(
\begin{smallmatrix}
1 1 1 1 1 1 1 1\\
0 1 0 1 0 1 0 1\\
0 0 1 1 0 0 1 1\\
0 0 0 0 1 1 1 1
\end{smallmatrix}
\right),
$
and we list 
the $16$ vectors $x_i$, which give
the set of complete representatives of
$\ZZ_2^8/RM(1,3)$.
To save space, we list the supports $\supp(x_i)$ in 
Table~\ref{Tab:B8},
where $\supp(v)=\{i \mid v_i \ne 0\}$
for a vector $v=(v_1,v_2,\ldots,v_n)$.
The set was found by the {\sc Magma} function {\tt Transversal}. 
The unique $(8,32)$ code $B_{8,1,1}$,
the unique $(8,48)$ code $B_{8,2,1}$,
the two $(8,64)$ codes $B_{8,3,i}$ $(i=1,2)$,
the unique $(8,80)$ code $B_{8,4,1}$,
the unique $(8,96)$ code $B_{8,5,1}$, 
the unique $(8,112)$ code $B_{8,6,1}$, and
the unique $(8,128)$ code $B_{8,7,1}$
are constructed via
$\cup_{k \in X(B_{8,j,i})} (x_{k}+RM(1,3))$,
where $X(B_{8,j,i})$ are listed in Table~\ref{Tab:B82}.
By a computer calculation, we
verified that the minimum distances of the eight codes
are $2$.

%%%%%%%%%%%%%%%%%%%%%%%%%%%
\begin{table}[thb]
\caption{Codes of length $8$ 
satisfying~\eqref{eq:B1} and~\eqref{eq:B2}}
\label{Tab:B82}
\begin{center}
%{\small
{\footnotesize
%{\scriptsize
%\begin{tabular}{c|c|l|ccccccc}
\begin{tabular}{c|l||c|l}
\noalign{\hrule height0.8pt}
$C$ & \multicolumn{1}{c||}{$X(C)$ }& $C$ & \multicolumn{1}{c}{$X(C)$ }\\
\hline
$B_{8,1,1}$& $\{1,4\}$        &$B_{8,4,1}$& $\{1,4,6,7,10\}$ \\
$B_{8,2,1}$& $\{1,4,6\}$      &$B_{8,5,1}$& $\{1,4,6,7,10,11\}$\\
$B_{8,3,1}$& $\{1,4,6,7\}$    &$B_{8,6,1}$& $\{1,4,6,7,10,11,13\}$ \\
$B_{8,3,2}$& $\{1,4,6,10\}$   &$B_{8,7,1}$& $\{1,4,6,7,10,11,13,16\}$\\
\noalign{\hrule height0.8pt}
\end{tabular}
}
\end{center}
\end{table}
%%%%%%%%%%%%%%%%%%%%%%%%%%%%%%%%%%%%%%%%%%%%%%%%

% \begin{remark}
% By a computer calculation, we
% verified that only the codes $B_{8,1,1}$, $B_{8,3,1}$ and $B_{8,7,1}$
% are linear.
% \end{remark}

%%%%%%%%%%%%
%R \subsubsection{Binary codes of length 12
%R satisfying~\eqref{eq:B1} and~\eqref{eq:B2}}\label{Subsec:B12}

% We denote the unique normalized Hadamard matrix by $H_{12}$.
% Let $C_{12}$ be the nonlinear code consisting of the $24$ rows of 
% $(1,0)$-matrices 
% $(H_{12}+J_{12})/2$ and $(-H_{12}+J_{12})/2$, 
% where $J_n$ denotes the $n \times n$ all-one matrix.

% By the method given in Section~\ref{Subsec:BM}, we have the following:

\begin{proposition}
$N_2(C(H_{12}),24f)=0\ (f=2)$.
\end{proposition}

%%%%%%%%%%%%
%R \subsubsection{Binary codes of length 16
%R satisfying~\eqref{eq:B1} and~\eqref{eq:B2}}\label{Subsec:B16}

% By the method given in Section~\ref{Subsec:BM}, we have the following:

\begin{proposition}
$N_2(RM(1,4), 32f)=2$ $(f=2,3)$, 
$N_2(RM(1,4), 32f)=5$ $(f=4)$, 
$N_2(RM(1,4), 32f)=3$  $(f=5,6,7,8)$, and
$N_2(RM(1,4),32f)=0$ $(f=9)$.
\end{proposition}

%%%%%%%%%%%%%%%%%%%%%%%%%%%
\begin{table}[thb]
\caption{Codes of length $16$ 
satisfying~\eqref{eq:B1} and~\eqref{eq:B2}}
\label{Tab:B16}
\begin{center}
%{\small
{\footnotesize
%{\scriptsize
%\begin{tabular}{c|c|l|ccccccc}
\begin{tabular}{c|l|c||c|l|c}
\noalign{\hrule height0.8pt}
$C$ & \multicolumn{1}{c|}{$X(C)$}& $d_H(C)$&
$C$ & \multicolumn{1}{c|}{$X(C)$}& $d_H(C)$\\
\hline
$B_{16,1,1}$&$\{1,2\}$        &4&$B_{16,4,3}$&$\{1,5,6,8,9\}$          &6\\
$B_{16,1,2}$&$\{1,5\}$        &6&$B_{16,5,1}$&$\{1,2,3,4,12,13\}$      &4\\
$B_{16,2,1}$&$\{1,2,3\}$      &4&$B_{16,5,2}$&$\{1,2,3,4,17,18\}$      &4\\
$B_{16,2,2}$&$\{1,5,6\}$      &6&$B_{16,5,3}$&$\{1,5,6,8,9,10\}$       &6\\
$B_{16,3,1}$&$\{1,2,3,4\}$    &4&$B_{16,6,1}$&$\{1,2,3,4,12,13,14\}$   &4\\
$B_{16,3,2}$&$\{1,2,3,12\}$   &4&$B_{16,6,2}$&$\{1,2,3,4,17,18,19\}$   &4\\
$B_{16,3,3}$&$\{1,2,3,17\}$   &4&$B_{16,6,3}$&$\{1,5,6,8,9,10,11\}$    &6\\
$B_{16,3,4}$&$\{1,5,6,7\}$    &6&$B_{16,7,1}$&$\{1,2,3,4,12,13,14,15\}$&4\\
$B_{16,3,5}$&$\{1,5,6,8\}$    &6&$B_{16,7,2}$&$\{1,2,3,4,17,18,19,20\}$&4\\
$B_{16,4,1}$&$\{1,2,3,4,12\}$ &4&$B_{16,7,3}$&$\{1,5,6,8,9,10,11,16\}$ &6\\
$B_{16,4,2}$&$\{1,2,3,4,17\}$ &4& &&\\
\noalign{\hrule height0.8pt}
\end{tabular}
}
\end{center}
\end{table}
%%%%%%%%%%%%%%%%%%%%%%%%%%%%%%%%%%%%%%%%%%%%%%%

% For length $16$, in order to save space, we only list the maximal 
% codes satisfying~\eqref{eq:B1} and~\eqref{eq:B2}
% listed in the above proposition.
% We verified that 
% the $(16,M)$ codes given in the 
% above proposition are not maximal for $M=64, 96, 128, 160, 192, 224$.
% Hence, only the $17$ $(16,256)$ codes are maximal.
% only the three $(16,256)$ codes are maximal
% (with respect to the subset relation).
To list the result of the classification,
we fix the generator matrix of $RM(1,4)$ as follows:
\[
\left(
\begin{array}{c}
1 1 1 1 1 1 1 1 1 1 1 1 1 1 1 1\\
0 1 0 1 0 1 0 1 0 1 0 1 0 1 0 1\\
0 0 1 1 0 0 1 1 0 0 1 1 0 0 1 1\\
0 0 0 0 1 1 1 1 0 0 0 0 1 1 1 1\\
0 0 0 0 0 0 0 0 1 1 1 1 1 1 1 1
\end{array}
\right).
\]
The two $(16,32)$ codes $B_{16,1,i}$ $(i=1,2)$,
the two $(16,64)$ codes $B_{16,2,i}$ $(i=1,2)$,
the five $(16,96)$ codes $B_{16,3,i}$ $(i=1,2,\ldots,5)$,
the three $(16,128)$ codes $B_{16,4,i}$ $(i=1,2,3)$,
the three $(16,160)$ codes $B_{16,5,i}$ $(i=1,2,3)$,
the three $(16,192)$ codes $B_{16,6,i}$ $(i=1,2,3)$, and
the three $(16,224)$ codes $B_{16,7,i}$ $(i=1,2,3)$,
are constructed via
$\cup_{k \in X(B_{16,j,i})} (x_{k}+RM(1,4))$,
where $X(B_{16,j,i})$ are listed in Table~\ref{Tab:B16} and
$\supp(x_m)$ $(m=1,2,\ldots,20)$ are listed in Table~\ref{Tab:B16-2}.
By a computer calculation, we
determined the minimum distances $d_H(B_{16,j,i})$,
which are also listed in Table~\ref{Tab:B16}.

%%%%%%%%%%%%%%%%%%%%%%%%%%%
\begin{table}[thb]
\caption{Some representatives of $\ZZ_2^{16}/RM(1,4)$}
\label{Tab:B16-2}
\begin{center}
%{\small
%{\footnotesize
{\scriptsize
%\begin{tabular}{c|c|l|ccccccc}
\begin{tabular}{c|l||c|l||c|l}
\noalign{\hrule height0.8pt}
$i$  & \multicolumn{1}{c||}{$\supp(x_i)$}&
$i$  & \multicolumn{1}{c||}{$\supp(x_i)$}&
$i$  & \multicolumn{1}{c}{$\supp(x_i)$}\\
% $i$ & $x_i$ & $i$ & $x_i$ \\
\hline
%  1 &$(0,0,0,0,0,0,0,0,0,0,0,0,0,0,0,0)$&11&$(0,0,0,1,0,0,0,1,0,1,0,0,1,0,1,1)$\\
%  2 &$(0,0,0,0,0,1,1,0,0,1,1,0,0,0,0,0)$&12&$(0,0,0,0,0,0,0,0,0,0,0,0,1,1,1,1)$\\
%  3 &$(0,0,0,0,0,0,1,1,0,0,1,1,0,0,0,0)$&13&$(0,0,0,0,0,1,1,0,0,1,1,0,1,1,1,1)$\\
%  4 &$(0,0,0,0,0,1,0,1,0,1,0,1,0,0,0,0)$&14&$(0,0,0,0,0,0,1,1,0,0,1,1,1,1,1,1)$\\
%  5 &$(0,0,0,1,0,1,1,1,0,0,0,1,1,0,0,0)$&15&$(0,0,0,0,0,1,0,1,0,1,0,1,1,1,1,1)$\\
%  6 &$(0,0,0,0,0,1,0,1,0,1,1,0,1,1,0,0)$&16&$(0,0,0,0,0,1,1,0,0,0,1,1,0,1,0,1)$\\
%  7 &$(0,0,0,1,0,0,1,0,0,1,1,1,0,1,0,0)$&17&$(0,0,0,1,0,0,0,1,0,0,0,1,0,0,0,1)$\\
%  8 &$(0,0,0,0,0,0,1,1,0,1,0,1,0,1,1,0)$&18&$(0,0,0,1,0,1,1,1,0,1,1,1,0,0,0,1)$\\
%  9 &$(0,0,0,1,0,0,1,0,0,0,1,0,1,1,1,0)$&19&$(0,0,0,1,0,0,1,0,0,0,1,0,0,0,0,1)$\\
% 10 &$(0,0,0,1,0,1,0,0,0,1,1,1,0,0,1,0)$&20&$(0,0,0,1,0,1,0,0,0,1,0,0,0,0,0,1)$\\
1&$\emptyset$          & 8&$\{7,8,10,12,14,15\}$      &15&$\{6,8,10,12,13,14,15,16\}$\\
 2&$\{6,7,10,11\}$      & 9&$\{4,7,11,13,14,15\}$      &16&$\{6,7,11,12,14,16\}$\\
 3&$\{7,8,11,12\}$      &10&$\{4,6,10,11,12,15\}$      &17&$\{4,8,12,16\}$\\
 4&$\{6,8,10,12\}$      &11&$\{4,8,10,13,15,16\}$      &18&$\{4,6,7,8,10,11,12,16\}$\\
 5&$\{4,6,7,8,12,13\}$  &12&$\{13,14,15,16\}$          &19&$\{4,7,11,16\}$\\
 6&$\{6,8,10,11,13,14\}$&13&$\{6,7,10,11,13,14,15,16\}$&20&$\{4,6,10,16\}$\\
 7&$\{4,7,10,11,12,14\}$&14&$\{7,8,11,12,13,14,15,16\}$& & \\
\noalign{\hrule height0.8pt}
\end{tabular}
}
\end{center}
\end{table}
%%%%%%%%%%%%%%%%%%%%%%%%%%%%%%%%%%%%%%%%%%%%%%%

% \begin{remark}
% By a computer calculation, we
% verified that only the codes $B_{16,j,i}$
% are linear, where $(j,i)=(1,1),(1,2),(3,1),(3,4),(7,1),(7,2)$.
% \end{remark}

%R Up to equivalence, there exist five Hadamard matrices of order $16$.
We denote by $H_{16,1},H_{16,2},H_{16,3},H_{16,4}$ 
\verb+had.16.1+, \verb+had.16.2+, \verb+had.16.3+, \verb+had.16.4+
in~\cite{Hadamard}, respectively, which are
the remaining four normalized Hadamard matrices.
% Let $C(H_{16,i})$ be the nonlinear code consisting of the $32$ rows of $(1,0)$-matrices 
% $(H_{16,i}+J_{16})/2$ and $(-H_{16,i}+J_{16})/2$ $(i=1,2,3,4)$.
To save space, 
we only list the numbers $N_2(C(H_{16,i}), 32f)$ in Table~\ref{Tab:B16O}
for $i=1,2,3,4$.

% %%%%%%%%%%%%%%%%%%%%%%%%%%%
% \begin{table}[thbp]
% \caption{$N_2(C(H_{16,i}), 32f)$  $(i=1,2,3,4)$}
% \label{Tab:B16O}
% \begin{center}
% %{\small
% {\footnotesize
% %{\scriptsize
% %\begin{tabular}{c|c|l|ccccccc}
% \begin{tabular}{c|cccc}
% \noalign{\hrule height0.8pt}
% %$f$&$N_2(C(H_{16,1}), 32f)$ &$N_2(C(H_{16,2}), 32f)$ 
% %&$N_2(C(H_{16,3}), 32f)$ &$N_2(C(H_{16,4}), 32f)$ \\
% $f$ & $i=1$ & $i=2$  & $i=3$  & $i=4$ \\
% \hline
% 2 & 4 & 7 & 2 & 2\\
% 3 &13 &18 & 3 & 9\\
% 4 &47 &62 &10 &22\\
% 5 &24 &34 & 3 &16\\
% 6 & 9 &14 & 3 & 4\\
% 7 & 3 & 3 & 1 & 1\\
% 8 & 2 & 2 & 1 & 1\\
% $\ge9$   & 0  & 0  & 0 &  0 \\
% \noalign{\hrule height0.8pt}
% \end{tabular}
% }
% \end{center}
% \end{table}
% %%%%%%%%%%%%%%%%%%%%%%%%%%%%%

%%%%%%%%%%%%%%%%%%%%%%%%%%%
\begin{table}[thbp]
\caption{$N_2(C(H_{16,i}), 32f)$  $(i=1,2,3,4)$}
\label{Tab:B16O}
\begin{center}
%{\small
{\footnotesize
%{\scriptsize
%\begin{tabular}{c|c|l|ccccccc}
\begin{tabular}{c|cccccccc}
\noalign{\hrule height0.8pt}
%$f$&$N_2(C(H_{16,1}), 32f)$ &$N_2(C(H_{16,2}), 32f)$ 
%&$N_2(C(H_{16,3}), 32f)$ &$N_2(C(H_{16,4}), 32f)$ \\
$f$ & 2&3&4&5&6&7&8 & 9 \\
\hline
$N_2(C(H_{16,1}), 32f)$ & 4 &13 &47 &24 & 9 & 3 & 2 &0\\
$N_2(C(H_{16,2}), 32f)$ & 7 &18 &62 &34 &14 & 3 & 2 &0\\
$N_2(C(H_{16,3}), 32f)$ & 2 & 3 &10 & 3 & 3 & 1 & 1 &0\\
$N_2(C(H_{16,4}), 32f)$ & 2 & 9 &22 &16 & 4 & 1 & 1 &0\\
\noalign{\hrule height0.8pt}
\end{tabular}
}
\end{center}
\end{table}
%%%%%%%%%%%%%%%%%%%%%%%%%%%%%

%%%%%%%%%%%%
%R \subsubsection{Binary codes of length 24
%R satisfying~\eqref{eq:B1} and~\eqref{eq:B2}}\label{Subsec:B24}

Let $H_{24,2}$ be the Paley Hadamard matrix of order $24$ having the 
form~\eqref{eq:R},
% \[
% \left(
% \begin{array}{cccc}
% + & +  & \cdots & + \\
% + & {}     & {}     &{} \\
% \vdots & {}     & R      &{} \\
% + & {}     &{}      &{} \\
% \end{array}
% \right),
% \]
where $R$ is the $23 \times 23$
circulant matrix with first row: 
\[
%(-1,-1,-1,-1,-1,1,-1,1,-1,-1,1,1,-1,-1,1,1,-1,1,-1,1,1,1,1).
(-----+-+--++--++-+-++++).
\]
%Let $C(H_{24,2})$ be the nonlinear code consisting of the 
%$48$ rows of $(1,0)$-matrices 
%$(H_{24,2}+J_{24})/2$ and $(-H_{24,2}+J_{24})/2$.
Our computer search found a $(24,768,8)$ code
$
C_{24}=\cup_{i=1}^{16} (u_i + C(H_{24,2}))
$
satisfying~\eqref{eq:B1} and~\eqref{eq:B2}.
The vector $u_1$ is $\mathbf 0$ and $\supp(u_i)$ $(i=2,3,\ldots,16)$
are listed in Table~\ref{Tab:B24}.
This gives a set of $16$ mutually
quasi-unbiased  Hadamard matrices of order $24$ with parameters 
$(9,64)$ by Theorem~\ref{thm:F2}.

%%%%%%%%%%%%%%%%%%%%%%%%%%%
\begin{table}[thb]
\caption{Vectors $u_i$ for $C_{24}$}
\label{Tab:B24}
\begin{center}
%{\small
{\footnotesize
%{\scriptsize
%\begin{tabular}{c|c|l|ccccccc}
\begin{tabular}{c|l||c|l}
\noalign{\hrule height0.8pt}
$i$  & \multicolumn{1}{c||}{$\supp(u_i)$}&
$i$  & \multicolumn{1}{c}{$\supp(u_i)$}\\
% $i$ & $u_i$ \\
\hline
% $2$ & $(0,0,1,1,1,1,1,0,0,1,0,0,1,0,1,0,0,0,0,0,0,0,0,0)$\\
% $3$ & $(0,0,0,1,1,1,1,1,0,0,1,0,0,1,0,1,0,0,0,0,0,0,0,0)$\\
% $4$ & $(0,0,1,0,0,0,0,1,0,1,1,0,1,1,1,1,0,0,0,0,0,0,0,0)$\\
% $5$ & $(0,0,1,1,0,0,0,1,1,1,0,1,1,0,0,0,1,0,0,0,0,0,0,0)$\\
% $6$ & $(0,0,0,0,1,1,1,1,1,0,0,1,0,0,1,0,1,0,0,0,0,0,0,0)$\\
% $7$ & $(0,0,1,0,1,1,1,0,1,1,1,1,1,1,0,1,1,0,0,0,0,0,0,0)$\\
% $8$ & $(0,0,0,1,0,0,0,0,1,0,1,1,0,1,1,1,1,0,0,0,0,0,0,0)$\\
% $9$ & $(0,0,0,1,1,0,0,0,1,1,1,0,1,1,0,0,0,1,0,0,0,0,0,0)$\\
% $10$ & $(0,0,1,0,0,1,1,0,1,0,1,0,0,1,1,0,0,1,0,0,0,0,0,0)$\\
% $11$ & $(0,0,0,0,0,1,1,1,1,1,0,0,1,0,0,1,0,1,0,0,0,0,0,0)$\\
% $12$ & $(0,0,1,1,1,0,0,1,1,0,0,0,0,0,1,1,0,1,0,0,0,0,0,0)$\\
% $13$ & $(0,0,1,0,1,0,0,1,0,0,1,1,0,1,0,0,1,1,0,0,0,0,0,0)$\\
% $14$ & $(0,0,0,1,0,1,1,1,0,1,1,1,1,1,1,0,1,1,0,0,0,0,0,0)$\\
% $15$ & $(0,0,1,1,0,1,1,0,0,0,0,1,0,0,0,1,1,1,0,0,0,0,0,0)$\\
% $16$ & $(0,0,0,0,1,0,0,0,0,1,0,1,1,0,1,1,1,1,0,0,0,0,0,0)$\\
2&$\{3,4,5,6,7,10,13,15\}$            &10&$\{3,6,7,9,11,14,15,18\}$\\
3&$\{4,5,6,7,8,11,14,16\}$            &11&$\{6,7,8,9,10,13,16,18\}$\\
4&$\{3,8,10,11,13,14,15,16\}$         &12&$\{3,4,5,8,9,15,16,18\}$\\
5&$\{3,4,8,9,10,12,13,17\}$           &13&$\{3,5,8,11,12,14,17,18\}$\\
6&$\{5,6,7,8,9,12,15,17\}$            &14&$\{4,6,7,8,10,11,12,13,14,15,17,18\}$\\
7&$\{3,5,6,7,9,10,11,12,13,14,16,17\}$&15&$\{3,4,6,7,12,16,17,18\}$\\
8&$\{4,9,11,12,14,15,16,17\}$         &16&$\{5,10,12,13,15,16,17,18\}$\\
9&$\{4,5,9,10,11,13,14,18\}$          & & \\
\noalign{\hrule height0.8pt}
\end{tabular}
}
\end{center}
\end{table}
%%%%%%%%%%%%%%%%%%%%%%%%%%%%%

%%%%%%%%%%%%
\subsection{Binary codes satisfying~\eqref{eq:B1} and~\eqref{eq:B2}
from $\ZZ_4$-codes}
\label{Subsec:Z4}

In order to construct binary codes satisfying~\eqref{eq:B1} 
and~\eqref{eq:B2} systematically,
we consider $\ZZ_4$-codes $\cC$ 
of length $n$ with $|\cC|=4 f n$
satisfying the following conditions:
\begin{align}
\label{eq:Z4C1}
&\{i \in \{0,1,\ldots,n\}\mid A_i(\cC) \ne 0\}=\{0,n\pm \beta,n,2n\},\\
\label{eq:Z4C2}
&\cC=\cC_1 \cup \cC_2 \cup \cdots \cup \cC_f,
%% &\text{$C \supset ZRM(1,m)$},
\end{align}
where $\beta$ is an integer with $0 < \beta < n$, and
each $\cC_i$ has Lee distance distribution
$(A_0(\cC_i),A_{n}(\cC_i),A_{2n}(\cC_i))=(1,4n-2,1)$.

\begin{proposition}\label{prop:code2-2}
Let $\cC$ be a $\ZZ_4$-code of length $n$ %and type $4^{k_1}2^{k_2}$
satisfying~\eqref{eq:Z4C1} and~\eqref{eq:Z4C2}.
Then 
there exists a set of $f$ mutually quasi-unbiased 
Hadamard matrices of order $2n$ with parameters $(n^2/\beta^2,4\beta^2)$.
\end{proposition}
\begin{proof}
Since the Lee distance distribution of $\cC$ is the same as the
distance distribution of $\phi(\cC)$,
$\phi(\cC)$ satisfies~\eqref{eq:B1}.
In addition, $\phi(\cC_i)$ 
has the same distance distribution as $RM(1,m+1)$
for $i=1,2,\ldots,f$.
% Hence, $\phi(C)$ satisfies~\eqref{eq:B2}.
Since
$
\phi(\cC)= 
\phi(\cC_1) \cup \phi(\cC_2) \cup 
\cdots \cup
\phi(\cC_f),
$
$\phi(\cC)$ satisfies~\eqref{eq:B2}.
The result follows from Theorem~\ref{thm:F2}.
\end{proof}

%%%%%%%%%%%%%%%%%%%%%%%%%%%%%%%%
Now, we restrict our attention to linear $\ZZ_4$-codes $\cC$
of length $n=2^m$ satisfying the following conditions:
\begin{align}
\label{eq:Z4C1-2}
&\{(n_0(x)-n_2(x))^2\mid x\in \cC\}=\{0,\beta^2,n^2\}, \\
\label{eq:Z4C2-2}
%&\text{$\cC \supset ZRM(1,m)$},
&\text{$\cC$  contains $ZRM(1,m)$ as a subcode},
\end{align}
where $\beta$ is an integer with $0 < \beta < n$.
Let $x$ be a codeword of $\cC$.
Since $n_1(x)+2n_2(x)+n_3(x)=n-(n_0(x)-n_2(x))$,
$\{\wt_L(x) \mid x \in \cC\}=\{0,n\pm \beta,n,2n\}$.
This means that $\cC$ satisfies~\eqref{eq:Z4C1}.
Let $\{t_1,t_2,\ldots,t_f\}$ be a set of complete representatives 
of $\cC/ZRM(1,m)$.
%, that is,
%\[
%\cC= 
%(t_1+ZRM(1,m))
%\cup \cdots \cup
%(t_f+ZRM(1,m)).
%\]
It is trivial that $t_i + ZRM(1,m)$ has the same Lee distance distribution 
as $ZRM(1,m)$ for $i=1,2,\ldots,f$.
% d_L(t_i+x,t_i+y)
% =n_1((t_i+x)-(t_i+y))+2n_2((t_i+x)-(t_i+y))+n_3((t_i+x)-(t_i+y))
% =n_1(x-y)+2n_2(x-y)+n_3(x-y)
% =d_L(x,y)
%%%% OK for nonlinear!
Hence, $\cC$ satisfies~\eqref{eq:Z4C2}.
We note that the Kerdock $\ZZ_4$-code ${\mathcal K}(m)$ of length $2^m$
defined in~\cite{Z4-HKCSS}
satisfies~\eqref{eq:Z4C1-2} and~\eqref{eq:Z4C2-2} for $m \ge 2$.

\begin{remark}
The above method is a slight generalization of that given in~\cite{NS}.
% where following condition is considered:
% \[
% (n_0(x)-n_2(x),n_1(x)-n_3(x))\in
% \{(\pm n,0),(0,\pm n),(0,0),(\pm \beta,0),(0,\pm \beta)\}
% \]
% for any $x \in \cC$.
% %, which is a weaker condition than~\eqref{eq:Z4C1-2}.
% Note that~\eqref{eq:Z4C1-2} is equivalent to
% that
% \begin{multline*}
% (n_0(x)-n_2(x),n_1(x)-n_3(x))
% \\
% \in
% \{(\pm n,0),(0,\pm n),(0,0),(\pm \beta,0),(0,\pm \beta),
% (\pm \beta,\pm \beta)\}
% \end{multline*}
% for any $x \in \cC$.
\end{remark}

%\noindent
%{\bf === Nonlinear version (Finish) ===}
%\bigskip

In the rest of this section, 
we study classifications of
linear $\ZZ_4$-codes of length $2^m$
satisfying~\eqref{eq:Z4C1-2} and~\eqref{eq:Z4C2-2}.
Note that the conditions~\eqref{eq:Z4C1-2} and~\eqref{eq:Z4C2-2}
are invariant under equivalences of linear $\ZZ_4$-codes.

%%%%%
%R \subsubsection{Method for classifications}

Although our method for classifications of
linear $\ZZ_4$-codes of length $2^m$
satisfying~\eqref{eq:Z4C1-2} and~\eqref{eq:Z4C2-2}
is straightforward,
we describe it for the sake of completeness.
%R 
Let $\cC$ be a linear $\ZZ_4$-code with $|\cC|=2^{k}$
satisfying~\eqref{eq:Z4C1-2} and~\eqref{eq:Z4C2-2}.
Every linear $\ZZ_4$-code $\overline{\cC}$
such that $|\overline{\cC}|=2^{k+1}$ and $\overline{\cC} \supset \cC$
satisfying~\eqref{eq:Z4C1-2} and~\eqref{eq:Z4C2-2}, 
can be constructed as $\langle \cC, x \rangle$, where $x$ is some vector
of  a set $R_m$ of complete representatives of $\ZZ_4^{2^m}/\cC$.
By considering all vectors of  $R_m$,
all linear $\ZZ_4$-codes $\overline{\cC}$
which must be checked further for equivalence,  can be obtained.
In addition, by considering all inequivalent
linear $\ZZ_4$-codes $\cC$ with $|\cC|=2^{k}$
satisfying~\eqref{eq:Z4C1-2} and~\eqref{eq:Z4C2-2}, 
all linear $\ZZ_4$-codes $\overline{\cC}$ with 
$|\overline{\cC}|=2^{k+1}$ 
satisfying~\eqref{eq:Z4C1-2} and~\eqref{eq:Z4C2-2},
which must be checked further for equivalences, can be obtained.
By checking equivalences among these codes, 
one can complete the classification of linear 
$\ZZ_4$-codes $\overline{\cC}$ with $|\overline{\cC}|=2^{k+1}$
satisfying~\eqref{eq:Z4C1-2} and~\eqref{eq:Z4C2-2}.
% In this way, linear $\ZZ_4$-codes of length $2^m$ 
% satisfying~\eqref{eq:Z4C1-2} and~\eqref{eq:Z4C2-2} were classified
% step by step for $m=2,3,4$.

%%%%%%%%%%%%%%%%%%%%%%%%%
We now describe how to test equivalences of linear $\ZZ_4$-codes.
In this paper, we modify
the method for linear codes over a finite field,
which is given in~\cite{O02}. %\cite[Section 7.3.3]{KO},
For a linear $\ZZ_4$-code $\cC$ of length $n$, we define the digraph $\Gamma(\cC)$
with the following vertex set $V(\Gamma(\cC))$ and arc set $A(\Gamma(\cC))$:
\begin{align*}
V(\Gamma(\cC))=&\cC^\# \cup (\cP \times \ZZ_4^\#),
\\
A(\Gamma(\cC))=&\{(c,(j,c_j)) \mid c = (c_1,c_2,\ldots,c_n) \in \cC^\#,
c_j \ne 0, j \in \cP\}
\\
&\cup \{((j,x),(j,2)), ((j,2),(j,x)) \mid j \in \cP, x \in \{1,3\}\},
\end{align*}
where 
$\cC^\#=\cC \setminus\{\mathbf{0}\}$,
$\cP=\{1,2,\ldots,n\}$ and 
$\ZZ_4^\#=\ZZ_4\setminus\{0\}$.
By an argument similar to that in~\cite{O02}, 
the following characterization is obtained.
% We give a proof for the sake of completeness.

\begin{proposition}\label{prop:Z4}
Two linear $\ZZ_4$-codes $\cC,\cC'$ are equivalent if and only if
$\Gamma(\cC),\Gamma(\cC')$ are isomorphic.
\end{proposition}
\begin{proof} %% Z4-graph
% Let $\cC$ be a linear $\ZZ_4$-code of length $n$.
% Let $c$ be a codeword of $\cC$.
% For $\sigma \in S_n$, 
% %we write $\sigma(\cC) = \{(c_{\sigma(1)},\dots,c_{\sigma(n)}) \mid (c_1,\dots,c_n) \in \cC\}$.
% let $\sigma(c)$ denote the vector obtained from $c$,
% by the permutation $\sigma$ of the coordinates.
% For $j \in \cP$,
% %let $\tau_j(\cC)=\{(c_1,\dots,-c_j,\dots,c_n) \mid (c_1,\dots,c_n) \in \cC\}$.
% let $\tau_j(c)$ denote the vector obtained from $c$,
% by changing the sign of the $j$-th coordinate.
% In addition, set 
% $\sigma(\cC)=\{\sigma(c) \mid c \in \cC\}$ and
% $\tau_j(\cC)=\{\tau_j(c) \mid c \in \cC\}$.
% 
Suppose that two linear $\ZZ_4$-codes $\cC,\cC'$ of length $n$
are equivalent.
Then there exist $\sigma \in S_n$ and $j_1,j_2,\ldots,j_\ell \in \cP$
such that $\tau_{j_1}\tau_{j_2}\cdots\tau_{j_\ell}\sigma(\cC)=\cC'$
(see Section~\ref{sec:2.2} for the notations).
For $\sigma \in S_n$,
define a map $f_\sigma$ from $V(\Gamma(\cC))$ to $V(\Gamma(\sigma(\cC)))$ 
mapping $(j,x) \in \cP\times \ZZ_4^\#$
to $(\sigma(j),x)$, and
$c \in \cC^\#$ to $\sigma(c)$. 
%Let $f: V(\Gamma(\cC)) \longrightarrow V(\Gamma(\sigma(\cC)))$ 
%be a map with $f((j,x))=(\sigma(j),x)$ 
%and $f((c_1,\dots,c_n))=(c_{\sigma(1)},\dots,c_{\sigma(n)})$ 
%for $(j,x) \in \{1,\dots,n\}\times(\ZZ_4-\{0\})$ and $(c_1,\dots,c_n)
% \in \cC-\{\mathbf{0}\}$.
Then the map $f_\sigma$ is an isomorphism from $\Gamma(\cC)$ to $\Gamma(\sigma(C))$.
Now, for $j \in \cP$,
define a map $g_j$ from $V(\Gamma(\cC))$ to $V(\Gamma(\tau_j(\cC)))$
mapping 
$(j,x)$ to $(j,-x)$,
$(i,x) \in (\cP\setminus\{j\})\times \ZZ_4^\#$ to $(i,x)$,  
and $c \in C^\#$ to $\tau_j(c)$.
% Let $g_j: V(\Gamma(\cC)) \longrightarrow V(\Gamma(\tau_j(\cC)))$ 
% be a map with $g_j((i,x))=(i,-x)$,
% $g_j((i,x))=(i,x)$  and $g((c_1,\dots,c_n))=(c_1,\dots,-c_j,\dots,c_n)$ 
% for  $(i,x) \in (\{1,\dots,n\}-\{j\})\times(\ZZ_4-\{0\})$ and
% $(c_1,\dots,c_n)%  \in C$.
Then the map $g_j$ is an isomorphism 
from $\Gamma(\cC)$ to $\Gamma(\tau_j(\cC))$.
Hence, $g_{j_1}g_{j_2}\cdots g_{j_\ell} f_\sigma$ is an isomorphism
from $\Gamma(\cC)$ to $\Gamma(\cC')$.

Conversely, we suppose that two digraphs $\Gamma(\cC),\Gamma(\cC')$ are isomorphic.
Then there exists a bijection $f$
from $V(\Gamma(\cC))$ to $V(\Gamma(\cC'))$
such that $(x,y) \in A(\Gamma(\cC))$ if and only if 
$(f(x),f(y)) \in A(\Gamma(\cC'))$.
By the definition of $A(\Gamma(\cC))$, 
the subsets $\cC^\#$ and $\cP\times \ZZ_4^\#$
of $V(\Gamma(\cC))$ are characterized as follows:
\begin{align*}
\cC^\# &= \{v \in V(\Gamma(\cC)) \mid 
\text{the indegree of }v \text{ is equal to } 0\},\\
\cP\times \ZZ_4^\# &= V(\Gamma(\cC))\setminus \cC^\#.
\end{align*}
We have a similar characterization for $\Gamma(\cC')$.
Thus, we have
\begin{align*}
f(\cC^\#)=\cC'^\#, \quad
f(\cP\times \ZZ_4^\#)=\cP\times \ZZ_4^\#.
\end{align*}
We put $f((j,x))=(j',x')$ for $(j,x) \in \cP\times \ZZ_4^\#$.
There exists a permutation $\sigma_f \in S_n$ with
$\sigma_f(j)=j'$ for $j \in \cP$.
%The set of edge of subgraph $\Gamma(\cC')$ induced by $V_1'$ is equal to 
%$\{((j',x'),(j',2')), ((j',2'),(j',x')) \mid j \in \{1,\dots,n\}, x \in \ZZ_4-\{0,2\}\}$.
The set of
the vertices of $\Gamma(\cC)$ (resp.\ $\Gamma(\cC')$)
% whose indegrees are equal to $2$ 
whose indegrees are at least $2$ and outdegrees are equal to $2$ 
is $\{(j,2)\mid j \in \cP\}$ (resp.\ $\{(j',2)\mid j' \in \cP\}$).
Hence, $(j',2')=(j',2)$ for each $j' \in \cP$.
Also, we have either that $(j',1')=(j',1)$ and $(j',3')=(j',3)$ or 
that $(j',1')=(j',3)$ and $(j',3')=(j',1)$ for each $j' \in \cP$.
% By $-3=1$ and $-1=3$ in $\ZZ_4$, we have 
Hence, we have
\[
\tau_{j_1}\tau_{j_2}\cdots\tau_{j_\ell}\sigma_f(\cC) = 
f(\cC^\#)\cup\{\mathbf{0}\}=\cC',
\]
where
$\{j_1,j_2,\ldots,j_\ell\}=\{j \in \cP \mid (j',1')=(j',3),(j',3')=(j',1)\}$ 
with $|\{j_1,j_2,\ldots,j_\ell\}|=\ell$.
Therefore, two linear $\ZZ_4$-codes $\cC, \cC'$ are equivalent.
\end{proof}

% \begin{remark}
% For a linear $\ZZ_4$-code $\cC$,
% we define a subset $S_{\alpha,\beta}(\cC)
% =\{x \in \cC \mid (n_1(x)+n_3(x),n_2(x)) =(\alpha,\beta)\}$,
% where $(\alpha,\beta) \in 
% Y=\{(a,b) \mid a,b \in \{0,1,\ldots,n\}, 
% (a, b)\ne (0,0)\}$.
% Let $\cC,\cC'$ be linear $\ZZ_4$-codes of length $n$.
% Suppose that there exists a subset $X$ of $Y$
% such that 
% $
% \langle 
% \cup_{(\alpha,\beta) \in X}S_{\alpha,\beta}(\cC) 
% \rangle=\cC$,
% $\langle 
% \cup_{(\alpha,\beta) \in X}S_{\alpha,\beta}(\cC')
% \rangle=\cC'$. 
% By applying the definition of $\Gamma(\cC)$, 
% one may define the digraph $\Gamma(S_{\alpha,\beta}(\cC))$.
% Then it is easy to see that 
% $\cC,\cC'$ are equivalent if and only if 
% $\Gamma(S_{\alpha,\beta}(\cC)),\Gamma(S_{\alpha,\beta}(\cC'))$
% are isomorphic.
% \end{remark}
% 

% In the process of 
% the classification of linear $\ZZ_4$-codes of 
% length $2^m$ satisfying~\eqref{eq:Z4C1-2} and~\eqref{eq:Z4C2-2},
% we need to find a set of complete representatives 
% for $\ZZ_4^{2^m}/ZRM(1,m)$.
% For the cases $m=3,4$, we found a set 
% of complete representatives for $\ZZ_4^{2^m}/ZRM(1,m)$
% using some set of complete representatives
% for $\ZZ_2^{2^{m+1}}/RM(1,m+1)$, which
% was found by the {\sc Magma} function {\tt Transversal}. 
% {\bf (Harada: I am not sure if this part needs or not)}

%%%%%
%R \subsubsection{Linear $\ZZ_4$-codes of 
%R length $16$ satisfying~\eqref{eq:Z4C1-2} and~\eqref{eq:Z4C2-2}}\label{Subsec:Z4-16}

Using the above method,
by a computer calculation,
we completed the classification of linear $\ZZ_4$-codes of 
length $16$ satisfying~\eqref{eq:Z4C1-2} and~\eqref{eq:Z4C2-2}.
By the {\sc Magma} function {\tt IsIsomorphic},
we determined whether $\Gamma(\cC),\Gamma(\cC')$
are isomorphic.
%R 
%R Using the above method,
%R by a computer calculation,
%R we completed the classification of linear $\ZZ_4$-codes of 
%R length $16$ satisfying~\eqref{eq:Z4C1-2} and~\eqref{eq:Z4C2-2}.
%R To find out equivalences of two linear $\ZZ_4$-codes $\cC,\cC'$,
%R by Proposition~\ref{prop:Z4}
%R we checked whether their digraphs $\Gamma(\cC),\Gamma(\cC')$
%R are isomorphic.
%R This was calculated by the {\sc Magma} function
%R {\tt IsIsomorphic}.

\begin{proposition}\label{prop:16-2}
Let $N_4(16,k)$ denote the number of 
inequivalent linear $\ZZ_4$-codes $\cC$ of length $16$ with $|\cC|=2^k$
satisfying~\eqref{eq:Z4C1-2} and~\eqref{eq:Z4C2-2}.
Then
$N_{4}(16,7)=5$,
$N_{4}(16,8)=21$,
$N_{4}(16,9)=62$,
$N_{4}(16,10)=28$,
$N_{4}(16,11)=2$ and 
$N_{4}(16,12)=0$.
% Up to  monomially equivalence,
% there are five $\ZZ_4$-codes $C$ of length $16$ with $|C|=2^7$
% satisfying~\eqref{eq:Z4C1} and~\eqref{eq:Z4C2},
% none of which is maximal.
% Up to  monomially  equivalence,
% there are $21$ $\ZZ_4$-codes $C$ of length $16$ with $|C|=2^8$
% satisfying~\eqref{eq:Z4C1} and~\eqref{eq:Z4C2},
% none of which is maximal.
% Up to monomially  equivalence,
% there are $62$ $\ZZ_4$-codes $C$ of length $16$ with $|C|=2^9$
% satisfying~\eqref{eq:Z4C1} and~\eqref{eq:Z4C2},
% $7$ of which is maximal.
% Up to  monomially equivalence,
% there are $28$ $\ZZ_4$-codes $C$ of length $16$ with $|C|=2^{10}$
% satisfying~\eqref{eq:Z4C1} and~\eqref{eq:Z4C2},
% $19$ of which is maximal.
% Up to monomially equivalence, 
% there are two $\ZZ_4$-codes $C$ of length $16$ with $|C|=2^{11}$
% satisfying~\eqref{eq:Z4C1} and~\eqref{eq:Z4C2},
% two of which is maximal.
% There is no $\ZZ_4$-code $C$ of length $16$ with $|C|=2^{k}$
% satisfying~\eqref{eq:Z4C1} and~\eqref{eq:Z4C2}
% for $k \ge 12$.
\end{proposition}

To list the result of the classification,
we fix the generator matrix of $ZRM(1,4)$ as follows:
%\begin{equation}\label{eq:G}
\[
\left(
\begin{array}{c}
1 1 1 1 1 1 1 1 1 1 1 1 1 1 1 1\\
0 2 0 2 0 2 0 2 0 2 0 2 0 2 0 2\\
0 0 2 2 0 0 2 2 0 0 2 2 0 0 2 2\\
0 0 0 0 2 2 2 2 0 0 0 0 2 2 2 2\\
0 0 0 0 0 0 0 0 2 2 2 2 2 2 2 2
\end{array}
\right).
\]
To save space, we only list the maximal linear $\ZZ_4$-codes
(with respect to the subset relation)
given in the above proposition.
The seven maximal linear
$\ZZ_4$-codes $\cC=\cC_{16,3,i}$ $(i=1,2,\ldots,7)$ with $|\cC|=2^9$
are constructed as
$\langle ZRM(1,4), x_1,x_2,x_3 \rangle$,
where $x_1,x_2,x_3$ are listed in Table~\ref{Tab:16-2-2-1}.
The $19$  maximal linear $\ZZ_4$-codes $\cC=\cC_{16,4,i}$  $(i=1,2,\ldots,19)$
with $|\cC|=2^{10}$ are constructed as
$\langle ZRM(1,4), x_1,x_2,x_3,x_4 \rangle$,
where $x_1,x_2,x_3,x_4$ are listed in Table~\ref{Tab:16-2-2-2}.
The two maximal linear
$\ZZ_4$-codes $\cC=\cC_{16,5,i}$  $(i=1,2)$ with $|\cC|=2^{11}$
are constructed as
$\langle ZRM(1,4), x_1,x_2,\ldots,x_5 \rangle$,
where $x_1,x_2,\ldots,x_5$ are listed in Table~\ref{Tab:16-2-2-3}.
For each code $\cC$, 
by a computer calculation, we determined
the value $\beta^2$ in~\eqref{eq:Z4C1-2},
the minimum Hamming distance $d_H(\cC)$  and
the minimum Lee distance $d_L(\cC)$, which are listed in
Table~\ref{Tab:16-2}.

%%%%%%%%%%%%%%%%%%%%%%%%%%%
\begin{table}[thb]
\caption{Vectors $x_j$ for $\cC_{16,3,i}$ $(i=1,2,\ldots,7)$}
\label{Tab:16-2-2-1}
\begin{center}
%{\small
{\footnotesize
%{\scriptsize
%\begin{tabular}{c|c|l|ccccccc}
\begin{tabular}{c|l}
\noalign{\hrule height0.8pt}
Code & \multicolumn{1}{c}{$x_j$ $(j=1,2,3)$}\\
\hline
$\cC_{16,3,1}$ 
&$(1,0,0,3,0,1,3,0,0,1,3,0,1,0,0,3),(0,1,1,2,0,1,1,2,0,1,1,2,0,1,1,2)$,\\
&$(0,0,0,0,1,1,3,3,1,1,3,3,0,0,0,0)$\\
$\cC_{16,3,2}$ 
&$(1,0,0,1,0,1,1,2,0,1,3,0,1,0,2,3),(0,1,0,1,1,2,3,0,1,0,1,0,0,3,2,1)$,\\
&$(0,0,1,1,1,3,0,2,1,3,2,0,2,2,1,1)$\\
$\cC_{16,3,3}$ 
&$(1,0,0,1,1,2,2,3,0,1,1,0,0,1,3,0),(0,1,0,1,0,1,2,3,1,0,3,0,3,0,3,2)$,\\
&$(0,0,1,1,1,3,0,2,1,3,2,0,0,0,3,3)$\\
$\cC_{16,3,4}$ 
&$(1,0,0,1,0,3,1,2,0,1,3,2,1,2,2,1),(0,1,0,1,1,0,3,0,0,1,0,3,1,2,1,2)$,\\
&$(0,0,1,1,1,1,0,2,0,2,1,1,3,3,0,0)$\\
$\cC_{16,3,5}$ 
&$(1,0,0,1,0,3,3,0,0,1,1,0,1,2,2,1),(0,1,1,0,0,1,1,2,1,2,2,1,3,2,0,3)$,\\
&$(0,0,2,0,0,0,0,2,1,1,1,1,1,3,3,1)$\\
$\cC_{16,3,6}$ 
&$(1,0,0,1,0,3,1,2,0,1,3,2,1,2,2,1),(0,1,0,1,0,1,2,3,1,2,3,0,3,0,3,0)$,\\
&$(0,0,1,1,0,2,3,1,1,3,0,2,1,1,2,2)$\\
$\cC_{16,3,7}$ 
&$(0,2,0,2,0,0,2,2,0,0,2,2,0,2,0,2),(0,0,0,0,2,0,0,2,0,0,0,0,0,2,2,0)$,\\
&$(0,0,0,0,0,2,2,0,0,0,0,0,0,2,2,0)$\\
\noalign{\hrule height0.8pt}
\end{tabular}
}
\end{center}
\end{table}
%%%%%%%%%%%%%%%%%%%%%%%%%%%%%%%%%%%%%%%%%%%%%%%%

%%%%%%%%%%%%%%%%%%%%%%%%%%%
\begin{table}[thbp]
\caption{Vectors $x_j$ for $\cC_{16,4,i}$  $(i=1,2,\ldots,19)$}
\label{Tab:16-2-2-2}
\begin{center}
%{\small
{\footnotesize
%{\scriptsize
%\begin{tabular}{c|c|l|ccccccc}
\begin{tabular}{c|l}
\noalign{\hrule height0.8pt}
Code & \multicolumn{1}{c}{$x_j$ $(j=1,2,3,4)$}\\
\hline
$\cC_{16,4,1}$
&$(1,0,0,1,0,3,1,0,1,0,2,3,0,1,1,2),(0,1,0,1,0,1,2,1,0,3,2,1,2,3,2,1)$,\\
&$(0,0,1,1,0,2,3,3,1,3,0,0,3,1,0,2),(0,0,0,2,0,2,2,2,0,0,0,2,2,0,0,0)$\\
$\cC_{16,4,2}$
&$(1,0,0,1,0,3,1,2,1,2,2,3,2,1,1,0),(0,1,0,1,0,3,0,3,0,1,2,3,0,3,2,1)$,\\
&$(0,0,1,1,0,2,1,1,1,1,2,2,3,3,0,2),(0,0,0,2,0,0,0,2,1,1,1,1,1,3,1,3)$\\
$\cC_{16,4,3}$
&$(1,0,0,1,0,1,3,2,0,1,1,0,1,0,2,3),(0,1,0,1,1,0,1,0,1,0,1,0,2,3,2,3)$,\\
&$(0,0,1,1,1,1,2,2,0,0,3,3,1,1,0,0),(0,0,0,0,2,0,0,2,1,1,1,1,1,3,3,1)$\\
$\cC_{16,4,4}$
&$(1,0,0,1,0,1,1,0,1,0,2,1,0,3,3,2),(0,1,1,0,0,3,1,0,0,3,1,0,2,1,1,2)$,\\
&$(0,0,2,0,0,0,0,2,0,0,0,2,2,2,0,2),(0,0,0,2,0,0,0,2,1,1,3,3,3,3,3,3)$\\
$\cC_{16,4,5}$
&$(1,0,0,1,0,1,3,2,1,2,2,1,2,1,3,0),(0,1,0,1,1,0,1,0,0,3,0,3,1,2,1,2)$,\\
&$(0,0,1,1,1,1,0,0,0,2,3,1,3,1,0,2),(0,0,0,0,2,0,2,0,1,3,1,3,3,3,3,3)$\\
$\cC_{16,4,6}$
&$(1,0,0,1,0,1,3,2,0,1,1,0,1,0,2,3),(0,1,0,1,0,3,2,1,0,1,0,1,2,1,0,3)$,\\
&$(0,0,1,1,0,2,1,3,0,2,3,1,0,0,3,3),(0,0,0,2,0,2,0,0,0,2,0,0,2,2,2,0)$\\
$\cC_{16,4,7}$
&$(1,0,0,1,0,1,1,0,0,3,1,2,3,0,2,1),(0,1,0,1,1,0,1,0,1,2,3,0,0,3,2,1)$,\\
&$(0,0,1,1,1,3,0,2,0,2,3,1,3,3,0,0),(0,0,0,0,2,2,0,0,1,3,3,1,3,1,3,1)$\\
$\cC_{16,4,8}$
&$(1,0,0,1,0,1,3,2,0,3,3,0,1,2,0,3),(0,1,0,1,1,0,1,0,1,2,3,0,2,1,0,3)$,\\
&$(0,0,1,1,1,1,2,2,0,2,3,1,3,1,2,0),(0,0,0,0,2,0,0,2,1,3,1,3,3,3,1,1)$\\
$\cC_{16,4,9}$
&$(1,0,0,1,0,1,1,0,0,1,3,2,1,0,2,3),(0,1,0,1,1,0,1,0,0,3,0,3,1,2,1,2)$,\\
&$(0,0,1,1,1,3,0,2,1,1,0,0,0,2,1,3),(0,0,0,0,2,2,0,0,1,1,3,3,3,3,3,3)$\\
$\cC_{16,4,10}$
&$(1,0,0,1,0,1,3,2,1,0,2,1,2,1,3,2),(0,1,0,1,1,0,3,0,1,2,1,0,0,3,2,1)$,\\
&$(0,0,1,1,1,1,2,2,0,2,3,3,1,1,0,2),(0,0,0,0,2,0,0,2,1,1,3,1,1,1,1,3)$\\
$\cC_{16,4,11}$
&$(1,0,0,1,0,1,3,0,0,1,1,2,1,2,0,3),(0,1,1,0,0,1,3,2,0,3,1,2,2,3,3,2)$,\\
&$(0,0,2,0,0,0,2,0,1,1,1,1,3,1,1,3),(0,0,0,0,1,1,3,1,0,2,2,0,3,3,1,3)$\\
$\cC_{16,4,12}$
&$(1,1,0,0,0,0,1,3,1,3,0,2,0,2,3,3),(0,2,0,0,0,0,2,0,1,3,3,1,3,3,1,1)$,\\
&$(0,0,1,1,0,0,1,1,1,1,0,2,3,3,0,2),(0,0,0,2,0,0,0,2,0,0,2,0,2,2,0,2)$\\
$\cC_{16,4,13}$
&$(1,0,1,0,0,1,0,1,1,0,1,2,2,3,0,3),(0,1,0,1,0,1,0,3,0,3,2,3,0,3,2,1)$,\\
&$(0,0,2,0,0,0,2,0,0,0,0,2,2,2,2,0),(0,0,0,0,1,1,3,1,1,3,3,3,0,2,0,2)$\\
$\cC_{16,4,14}$
&$(1,0,0,1,0,3,3,0,0,1,1,0,1,2,2,1),(0,1,0,1,0,3,2,1,1,2,3,2,1,2,1,0)$,\\
&$(0,0,1,1,0,2,1,1,1,1,2,2,3,3,0,2),(0,0,0,2,0,0,0,2,1,1,1,1,1,3,1,3)$\\
$\cC_{16,4,15}$
&$(2,0,0,0,0,0,0,2,2,0,0,0,2,2,2,0),(3,0,0,0,0,2,0,1,3,2,0,0,3,1,1,2)$,\\
&$(0,1,0,1,0,3,0,3,0,3,2,3,2,3,0,3),(0,0,1,0,0,3,2,2,0,2,3,0,1,0,1,3)$\\
$\cC_{16,4,16}$
&$(1,0,0,1,0,1,3,2,1,2,2,1,2,1,3,0),(0,1,0,1,0,1,2,3,0,1,2,3,0,1,0,1)$,\\
&$(0,0,1,1,0,0,1,1,1,3,2,0,3,1,0,2),(0,0,0,2,0,0,0,2,1,3,3,3,1,3,3,3)$\\
$\cC_{16,4,17}$
&$(1,1,0,0,0,2,1,3,1,3,0,2,0,0,3,3),(0,2,0,0,0,0,0,2,0,0,0,2,2,0,2,2)$,\\
&$(0,0,1,1,0,2,1,1,1,1,0,2,3,3,2,2),(0,0,0,0,1,1,1,3,1,1,3,1,0,2,0,2)$\\
$\cC_{16,4,18}$
&$(1,0,0,1,1,0,0,1,1,0,0,1,1,0,0,1),(0,0,2,0,0,0,2,0,0,0,2,0,0,0,2,0)$,\\
&$(0,0,0,0,2,0,0,2,0,0,0,0,2,0,0,2),(0,0,0,0,0,0,0,0,2,0,0,2,2,0,0,2)$\\
$\cC_{16,4,19}$
&$(1,0,0,1,0,1,3,2,0,3,1,2,3,2,2,3),(0,2,0,0,0,2,2,2,0,2,2,2,2,0,2,2)$,\\
&$(0,3,0,0,0,3,3,3,0,3,1,3,1,2,3,1),(0,0,1,0,0,2,0,3,0,2,2,3,3,3,0,3)$\\
\noalign{\hrule height0.8pt}
\end{tabular}
}
\end{center}
\end{table}
%%%%%%%%%%%%%%%%%%%%%%%%%%%%%%%%%%%%%%%%%%%%%%%%

%%%%%%%%%%%%%%%%%%%%%%%%%%%
\begin{table}[thb]
\caption{Vectors $x_j$ for $\cC_{16,5,i}$  $(i=1,2)$}
\label{Tab:16-2-2-3}
\begin{center}
%{\small
{\footnotesize
%{\scriptsize
%\begin{tabular}{c|c|l|ccccccc}
\begin{tabular}{c|l}
\noalign{\hrule height0.8pt}
Code & \multicolumn{1}{c}{$x_j$ $(j=1,2,\ldots,5)$}\\
\hline
$\cC_{16,5,1}$ 
&$(1,0,0,1,0,1,3,0,1,0,2,3,0,3,3,2),(0,1,0,1,0,1,2,1,1,0,3,0,3,2,1,0)$,\\
&$(0,0,1,1,0,0,1,3,0,2,3,1,2,0,3,3),(0,0,0,2,0,0,0,2,1,1,1,1,3,1,1,3)$,\\
&$(0,0,0,0,1,1,1,3,1,1,3,1,0,2,0,2)$\\
$\cC_{16,5,2}$ 
&$(1,0,0,1,0,1,1,2,1,0,0,1,0,3,3,2),(0,1,0,1,0,1,2,1,1,2,3,2,3,0,1,2)$,\\
&$(0,0,1,1,0,0,3,1,0,2,3,1,0,2,3,3),(0,0,0,2,0,0,2,0,1,1,1,1,1,3,1,3)$,\\
&$(0,0,0,0,1,1,1,3,1,3,1,1,2,2,0,0)$\\
\noalign{\hrule height0.8pt}
\end{tabular}
}
\end{center}
\end{table}
%%%%%%%%%%%%%%%%%%%%%%%%%%%%%%%%%%%%%%%%%%%%%%%%

% \begin{remark}
% By a computer calculation, we
% verified that the Gray images of the maximal linear $\ZZ_4$-codes
% $\cC_{16,i,j}$ ($(i,j)=(3,1),(3,7),(4,18)$)
% are linear.
% This was calculated by the {\sc Magma} function 
% {\tt HasLinearGrayMapImage}.
% \end{remark}

%%%%%%%%%%%%%%%%%%%%%%%%%%%
\begin{table}[thb]
\caption{Maximal linear $\ZZ_4$-codes of length $16$ 
satisfying~\eqref{eq:Z4C1-2} and~\eqref{eq:Z4C2-2}}
\label{Tab:16-2}
\begin{center}
%{\small
{\footnotesize
%{\scriptsize
%\begin{tabular}{c|c|l|ccccccc}
\begin{tabular}{c|c|c||c|c|c}
\noalign{\hrule height0.8pt}
$\cC$    & $\beta^2$ & $(d_H(\cC),d_L(\cC))$ &
$\cC$    & $\beta^2$ & $(d_H(\cC),d_L(\cC))$ \\
\hline
$\cC_{16,3,1}$ &64 &$( 4,  8)$&% 589824&
$\cC_{16,3,2}$ &16 &$( 8, 12)$\\%& 64\\
$\cC_{16,3,3}$ &16 &$( 8, 12)$&% 96&
$\cC_{16,3,4}$ &16 &$( 8, 12)$\\%& 256\\
$\cC_{16,3,5}$ &16 &$( 8, 12)$&% 192&
$\cC_{16,3,6}$ &16 &$( 8, 12)$\\%& 43008\\
$\cC_{16,3,7}$ &64 &$( 4,  8)$&% 1321205760&
 & &\\
\hline
$\cC_{16,4, 1}$&16& $(6, 12)$&%  64&
$\cC_{16,4, 2}$&16& $(8, 12)$\\%&   4\\
$\cC_{16,4, 3}$&16& $(8, 12)$&%  32&
$\cC_{16,4, 4}$&16& $(6, 12)$\\%&  64\\
$\cC_{16,4, 5}$&16& $(8, 12)$&%  64&
$\cC_{16,4, 6}$&16& $(6, 12)$\\%&  32\\
$\cC_{16,4, 7}$&16& $(8, 12)$&% 128&
$\cC_{16,4, 8}$&16& $(8, 12)$\\%&  32\\
$\cC_{16,4, 9}$&16& $(8, 12)$&% 128&
$\cC_{16,4,10}$&16& $(8, 12)$\\%&   4\\
$\cC_{16,4,11}$&16& $(6, 12)$&%   8&
$\cC_{16,4,12}$&16& $(6, 12)$\\%&   8\\
$\cC_{16,4,13}$&16& $(6, 12)$&%   8&
$\cC_{16,4,14}$&16& $(8, 12)$\\%&   4\\
$\cC_{16,4,15}$&16& $(6, 12)$&%  12&
$\cC_{16,4,16}$&16& $(8, 12)$\\%&1920\\
$\cC_{16,4,17}$&16& $(6, 12)$&% 192&
$\cC_{16,4,18}$&64& $(4,  8)$\\%& 924844032\\
$\cC_{16,4,19}$&16& $(6, 12)$&%  48&
 &&\\
\hline
$\cC_{16,5,1}$ &16 &$( 6, 12)$&% 160&
$\cC_{16,5,2}$ &16 &$( 6, 12)$\\%& 160\\
\noalign{\hrule height0.8pt}
\end{tabular}
}
\end{center}
\end{table}
%%%%%%%%%%%%%%%%%%%%%%%%%%%%%%%%%%%%%%%%%%%%%%%%

\begin{remark}
% By a computer calculation based on Proposition~\ref{prop:Z4}, we
% verified that $\cC_{16,4,16}$ is equivalent to the Kerdock $\ZZ_4$-code
% of length $16$.
By a computer calculation, we
verified that $\Gamma(\cC_{16,4,16})$ and $\Gamma({\mathcal K}(4))$
are isomorphic.
\end{remark}

%%%%%%%%%%%%%%%%%%%%%%%%%%%%%%%%%%%%%%%%%%%%
\section{Weakly unbiased Hadamard matrices}
\label{sec:weakly}

In analogy to the case of quasi-unbiased Hadamard matrices,
this section studies weakly unbiased Hadamard matrices.
All feasible parameter sets for weakly unbiased 
Hadamard matrices are examined for orders up to $48$.
It is also shown that the size of a set of 
mutually weakly unbiased Hadamard matrices of
order $n$ is at most $2$.

%%%%%%%%%%%%%%%%%%%%%%%%
\subsection{Basic properties and feasible parameters}

Let $H,K$ be Hadamard matrices of order $n$.
Let $a_{ij}$ denote the $(i,j)$-entry of $HK^T$.
Recall that
$H,K$ are weakly unbiased if 
$a_{ij} \equiv 2 \pmod 4$ for $i,j \in \{1,2,\ldots,n\}$ and
$|\{ |a_{ij}| \mid i,j \in \{1,2,\ldots,n\}\}| \le 2$.
In this paper, we exclude unbiased Hadamard matrices
from weakly unbiased Hadamard matrices.
This implies that $|\{ |a_{ij}| \mid i,j \in \{1,2,\ldots,n\}\}| = 2$.
It follows immediately from the definition that $n \ge 8$.

Let $(H,K)$ be a pair of weakly unbiased Hadamard matrices of order $n$.
Suppose that $a,b$ are positive integers satisfying
$\{ |a_{ij}| \mid i,j \in \{1,2,\ldots,n\}\}=\{a,b\}$.
We denote the set $\{a,b\}$ by $\sigma(H,K)$.
%Let $n_i(a)$ be the number of components $j$ with $a_{ij}=\pm a$ 
%for $1 \le i \le n$.
%It follows from $(HK^T)(HK^T)^T=n^2 I$ that
%\begin{equation}
%\label{eq:weak}
%a^2 n_i(a) + b^2 (n-n_i(a)) = n^2,
%\end{equation}
%for $1 \le i \le n$.
%Then we have that $(a,b,n_i(a))=(2,n-2,n-1)$ satisfies~\eqref{eq:weak}.
Let $n(a)$ be the number of components $j$ with $a_{ij}=\pm a$ 
for $i=1,2,\ldots, n$.
From now on, 
we assume that $a < b$.
The value $n(a)$ does not depend on $i$. 
Indeed, it follows from $(HK^T)(HK^T)^T=n^2 I_n$ that
\begin{equation}
\label{eq:weak}
a^2 n(a) + b^2 (n-n(a)) = n^2.
\end{equation}
%for $1 \le i \le n$.
We say that parameters $(a,b)$ satisfying~\eqref{eq:weak}
are {\em feasible}.   % for each order $n$.
Since $(a,b,n(a))=(2,n-2,n-1)$ satisfies~\eqref{eq:weak},
the parameters $(a,b)=(2,n-2)$ are feasible for each order $n$.
% A pair of weakly unbiased Hadamard matrices $H,K$ with $\sigma(H,K)=\{2,6\}$
% is known for orders $12$ and $20$~\cite[Tables~3 and 6]{BK10}, respectively,
The following theorem gives an upper bound on the size of a set of 
mutually weakly unbiased Hadamard matrices, which is
one of the main results of this paper.

\begin{theorem}\label{thm:weakUB}
The size of a set of 
mutually weakly unbiased Hadamard matrices of
order $n$ is at most $2$.
\end{theorem}
\begin{proof}
Note that $n \ge 8$ by the definition.
Suppose that $\{H_1,H_2,H_3\}$ is a set of 
three mutually weakly unbiased Hadamard 
matrices of order $n$.
Let $h_i$ denote the first row of $H_i$ $(i=1,2,3)$.
By an argument similar to that in Proposition~\ref{prop:reduction},
$(H,K)$ is a pair of weakly unbiased Hadamard matrices 
if and only if  
$(HP,KP)$ is a pair of weakly unbiased Hadamard matrices 
for any monomial $(1,-1,0)$-matrix $P$.
Hence, without loss of generality, we may assume the following:
\[
\begin{array}{cccccccc}
h_1 =&(&+ \cdots + & + \cdots + & + \cdots + & + \cdots + &),\\
h_2 =&(&+ \cdots + & + \cdots + & - \cdots - & - \cdots - &),\\
h_3 =&(&
\underbrace{+ \cdots +}_{s \text{ columns}} &
{- \cdots -} &
{+ \cdots +} &
{- \cdots -}&).
%\underbrace{- \cdots -}_{d \text{ columns}} &
%\underbrace{+ \cdots +}_{e \text{ columns}} &
%\underbrace{- \cdots -}_{f \text{ columns}}).
\end{array}
\]
It follows that
$4s=n+ h_1 \cdot h_2
+ h_1 \cdot h_3
+ h_2 \cdot h_3$.
This gives a contradiction to the fact
that $
h_1 \cdot h_2  \equiv
h_1 \cdot h_3  \equiv
h_2 \cdot h_3  \equiv 2 \pmod 4$.
\end{proof}
\begin{remark}
The above theorem is known for 
$n \equiv 4 \pmod 8$~\cite[Lemma~13]{BK10}.
\end{remark}

For $n=4,8,\ldots,48$, we give in Table~\ref{Tab:weakP}
the feasible parameters $(a,b)$ along with $n(a)$.
%% except $(2,n-2,n-1)$.
The third column of the table indicates 
our present state of knowledge about the existence of
a pair of weakly unbiased Hadamard matrices for $n$ and $(a,b,n(a))$.

%%%%%%%%%%%%%%%%%%%%%%%%%%%
\begin{table}[thbp]
\caption{Weakly unbiased Hadamard matrices $(n=4,8,\ldots,48)$}
\label{Tab:weakP}
\begin{center}
%{\small
{\footnotesize
%{\scriptsize
%\begin{tabular}{c|c|l|ccccccc}
\begin{tabular}{c|c|c|l}
\noalign{\hrule height0.8pt}
$n$ &$(a,b,n(a))$ & Existence & \multicolumn{1}{c}{Reference}\\
\hline
8 & $(2,6,7)$ & Yes & Proposition~\ref{prop:weak}\\
\hline
12& $(2,6,9)$ & Yes & \cite[Table~3]{BK10}, Section~\ref{Subsec:weakB} \\
  & $(2,10,11)$ & Yes & Proposition~\ref{prop:weak}, Section~\ref{Subsec:weakB} \\
\hline
16& $(2,6,10)$ & Yes & Proposition~\ref{prop:weak2}, Section~\ref{Subsec:weakB} \\
  & $(2,10,14)$& No & Section~\ref{Subsec:weakCon} \\
  & $(2,14,15)$ & Yes & Proposition~\ref{prop:weak}, Section~\ref{Subsec:weakB} \\
\hline
20& $(2,6,10)$ & Yes & \cite[Table~6]{BK10} \\
  & $(2,18,19)$ & Yes & Proposition~\ref{prop:weak}, Section~\ref{Subsec:weakB} \\
\hline
24& $(2,6,9)$  & Yes  &  Section~\ref{Subsec:weakCon}\\
  & $(2,10,19)$ & No & Section~\ref{Subsec:weakCon} \\
  & $(2,22,23)$ & Yes & Proposition~\ref{prop:weak}, Section~\ref{Subsec:weakB} \\
\hline
28& $(2,6,7)$  & No &   Section~\ref{Subsec:weakCon}\\
  & $(2,10,21)$ & No &  Section~\ref{Subsec:weakCon}\\
  & $(2,26,27)$ & Yes & Proposition~\ref{prop:weak}\\
\hline
32& $(2,6,4)$  & ? & \\
  & $(2,30,31)$ & Yes & Proposition~\ref{prop:weak}, Section~\ref{Subsec:weakB}\\
\hline
36& $(2,10,24)$  & ? & \\
  & $(2,14,30)$  & ? & \\
  & $(2,34,35)$ & Yes & Proposition~\ref{prop:weak}\\
\hline
40& $(2,10,25)$  & ? & \\
  & $(2,22,37)$  & ? & \\
  & $(2,38,39)$ & Yes & Proposition~\ref{prop:weak}\\
  & $(6,14,39)$  & ? & \\
\hline
44& $(2,42,43)$ & Yes & Proposition~\ref{prop:weak}\\
\hline
48& $(2,10,26)$  & ? & \\
  & $(2,14,37)$  & ? & \\
  & $(2,46,47)$ & Yes & Proposition~\ref{prop:weak}\\
  & $(6,10,39)$  & ? & \\
  & $(6,18,46)$  & ? & \\
\noalign{\hrule height0.8pt}
\end{tabular}
}
\end{center}
\end{table}

Now, we give two methods for constructing weakly unbiased Hadamard matrices.
Let 
%$H=\begin{pmatrix}a&y^T\\x&H_1\end{pmatrix}$,
$H=\left(\begin{smallmatrix}a&y^T\\x&H_1\end{smallmatrix}\right)$,
%$H'=\begin{pmatrix}a'&y'^T\\x'&H_1'\end{pmatrix}$
$H'=\left(\begin{smallmatrix}a'&y'^T\\x'&H_1'\end{smallmatrix}\right)$
be Hadamard matrices of order $n$,
where $H_1,H'_1$ are $(n-1) \times (n-1)$ matrices,
$x,x',y,y'$ are $(n-1) \times 1$ matrices and
$a,a' \in \{1,-1\}$.
Let $K$ be the Hadamard matrix obtained from $H'$ by negating the first column. 
Then we have
\begin{equation}\label{eq:HH'}
H K^T
=HH'^T
+\begin{pmatrix}-2a a'&-2ax'^T\\-2a'x&-2xx'^T \end{pmatrix}.
\end{equation}
% Using this observation, we have the following two methods for
% constructing weakly unbiased Hadamard matrices.

\begin{proposition}\label{prop:weak}
If there exists a Hadamard matrix of order $n \ge 8$,
then there exists a pair of weakly unbiased Hadamard matrices 
$H,K$ of
order $n$ with $\sigma(H,K)=\{2,n-2\}$.
\end{proposition}
\begin{proof}
Suppose that $H'=H$.
%\[
%HK^T= nI_n+\begin{pmatrix}-2 &-2a x^T\\-2a x&-2xx^T \end{pmatrix}.
%\]
From~\eqref{eq:HH'}, 
the entries of $H K^T$ are $n-2, \pm 2$. 
The result follows.
\end{proof}

\begin{proposition}\label{prop:weak2}
Suppose that $n=4k^2$, where $k$ is even.
If there exists a pair of unbiased Hadamard matrices of order $n$, 
then there exists a pair of weakly unbiased Hadamard matrices $H,K$ of order $n$ with $\sigma(H,K)=\{\sqrt{n}-2,\sqrt{n}+2\}$.
\end{proposition}
\begin{proof}
Suppose that $(H,H')$ is a pair of unbiased Hadamard matrices of order $n$.
From~\eqref{eq:HH'}, 
the entries of $H K^T$ are $\pm\sqrt{n}\pm2$. 
The result follows.
\end{proof}

Since there exists a pair of unbiased Hadamard matrices of order 
$4^k$ for a positive integer $k$~\cite{CS73},
the above proposition implies the existence of 
a pair of weakly unbiased Hadamard matrices $H,K$ of order 
$4^k$ with $\sigma(H,K)=\{2^k-2,2^k+2\}$ for $k \ge 2$.

%%%%%%%%%%%%%%%%%%%%%%%%
\subsection{Observations by straightforward construction}
\label{Subsec:weakCon}

%R The observations given in Section~\ref{Subsec:Const} may be
%R applied to  weakly unbiased Hadamard matrices.

%R We denote by $H_{16,0}$ 
%R the Sylvester Hadamard matrix of order $16$.
%R By an exhaustive computer search, 
%R for each $i$ $(i=0,1,2,3,4)$
%R we verified that
%R there exists no $(1,-1)$-vector $x$ of length $16$
%R such that 
%R $|x \cdot r| \in \{2,10\}$ for all rows $r$ of $H_{16,i}$.
%R This means that there exists no pair of 
%R weakly unbiased Hadamard matrices
%R $H,K$ of order $16$ with $\sigma(H,K)=\{2,10\}$. 
%R 
%R Similarly, 
%R for each $H$ of the $60$ inequivalent Hadamard matrices of order $24$,
%R our exhaustive computer search verified that
%R there exists no $(1,-1)$-vector $x$ of length $24$
%R such that 
%R $ |x \cdot r| \in \{2,10\}$ for all rows $r$ of $H$.
%R This means that there exists no pair of 
%R weakly unbiased Hadamard matrices
%R $H,K$ of order $24$ with $\sigma(H,K)=\{2,10\}$. 

For each $H$ of the five inequivalent Hadamard matrices of order $16$
and the $60$ inequivalent Hadamard matrices of order $24$,
our exhaustive computer search verified that
there exists no $(1,-1)$-vector $x$ of lengths $16$ and $24$,
respectively, such that 
$ |x \cdot r| \in \{2,10\}$ for all rows $r$ of $H$.
This means that there exists no pair of 
weakly unbiased Hadamard matrices
$H,K$ of orders $16$ and $24$ with $\sigma(H,K)=\{2,10\}$. 
%
%By an argument similar to that given in Section~\ref{Subsec:Const},
%
We denote by $H_{24,3}$ \verb+had.24.8+ in~\cite{Hadamard},
which is a Hadamard matrix of order $24$.
Our computer search under the condition~\eqref{eq:3col} on $K$
found a Hadamard matrix $\overline{K_{24,3}}$ such that
$(H_{24,3}, \overline{K_{24,3}})$ is a pair of
weakly unbiased Hadamard matrices with
$\sigma(H_{24,3}, \overline{K_{24,3}})=\{2,6\}$,
where $\overline{K_{24,3}}$ is listed in Figure~\ref{Fig:24weak}.

%%%%%%%%%%%%%%%%%  Fig  %%%%%%%%%%%%%%%%%
\begin{figure}[thb]
\centering
{\scriptsize
\[
\overline{K_{24,3}}=
\left( \begin{array}{c}
+++-+--++---+--++--+--+-\\
++++-+++++--++--+++-+---\\
+++---+---+++++---------\\
++++-----++-+-++++--++++\\
++++++--+++--+++--++---+\\
+++++-++--++-----++++-++\\
++-------+---+----+++++-\\
++--++---+-+----++-----+\\
++--+-++++++--+++-+-++--\\
++--++++----++++-++--+++\\
++-+-++-+--++--+---+++-+\\
++-+-+-++-++-++-++-+-++-\\
+-+--+--+--+--++-++-+-+-\\
+-+--+++-++----+-+-+-+--\\
+-+-+-+-+----++-++-+++-+\\
+-+++++--+-++-+-+-++-++-\\
+-+-++-+++++++------++++\\
+-++---+---+-+-++-+--+-+\\
+--++++---+--+-++---+-+-\\
+--+--++++----+-------++\\
+--++--+-+-+++++-+-++---\\
+----+-+--+-+-+-+-+++--+\\
+-----+-++++++-+++++--++\\
+--++---+-+-+----++--+--\\
\end{array}
\right)
\]
}
\caption{The matrix $\overline{K_{24,3}}$}
\label{Fig:24weak}
\end{figure}
%%%%%%%%%%%%%%%%%  Fig  %%%%%%%%%%%%%%%%%

Now, for a Hadamard matrix $H$ of order $n$,
we consider the following graph $\Gamma(H,\{a,b\})$, in order to
convert the problem of finding $K$ such that $(H,K)$ is 
a pair of weakly unbiased Hadamard matrices
of order $n$ with $\sigma(H,K)=\{a,b\}$ into that of
finding an $n$-clique in the graph.
% We describe a method to establish the nonexistence of a pair of
% weakly unbiased Hadamard matrices $H,K$ of order $n$ 
% with $\sigma(H,K)=\{a,b\}$.
Let $h_i$ be the $i$-th row of $H$.
Set
\[
V_j=
\{x \in X_j \mid |x \cdot h_i| \in \{a,b\}\ (i=1,2,\ldots,n)\}
\quad (j=1,2,3,4),
\]
where
% \begin{align*}
% X_1&=\{(x_1,x_2,\ldots,x_{n}) \in \{1,-1\}^n \mid x_1=x_2=x_3=1\}, \\
% X_2&=\{(x_1,x_2,\ldots,x_{n}) \in \{1,-1\}^n \mid x_1=x_2=1,x_3=-1\}, \\
% X_3&=\{(x_1,x_2,\ldots,x_{n}) \in \{1,-1\}^n \mid x_1=x_3=1,x_2=-1\}, \\
% X_4&=\{(x_1,x_2,\ldots,x_{n}) \in \{1,-1\}^n \mid x_1=1,x_2=x_3=-1\}. 
% \end{align*}
$
X_j=\{(x_1,x_2,\ldots,x_{n}) \in \{1,-1\}^n \mid (x_1,x_2,x_3)=
Y_j\}$ with
$Y_1=(1,1,1)$, $Y_2=(1,1,-1)$,
$Y_3=(1,-1,1)$ and $Y_4=(1,-1,-1)$.
We define the simple graph $\Gamma(H,\{a,b\})$, 
whose set of vertices
is $V=V_1\cup V_2 \cup V_3 \cup V_4$ 
and two vertices $x,y \in V$ are adjacent 
if $x \cdot y =0$.
It follows that
the graph $\Gamma(H,\{a,b\})$ contains an $n$-clique  if and only if
there exists a Hadamard matrix $K$ of order $n$ such that
$(H,K)$ is a pair of weakly unbiased Hadamard matrices of order $n$ 
with $\sigma(H,K)=\{a,b\}$.
%%%%%%%%%%%%%
%R Up to equivalence, there exist $487$ Hadamard matrices of order $28$.
We denote by $H_{28,1},H_{28,2},\ldots,H_{28,487}$ 
\verb+had.28.1+, 
\verb+had.28.2+, $\ldots$,
\verb+had.28.487+ in~\cite{Hadamard}, respectively,
which are the $487$ inequivalent Hadamard matrices of order $28$.
By a computer calculation, we verified that each of 
the four induced subgraphs on $V_j$ $(j=1,2,3,4)$ of 
$\Gamma(H_{28,i},\{2,6\})$ (resp.\ $\Gamma(H_{28,i},\{2,10\})$)
contains a $7$-clique 
for only $i=54,295,456,479,484,487$
(resp.\ $i=128,197,295,297,374,445,453,456,476$, $477, 478, 479,481,485$).
For these $i$, we list in Table~\ref{Tab:mc2826210}
the sizes $mc(i)$ of the maximum cliques of $\Gamma(H_{28,i},\{2,6\})$
and $\Gamma(H_{28,i},\{2,10\})$.
From Table~\ref{Tab:mc2826210},
there exists no pair 
of weakly unbiased Hadamard matrices $H,K$ of order $28$ 
with $\sigma(H,K)=\{2,6\}$ and $\{2,10\}$.
Our calculations for finding cliques in this section were done by
a computer calculation using the {\sc Cliquer} software~\cite{Cliquer}.

%%%%%%%%%%%%%%%%%%%%%%%%%%%%%%%%%%%%%%%%%
\begin{table}[thb]
\caption{Maximum cliques of
 $\Gamma(H_{28,i},\{2,6\})$ and $\Gamma(H_{28,i},\{2,10\})$}
\label{Tab:mc2826210}
\begin{center}
%{\small
{\footnotesize
%{\scriptsize
%\begin{tabular}{c|c|l|ccccccc}
\begin{tabular}{c|ccccccc}
\noalign{\hrule height0.8pt}
Graph & \multicolumn{6}{c}{$(i,mc(i))$}\\
\hline
$\Gamma(H_{28,i},\{2,6\})$ &
$( 54, 12)$&$(295, 14)$&$(456, 12)$&$(479, 26)$&$(484, 26)$ & $(487, 16)$ \\
\hline
$\Gamma(H_{28,i},\{2,10\})$ &
$(128,  9)$&$(197, 10)$&$(295, 16)$&$(297, 12)$&$(374, 10)$&$(445, 12)$\\
 & 
$(453, 10)$&$(456, 12)$&$(476, 12)$&$(477, 10)$&$(478, 12)$&$(479, 14)$\\
 & 
$(481, 12)$&$(485, 12)$\\
\noalign{\hrule height0.8pt}
\end{tabular}
}
\end{center}
\end{table}
%%%%%%%%%%%%%%%%%%%%%%%%%%%%%%%%%%%%%%%%%

% To find a pair of weakly unbiased Hadamard matrices of 
% order $32$ with $\sigma(H,K)=\{2,6\}$, 
% we examined over $300$ matrices as $H$,
% by the method is similar to the above, 
% however, our computer search failed to discover such a pair.

%%%%%%%%%%%%%%%%%%%%%%%%%%%%%%%%%%%%%%%%%%%%%%
\section{A coding-theoretic approach to weakly unbiased Hadamard matrices}
\label{sec:weakF2}

In this section, we give a coding-theoretic approach to 
weakly unbiased Hadamard matrices.
For modest lengths, we give classifications of some binary 
self-complementary codes, in order to construct 
weakly unbiased Hadamard matrices.

%%%%%%%%%%%%
\subsection{Binary codes and weakly unbiased Hadamard matrices}
%R \label{Subsec:weakB}

Similar to Theorem~\ref{thm:F2},
we give a coding-theoretic approach to 
weakly unbiased Hadamard matrices.

\begin{theorem}\label{thm:weakF2}
Let $a, b$ be odd integers with $0<a < b < n/2$.
There exists a self-complementary 
$(n,4n)$ code $C$ satisfying the following conditions:
\begin{align}
\label{eq:weakB1}
&\{i \in \{0,1,\ldots,n\}\mid A_i(C) \ne 0\}=\{0,n/2 \pm a,n/2 \pm b,n/2,n\},\\
\label{eq:weakB2}
&A_{n/2}(C)=2n-2, \\
\label{eq:weakB3}
&C=C_1 \cup C_2,
\end{align}
where 
% $C_i$ has the same distance distribution as $RM(1,m)$.
each $C_i$ has distance distribution
$(A_0(C_i),A_{n/2}(C_i),A_n(C_i))=(1,2n-2,1)$
if and only if
there exists a pair of weakly unbiased Hadamard matrices $H,K$ with 
$\sigma(H,K)=\{2a,2b\}$.
\end{theorem}
\begin{proof}
The proof is similar to that of Theorem~\ref{thm:F2}. 
We remark that the condition $A_{n/2}(C)=2n-2$ corresponds to
the condition that $HK^T$ contains no zero entry.
\end{proof}

% \begin{corollary}
% Let $\alpha, \beta$ be odd integers with $0<\alpha < \beta < n/2$.
% If there exists a self-complementary 
% code $C$ of length $n$ satisfying 
%~\eqref{eq:weakB1} and~\eqref{eq:weakB2}.
% Then $|C| \le 4n$.
% \end{corollary}
% \begin{proof}
% Immediate from Propositions~\ref{prop:weakUB} and~\ref{prop:weakF2}.
% \end{proof}

%%%%%%%%%%%%
\subsection{Binary codes satisfying~\eqref{eq:weakB1}--\eqref{eq:weakB3}}
\label{Subsec:weakB}

By the method given in Section~\ref{Subsec:B},
for some $(n,2n)$ codes $C_1$ ($n=8,12,16,20,$ $24$),
our computer calculation completed the classification of codes of the form 
$C=C_1 \cup (u +C_1)$
satisfying~\eqref{eq:weakB1}--\eqref{eq:weakB3}.
Let $N_2(C_1)$ denote the number of inequivalent $(n,4n)$ codes 
of the form $C_1 \cup (u +C_1)$
satisfying~\eqref{eq:weakB1}--\eqref{eq:weakB3}.

%%%%%%%%%%%%
%R \subsubsection{Binary codes of length 8
%R satisfying~\eqref{eq:weakB1}--\eqref{eq:weakB3}}

% By the method given in Section~\ref{Subsec:BM}, we have the following:

% Let $C(H_{12})$ be the code obtained as 
% $24$ rows of $(1,0)$-matrices 
% $(H_{12}+J_{12})/2$ and $(-H_{12}+J_{12})/2$.

\begin{proposition}
$N_2(RM(1,3))=1$.
$N_2(C(H_{12}))=2$.
$N_2(RM(1,4))=2$,
$N_2(C(H_{16,1}))=4$,
$N_2(C(H_{16,2}))=6$,
$N_2(C(H_{16,3}))=3$ and
$N_2(C(H_{16,4}))=3$.
$N_2(C(H_{20,i}))=1$ $(i=1,2,3)$.
$N_2(C(H_{24,i}))=1$ $(i=1,2)$.
$N_2(RM(1,5))=1$.
\end{proposition}

The unique $(8,32)$ code $D_{8,1}$
is constructed as 
%$\langle RM(1,3),(1, 0, 0, 0, 0, 0, 0, 0) \rangle$.
$\langle RM(1,3),u_1 \rangle$, where $\supp(u_1)=\{1\}$.
% Here, $\supp(x)$ denotes the support of a vector $x$.
% By a computer calculation, we
% verified that the distance distribution is given by
% $( 1, 1, 0, 7, 14, 7, 0, 1, 1 )$.
%%% 
%R The distance distribution is listed in Table~\ref{Tab:weakDD}.
%R All  distance distributions given in this section were obtained by
%R a computer calculation.
%R 
%R 
%%%%%%%%%%%%
%R \subsubsection{Binary codes of length 12
%R satisfying~\eqref{eq:weakB1}--\eqref{eq:weakB3}}\label{Subsec:weakB12}
%R 
%R \begin{proposition}
%R $N_2(C(H_{12}))=2$.
%R \end{proposition}
%R 
The two  $(12,48)$ codes $D_{12,i}$ $(i=1,2)$
are constructed as $C(H_{12}) \cup (u_i+C(H_{12}))$,
where
$\supp(u_1)=\{1\}$ and $\supp(u_2)=\{1,2,3\}$.
% \[
% u_1=    (1, 0, 0, 0, 0, 0, 0, 0, 0, 0, 0, 0),
% u_2=    (1, 1, 1, 0, 0, 0, 0, 0, 0, 0, 0, 0).
% \]
% By a computer calculation, we
% verified that these distance distributions are given by:
% \begin{align*}
% (1, 1, 0, 0, 0, 11, 22, 11, 0, 0, 0, 1, 1), \quad
% (1, 0, 0, 3, 0, 9, 22, 9, 0, 3, 0, 0, 1), 
% \end{align*}
% respectively.
%R The distance distributions are listed in Table~\ref{Tab:weakDD}.
%R By Theorem~\ref{thm:weakF2},
%R these codes imply pairs of 
%R weakly unbiased Hadamard matrices $H,K$ of 
%R order $12$ with $\sigma(H,K)=\{2,10\}$ and $\{2,6\}$, respectively.
%R 
%%%%%%%%%%%%
%R \subsubsection{Binary codes of length 16
%R satisfying~\eqref{eq:weakB1}--\eqref{eq:weakB3}}\label{Subsec:weakB16}
%R 
%R \begin{proposition}
%R $N_2(RM(1,4))=2$,
%R $N_2(C(H_{16,1}))=4$,
%R $N_2(C(H_{16,2}))=6$,
%R $N_2(C(H_{16,3}))=3$ and
%R $N_2(C(H_{16,4}))=3$.
%R \end{proposition}
%R 
To save space, 
we only give the two $(16,64)$ codes
$D_{16,0,i}$ $(i=1,2)$ corresponding to $N_2(RM(1,4))$.
The two codes are constructed as $\langle RM(1,4),u_i \rangle$,
where
$\supp(u_1)=\{1\}$ and 
$\supp(u_2)=\{1,2,3,5,9\}$.
Let $H_{20,1}$ be the Paley Hadamard matrix of order $20$ having the 
form~\eqref{eq:R},
% \[
% \left(
% \begin{array}{cccc}
% + & +  & \cdots & + \\
% + & {}     & {}     &{} \\
% \vdots & {}     & R      &{} \\
% + & {}     &{}      &{} \\
% \end{array}
% \right),
% \]
where $R$ is the $19 \times 19$
circulant matrix with first row:
\[
%(-1,1,-1,-1,1,1,1,1,-1,1,-1,1,-1,-1,-1,-1,1,1,-1).
(-+--++++-+-+----++-).
\]
%R Up to equivalence, there exist three Hadamard matrices of order $20$.
We denote by $H_{20,2},H_{20,3}$ 
\verb+had.20.toncheviii+, \verb+had.20.toncheviv+
in~\cite{Hadamard}, respectively, which are the
remaining two Hadamard matrices of order $20$.
%Let $C_{20,i}$ be the nonlinear code consisting of the $40$ rows of 
%$(1,0)$-matrices
%$(H_{20,i}+J_[20])/2$ and $(-H_{20,i}+J_{20})/2$ $(i=1,2,3)$.
%R 
%R \begin{proposition}
%R $N_2(C(H_{20,i}))=1$ $(i=1,2,3)$.
%R $N_2(C(H_{24,i}))=1$ $(i=1,2)$.
%R $N_2(RM(1,5))=1$.
%R \end{proposition}
%R 
The unique $(20,80)$ code $D_{20,i}$
is constructed as 
$C(H_{20,i}) \cup (u+C(H_{20,i}))$,
where $\supp(u)=\{1\}$ $(i=1,2,3)$.
%R The distance distributions are listed in Table~\ref{Tab:weakDD}.
% % \begin{align*}
% % u_1=& (1,0,0,0,0,0,0,0,0,0,0,0,0,0,0,0,0,0,0,0),\\
% % u_2=& (0,1,0,0,0,0,0,0,0,0,0,0,0,0,0,0,0,0,0,0).
% % \end{align*}
% The unique code $D_{20,2,1}$
% is constructed as 
% $C(H_{20,2}) \cup (u_1+C(H_{20,2}))$.
% The unique code $D_{20,3,1}$
% is constructed as 
% $C(H_{20,3}) \cup (u_1+C(H_{20,3}))$.
% % \begin{align*}
% % u_3=& (0,0,1,0,0,0,0,0,0,0,0,0,0,0,0,0,0,0,0,0).
% % \end{align*}
%%%
% By a computer calculation, we
% verified that the distance distributions of the codes are given by:
% %\begin{align*}
% %(A_0,A_1,\ldots,A_{20})
% %= (1, 1, 0, 0, 0, 0, 0, 0, 0, 19, 38, 19, 0, 0, 0, 0, 0, 0, 0, 1, 1).
% %\end{align*}
% \begin{align*}
% (1, 1, 0, 0, 0, 0, 0, 0, 0, 19, 38, 19, 0, 0, 0, 0, 0, 0, 0, 1, 1).
% \end{align*}
% %for $k=1,2,3$. 
%R 
%R 
% %%%%%%%%%%%%
%R \subsubsection{Binary codes of length 24
%R satisfying~\eqref{eq:weakB1}--\eqref{eq:weakB3}}\label{Subsec:weakB24}
%R 
%R 
%R \begin{proposition}
%R $N_2(C(H_{24,i}))=1$ $(i=1,2)$.
%R \end{proposition}
%R 
The unique $(24,96)$ code $D_{24,i}$
is constructed as $C(H_{24,i}) \cup (u+C(H_{24,i}))$,
where $\supp(u)=\{1\}$ $(i=1,2)$.
%R The distance distributions are listed in Table~\ref{Tab:weakDD}.
% By a computer calculation, we
% verified that the distance distributions of all the codes are given by:
% %\[
% %(A_0,A_1,\ldots,A_{24})
% %=
% %(1, 1, 0, 0, \ldots, 0, 0, 23, 46, 23, 0, 0, \ldots, 0, 0, 1, 1).
% %\]
% \[
% (1, 1, 0, 0, \ldots, 0, 0, 23, 46, 23, 0, 0, \ldots, 0, 0, 1, 1).
% \]
% %for $(k,i)=(1,1),(2,1),(2,2)$. 
%R 
%%%%%%%%%%%%
%R \subsubsection{Binary codes of length 32
%R satisfying~\eqref{eq:weakB1}--\eqref{eq:weakB3}}\label{Subsec:weakB32}
%R 
% By the method given in Section~\ref{Subsec:BM}, we have the following:
%R 
%R \begin{proposition}
%R $N_2(RM(1,5))=1$.
%R \end{proposition}
%R 
The unique $[32,7]$ code $D_{32,1}$
is constructed as $\langle RM(1,5),u \rangle$,
where
$\supp(u)=\{4\}$
% \[
% u=(0,0,0,1,0,0,0,0,0,0,0,0,0,0,0,0,0,0,0,0,0,0,0,0,0,0,0,0,0,0,0,0),
% \]
and the generator matrix of $RM(1,5)$ is given by:
\[
\left(
\begin{array}{c}
1 0 0 1 0 1 1 0 0 1 1 0 1 0 0 1 0 1 1 0 1 0 0 1 1 0 0 1 0 1 1 0\\
0 1 0 1 0 1 0 1 0 1 0 1 0 1 0 1 0 1 0 1 0 1 0 1 0 1 0 1 0 1 0 1\\
0 0 1 1 0 0 1 1 0 0 1 1 0 0 1 1 0 0 1 1 0 0 1 1 0 0 1 1 0 0 1 1\\
0 0 0 0 1 1 1 1 0 0 0 0 1 1 1 1 0 0 0 0 1 1 1 1 0 0 0 0 1 1 1 1\\
0 0 0 0 0 0 0 0 1 1 1 1 1 1 1 1 0 0 0 0 0 0 0 0 1 1 1 1 1 1 1 1\\
0 0 0 0 0 0 0 0 0 0 0 0 0 0 0 0 1 1 1 1 1 1 1 1 1 1 1 1 1 1 1 1
\end{array}
\right).
\]
%R The distance distribution is listed in Table~\ref{Tab:weakDD}.
All distance distributions are listed in Table~\ref{Tab:weakDD}.
The distance distributions were obtained by a computer calculation.

% By a computer calculation, we
% verified that the distance distribution of the code is given by:
% \[
% %(1, 1, 0, 0, 0, 0, 0, 0, 0, 0, 0, 0, 0, 0, 0, 31, 62, 31, 0, 0, 0, 0,
% %0, 0, 0, 0, 0, 0, 0, 0, 0, 1, 1).
% (1, 1, 0, 0, \ldots, 0, 0, 31, 62, 31, 0, 0, \ldots, 0, 0, 1, 1).
% \]

%%%%%%%%%%%%%%%%%%%%%%%%%%%
\begin{table}[thb]
\caption{Distance distributions}
\label{Tab:weakDD}
\begin{center}
%{\small
{\footnotesize
%{\scriptsize
%\begin{tabular}{c|c|l|ccccccc}
\begin{tabular}{c|c}
\noalign{\hrule height0.8pt}
Code & \multicolumn{1}{c}{$(A_0,A_1,\ldots,A_n)$} \\
\hline
$D_{8,1}$ & $( 1, 1, 0, 7, 14, 7, 0, 1, 1 )$ \\
$D_{12,1}$& $(1, 1, 0, 0, 0, 11, 22, 11, 0, 0, 0, 1, 1)$\\
$D_{12,2}$& $(1, 0, 0, 3, 0, 9, 22, 9, 0, 3, 0, 0, 1)$\\
$D_{16,0,1}$&$(1, 1, 0, 0, 0, 0, 0, 15, 30, 15, 0, 0, 0, 0, 0, 1, 1)$\\
$D_{16,0,2}$&$(1, 0, 0, 0, 0, 6, 0, 10, 30, 10, 0, 6, 0, 0, 0, 0, 1)$\\
$D_{20,i}$ $(i=1,2,3)$ & $(1, 1, 0, 0, \ldots, 0, 0, 19, 38, 19, 0, 0, \ldots, 0, 0, 1, 1)$\\
$D_{24,i}$ $(i=1,2)$ &$(1, 1, 0, 0, \ldots, 0, 0, 23, 46, 23, 0, 0, \ldots, 0, 0, 1, 1)$\\
$D_{32,1}$&$(1, 1, 0, 0, \ldots, 0, 0, 31, 62, 31, 0, 0, \ldots, 0, 0, 1, 1)$\\
\noalign{\hrule height0.8pt}
\end{tabular}
}
\end{center}
\end{table}
%%%%%%%%%%%%%%%%%%%%%%%%%%%%%%%%%%%%%%%%%%%%%%%%

%%%%%%%%%%%%
%R \subsubsection{Binary codes satisfying~\eqref{eq:weakB1}--\eqref{eq:weakB3}
%R from $\ZZ_4$-codes}

Similar to Section~\ref{Subsec:Z4},
we consider linear $\ZZ_4$-codes $\cC$ of length $n=2^m$
satisfying the following conditions:
\begin{align}
\label{eq:Z4-weak1}
&\{(n_0(x)-n_2(x))^2\mid x\in \cC\}=\{0,a^2,b^2,n^2\},
\\
\label{eq:Z4-weak2}
&|\{x \in \cC \mid n_0(x)=n_2(x)\}|=4n-2,\\
%%%%%
\label{eq:Z4-weak3}
&\text{$\cC$  contains $ZRM(1,m)$ as a subcode},
\end{align}
where $a,b$ are odd integers with $0<a<b<n$.
Then $\phi(\cC)$ satisfies~\eqref{eq:weakB1}--\eqref{eq:weakB3}.
%
% Similar to Section~\ref{Subsec:Z4-16},
Our exhaustive computer search based on 
the method in Section~\ref{Subsec:Z4}
verified that
there exists no $\ZZ_4$-code 
satisfying~\eqref{eq:Z4-weak1}--\eqref{eq:Z4-weak3}
for lengths $8$ and $16$.

%%%%%%%%%%%%%%%%%%%%%%%%%%%%%%%%%%%%%%%%%%%%%%%%%%%%
\section{Some modification of 
weakly unbiased Hadamard matrices}
\label{sec:II}

Finally, some modification of the notion of
weakly unbiased Hadamard matrices is given.
We derive some results which are an analogy to those of quasi-unbiased 
Hadamard matrices and weakly unbiased Hadamard matrices.

%%%%%%%%%%%%%
\subsection{Type~II weakly unbiased Hadamard matrices}
Let $H,K$ be Hadamard matrices of order $n$.
Let $a_{ij}$ denote the $(i,j)$-entry of $HK^T$.
We say that $H,K$ are {\em Type~II weakly unbiased}
% if $a_{ij}\ne 0$ and $a_{ij} \equiv 0 \pmod 4$ for $1 \le i,j \le n$, and
% $|\{ |a_{ij}| \mid 1 \le i,j \le n\}| = 2$.
% if $\{ |a_{ij}| \mid i,j \in \{1,2,\ldots,n\}\}= \{a,b\}$,
% where $a,b \equiv 0 \pmod 4$.
if $a_{ij} \equiv 0 \pmod 4$ for $i,j \in \{1,2,\ldots,n\}$ and
$|\{ |a_{ij}| \mid i,j \in \{1,2,\ldots, n\}\}| \le 2$.
For an even square $n$,
a pair of unbiased Hadamard matrices of order $n$ is
a pair of Type~II weakly unbiased 
Hadamard matrices.
Hence, the notion of
Type~II weakly unbiased Hadamard matrices of order $n$ is
some natural extension of the notion of unbiased 
Hadamard matrices for an even square $n$.
Similar to weakly unbiased Hadamard matrices,
in this paper, 
we exclude unbiased Hadamard matrices
from Type~II weakly unbiased Hadamard matrices.
It follows immediately from the definition that $n \ge 8$.

%%%%%%%%%%%%%
\subsection{Basic properties and feasible parameters}
Let $(H,K)$ be a pair of Type~II weakly unbiased
Hadamard matrices of order $n$.
Suppose that $a,b$ are positive integers satisfying
$\{ |a_{ij}| \mid i,j \in \{1,2,\ldots,n\}\}=\{a,b\}$.
We denote the set $\{a,b\}$ by $\sigma(H,K)$.
Let $n(a)$ be the number of components $j$ with $a_{ij}=\pm a$ 
for $i=1,2,\ldots,n$, where
$a_{ij}$ denotes the $(i,j)$-entry of $HK^T$.
Similar to weakly unbiased Hadamard matrices,
\eqref{eq:weak} holds.
From now on, we assume that $a < b$.
We say that parameters $(a,b)$ satisfying~\eqref{eq:weak}
are {\em feasible}.  % for each order $n$.
Since $(a,b,n(a))=(4,n/2-4,n-4)$ satisfies~\eqref{eq:weak},
the parameters $(a,b)=(4,n/2-4)$ are feasible for each order $n \equiv 0 \pmod 8$.

%For $8 \le n \le 48$, we give in Table~\ref{Tab:weakIIP}
%the possible parameters $(a,b,n(a))$ satisfying $a < b$.
%The third column of the table indicates the existence of
%Type II weakly unbiased Hadamard matrices for $n$ and $(a,b,n(a))$.
For $n=4,8,\ldots,48$, we give in Table~\ref{Tab:weakIIP}
feasible parameters $(a,b)$ along with $n(a)$ and 
our present state of knowledge about 
the maximum size $f_{max}$
among sets of mutually Type~II weakly unbiased Hadamard matrices of order $n$
for $(a,b,n(a))$.
In the third column of the table, ``-'' means that there exists no pair of
Type~II weakly unbiased Hadamard matrices.
%R The last two columns provide references for
%R the lower and upper bounds on $f_{max}$.

%%%%%%%%%%%%%%%%%%%%%%%%%%%
\begin{comment}
\begin{table}[thb]
\caption{Type~II weakly unbiased Hadamard matrices  $(n=4,8,\ldots,48)$}
\label{Tab:weakIIP}
\begin{center}
%{\small
{\footnotesize
%{\scriptsize
%\begin{tabular}{c|c|l|ccccccc}
\begin{tabular}{c|c|c|l}
\noalign{\hrule height0.8pt}
$n$ &$(a,b,n(a))$ & $f_{max}$ & \multicolumn{1}{c}{Reference}\\
\hline
%8 & -- &  & \\
24 & $(4, 8, 20)$& Yes & Section~\ref{Subsec:weakIICon}\\ 
\hline
28 & $(4, 8, 21)$& \bf ?????? & \\ 
\hline
32 & $(4, 12, 28)$& Yes & Corollary~\ref{cor:weakII}, 
                    Section~\ref{Subsec:weakIIcode} \\
\hline
36 & $(4, 8, 21)$& \bf ?????? & \\ 
     & $(4, 16, 33)$& \bf ?????? & \\
\hline
40 & $(4, 8, 20)$& \bf ?????? & \\
     & $(4, 16, 36)$& \bf ?????? & \\
\hline
48 & $(4, 8, 16)$& \bf ?????? & \\
    & $(4, 12, 36)$& Yes & Proposition~\ref{prop:weakII}\\
    & $(4, 20, 44)$& Yes & Corollary~\ref{cor:weakII}\\
    & $(4, 28, 46)$& \bf ?????? & \\
\noalign{\hrule height0.8pt}
\end{tabular}
}
\end{center}
\end{table}
\end{comment}

\begin{table}[thb]
\caption{Type~II weakly unbiased Hadamard matrices  $(n=4,8,\ldots,48)$}
\label{Tab:weakIIP}
\begin{center}
%{\small
{\footnotesize
%{\scriptsize
%\begin{tabular}{c|c|l|ccccccc}
\begin{tabular}{c|c|c|l|l}
\noalign{\hrule height0.8pt}
%$n$ &$(l,a)$ & \shortstack{Upper and lower bounds\\ on the maximal value $f$} & \multicolumn{1}{c}{Reference}\\
$n$ &$(a,b,n(a))$ & $f_{max}$ & \multicolumn{2}{c}{Reference}\\
\hline
24 & $(4, 8, 20)$& $2-42$ & Section~\ref{Subsec:weakIICon}& Table~\ref{Tab:weakIIUB}\\ 
\hline
%28 & $(4, 8, 21)$&$\leq 51$  & & Table~\ref{Tab:weakIIUB}\\ 
28 & $(4, 8, 21)$& -  & & Section~\ref{Subsec:weakIICon}\\
\hline
32 & $(4, 12, 28)$& $4-264$ & Section~\ref{Subsec:weakIIcodeZ4code} &Table~\ref{Tab:weakIIUB}\\
\hline
36 & $(4, 8, 21)$& $2 - 72$ & Proposition~\ref{prop:weakII2} &Table~\ref{Tab:weakIIUB}\\ 
     & $(4, 16, 33)$& $\leq 10671$  & &Table~\ref{Tab:weakIIUB}\\
\hline
40 & $(4, 8, 20)$& $\leq 84$ & &Table~\ref{Tab:weakIIUB}\\
     & $(4, 16, 36)$& $\leq 16698$  & &Table~\ref{Tab:weakIIUB}\\
\hline
48 & $(4, 8, 16)$& $\leq 112$ & &Table~\ref{Tab:weakIIUB}\\
    & $(4, 12, 36)$& $2-194$  & Proposition~\ref{prop:weakII}&Table~\ref{Tab:weakIIUB}\\
    & $(4, 20, 44)$& $2-36034$ & Corollary~\ref{cor:weakII}& Table~\ref{Tab:weakIIUB}\\
    & $(4, 28, 46)$& $\leq36034$ & &Table~\ref{Tab:weakIIUB}\\
\noalign{\hrule height0.8pt}
\end{tabular}
}
\end{center}
\end{table}

%%%%%%%%%%%%%%%%%%%%%%%%%%%%%%%%%%%%%%%%%%%%%%%%

\begin{proposition}
\label{prop:weakII}
Suppose that there exists a set of $f$
mutually unbiased Hadamard matrices of order $m$.
Assume that one of the following holds:
\begin{itemize}
\item[\rm (i)]
There exists a set of $f$ mutually 
weakly unbiased Hadamard matrices $H_1,H_2,\ldots,H_f$
of order $n$ with $\sigma(H_i,H_j)=\{a,b\}$ $(i,j \in \{1,2,\ldots,f\}$
and $i \ne j$).
\item[\rm (ii)]
There exists a set of $f$ mutually 
Type~II weakly unbiased Hadamard matrices $H_1,H_2,\ldots,H_f$
of order $n$ with $\sigma(H_i,H_j)=\{a,b\}$ $(i,j \in \{1,2,\ldots,f\}$
and $i \ne j$).
\end{itemize}
Then 
there exists a set of $f$ mutually 
Type~II weakly unbiased Hadamard matrices $L_1,L_2,\ldots,L_f$ of
order $mn$ with $\sigma(L_i,L_j)=\{\sqrt{m}a,\sqrt{m}b\}$ 
$(i,j \in \{1,2,\ldots,f\}$ and $i \ne j$).
\end{proposition}
\begin{proof}
It is sufficient to give a proof for the case $f=2$.
Let $(H',K')$ be a pair of unbiased Hadamard matrices of order $m$.
Then 
$(H_1 \otimes H', H_2 \otimes K')$
is a pair of Type~II weakly unbiased Hadamard matrices of
order $mn$ with 
$\sigma(H_1 \otimes H', H_2 \otimes K')=\{\sqrt{m}a,\sqrt{m}b\}$.
\end{proof}

As an example, a pair of Type~II weakly unbiased Hadamard matrices 
$L_1,L_2$ of order $48$ with $\sigma(L_1,L_2)=\{4,12\}$
is constructed from 
a pair of weakly unbiased Hadamard matrices $H_1,K_1$ of
order $12$ with $\sigma(H_1,K_1)=\{2,6\}$ (see Table~\ref{Tab:weakP})
and a pair of unbiased Hadamard matrices of order $4$.

\begin{corollary}\label{cor:weakII}
If there exists a Hadamard matrix of order $n \ge 8$,
then there exists a pair of Type~II weakly unbiased Hadamard matrices 
$H,K$ of order $4n$ with $\sigma(H,K)=\{4,2n-4\}$.
\end{corollary}
\begin{proof}
Suppose that there exists a Hadamard matrix of order $n \ge 8$.
By Proposition~\ref{prop:weak}, 
there exists a pair of weakly unbiased Hadamard matrices 
$H,K$ of order $n$ with $\sigma(H,K)=\{2,n-2\}$.
Since there exists a pair of unbiased Hadamard matrices of order $4$,
the result follows from Proposition~\ref{prop:weakII}.
\end{proof}

Similar to Proposition~\ref{prop:weak2}, we 
immediately have the following:

\begin{proposition}\label{prop:weakII2}
Suppose that $n=4k^2$, where $k$ is odd.
If there exists a pair of unbiased Hadamard matrices of order $n$, 
then there exists a pair of Type~II 
weakly unbiased Hadamard matrices $H,K$ of order $n$ with 
$\sigma(H,K)=\{\sqrt{n}-2,\sqrt{n}+2\}$.
\end{proposition}

As an example,
a pair of Type~II weakly unbiased Hadamard matrices $H,K$ of order 
$36$ with $\sigma(H,K)=\{4,8\}$ is constructed from 
that of unbiased Hadamard matrices of order $36$ given in~\cite{HKO}.

%%%%%%%%%%%%%%%%%%%%%%%%
\subsection{Observations by straightforward construction}
\label{Subsec:weakIICon}

%R The observations given in Section~\ref{Subsec:Const} may be
%R applied to  weakly unbiased Hadamard matrices.

%R We denote by $H_{24,4}$ \verb+had.24.49+ in~\cite{Hadamard},
%R which is a Hadamard matrix of order $24$.
% By an argument similar to that given in Section~\ref{Subsec:Const},
Under the condition~\eqref{eq:3col} on $K$, 
our computer search found a Hadamard matrix $\overline{K_{24,4}}$ such that
$(H_{24,4}, \overline{K_{24,4}})$ is a pair of
Type~II weakly unbiased Hadamard matrices with
$\sigma(H_{24,4}, \overline{K_{24,4}})=\{4,8\}$,
where $H_{24,4}$ is \verb+had.24.49+ in~\cite{Hadamard}.
The matrix $\overline{K_{24,4}}$ is listed in Figure~\ref{Fig:24weakII}.

%%%%%%%%%%%%%%%%%  Fig  %%%%%%%%%%%%%%%%%
\begin{figure}[thb]
\centering
{\scriptsize
\[
\overline{K_{24,4}}=
\left( \begin{array}{c}
+++-+-+----+-++-----+-+-\\
+++--+-+-+--+-+-+-+++++-\\
+++-+-+----++--+++++-+-+\\
++++-++++-++-+--++++--+-\\
++++-++++-+++-++----++-+\\
+++--+-+-+---+-+-+-----+\\
++-+--+--++-++--++--++++\\
++-+--+--++---++--++----\\
++-++--++----++++++-++--\\
++--++--+++++-++++----+-\\
++--++--++++-+----++++-+\\
++-++--++---+------+--++\\
+-++----++-+----+----+--\\
+-+-+-+++++-+++-+-+----+\\
+-+-+-+++++----+-+-++++-\\
+-++++----+-+----++-+---\\
+-++----++-+++++-++++-++\\
+-++++----+--++++--+-+++\\
+------+--++--+--++--+++\\
+--+++++-+-++++--+-+-+--\\
+----++-+---++-+--+--++-\\
+----++-+-----+-++-++--+\\
+--+++++-+-+---++-+-+-++\\
+------+--++++-++--++---\\
\end{array}
\right)
\]
}
\caption{The matrix $\overline{K_{24,4}}$}
\label{Fig:24weakII}
\end{figure}
%%%%%%%%%%%%%%%%%  Fig  %%%%%%%%%%%%%%%%%

%R Similar to Section~\ref{Subsec:weakCon}, we consider
%R $\Gamma(H_{28,i},\{4,8\})$ to determine whether there exists no pair 
%R of Type~II weakly unbiased Hadamard matrices of order $28$ 
%R with $\sigma(H,K)=\{4,8\}$.
By a computer calculation, we verified that
each of the four induced subgraphs on $V_j$ $(j=1,2,3,4)$ of 
$\Gamma(H_{28,i},\{4,8\})$ 
contains a $7$-clique for $355$ Hadamard matrices.  
In addition, we verified that
the induced subgraph on $V_1 \cup V_2$ 
of only $\Gamma(H_{28,484},\{4,8\})$ 
contains a $14$-clique among the 355 graphs $\Gamma(H_{28,i},\{4,8\})$.
By a computer calculation, we obtained that
the size of the maximum cliques of $\Gamma(H_{28,484},\{4,8\})$
is $24$.
Hence, there exists no pair 
of Type~II weakly unbiased Hadamard matrices of order $28$ 
with $\sigma(H,K)=\{4,8\}$.

%%%%%%%%%%%%
\subsection{A coding-theoretic approach}
\label{Subsec:weakIIcode}

Similar to Theorems~\ref{thm:F2} and~\ref{thm:weakF2},
we have the following coding-theoretic approach to 
Type~II weakly unbiased Hadamard matrices.

\begin{theorem}\label{thm:weakIIF2}
Let $a, b$ be even integers with $0<a < b < n/2$.
There exists  a self-complementary
$(n,2fn)$ code $C$ satisfying the following conditions:
\begin{align}
\label{eq:weakIIB1}
&\{i \in \{0,1,\ldots,n\} \mid A_i(C) \ne 0\}=\{0,n/2 \pm a,n/2 \pm b,n/2,n\},\\
\label{eq:weakIIB2}
&A_{n/2}(C)=2n-2, \\
\label{eq:weakIIB3}
&C=C_1 \cup C_2 \cup \cdots\cup C_f,
\end{align}
where 
% $C_i$ has the same distance distribution as $RM(1,m)$.
$C_i$ has distance distribution $(A_0(C_i),A_{n/2}(C_i),A_n(C_i))=(1,2n-2,1)$
if and only if
there exists a set of $f$ Type~II weakly unbiased Hadamard matrices $H,K$ with 
$\sigma(H,K)=\{2a,2b\}$.
\end{theorem}

Similar to Lemma~\ref{lem:bound}, 
as the case $s=6$ of Theorems~\ref{thm:SCbound} and~\ref{thm:SCLP},
we have two upper bounds
on the number of the codewords of
self-complementary codes satisfying~\eqref{eq:weakIIB1}.

\begin{lemma}\label{lem:weakIIbound}
Let $C$ be a self-complementary code of length $n$ 
satisfying~\eqref{eq:weakIIB1}.
Then
\begin{enumerate}[\rm (i)]
\item $|C|\leq 2(\binom{n}{5}+\binom{n}{3}+\binom{n}{1})$.
\item If 
$15n^2-30n+16-4(3n-2)(a^2+b^2)+16a^2b^2>0$ and
$5(n-2)-2a^2-2b^2\ge 0$,
then $|C|\leq \lfloor
\frac{2n(n^2-4a^2)(n^2-4b^2)}{15n^2-30n+16-4(3n-2)(a^2+b^2)+16a^2b^2}
\rfloor$. 
%If equality holds, then a pair $(C,\{R_i\}_{i=0}^4)$ is a $Q$-polynomial association scheme, where $R_i=\{(x,y)\mid x,y\in C, d(x,y)=\beta_i\}$ and $\{i\mid A_i\neq0\}=\{\beta_0,\beta_1,\ldots,\beta_4\}$ with $\beta_0<\beta_1<\cdots<\beta_4$.
\end{enumerate}
\end{lemma}
\begin{proof}
(i)
% $Since  $C$ is a self-complementary code of length $n$ and degree $5$,
The upper bound is the case $s=6$ of Theorem~\ref{thm:SCbound}.

(ii)
Expanding by the Krawtchouk polynomials, we have 
\begin{align*}
\overline{\alpha}_C(z)=
&\Big(1-\frac{2z}{2a+n}\Big)\Big(1-\frac{2z}{2b+n}\Big)\Big(1-\frac{2z}{n}\Big)
\Big(1-\frac{2z}{-2a+n}\Big)\Big(1-\frac{2z}{-2b+n}\Big)
\\
=&\frac{15n^2-30n+16-4(3n-2)(a^2+b^2)+16a^2b^2}{n(n^2-4a^2)(n^2-4b^2)}K_1(z)\\
&+\frac{12(5n-10-2a^2-2b^2)}{n(n^2-4a^2)(n^2-4b^2)}K_3(z)+\frac{120}{n(n^2-4a^2)(n^2-4b^2)}K_5(z)\\
=&\alpha_1K_1(z)+\alpha_3K_3(z)+\alpha_5K_5(z) \text{ (say)}.
\end{align*}
The assumption on $a,b$ and $n$ yields that
$\alpha_1$ is positive and $\alpha_3,\alpha_5$ are nonnegative. 
Therefore, Theorem~\ref{thm:SCLP} implies the desired bound.
\end{proof}

% \begin{question}
% Improve the upper bounds given in Theorem~\ref{thm:weakIIbound}.
% \end{question}

% These bounds imply two upper bounds
% on the size of a set of mutually
% Type~II weakly unbiased Hadamard matrices of order $n$.
% 
%One of the bounds 
%depends only on $n$ but not on $l,a$, the other depends on $n,l,a$.  
%The bounds are referred to as the absolute bounds and
%the linear programming bounds, respectively.
% The bounds are referred to as the absolute bounds and
% the linear programming bounds, respectively.

Similar to Theorem~\ref{thm:bound}, 
by Theorem~\ref{thm:weakIIF2}, 
we immediately have the following two upper bounds on 
the maximum size among sets of 
mutually
Type~II weakly unbiased Hadamard matrices,
one of which
depends only on $n$, and
the other depends on $n,\alpha$.  
% The bounds are referred to as the absolute bounds and
% the linear programming bounds, respectively.
This is also one of the main results of this paper.

\begin{theorem}\label{thm:weakIIbound}
Suppose that there exists a set of $f$ mutually
Type~II weakly unbiased Hadamard matrices $H,K$ of order $n$
with $\sigma(H,K)=\{a,b\}$.
Then
\begin{enumerate}[\rm (i)]
\item $f \leq \lfloor \frac{n^4- 10n^3 + 55n^2 - 110n +184}{5!}\rfloor $. 
%InputForm[Expand[
%(5*4*3*2)*(((n*(n-1)*(n-2)*(n-3)*(n-4))/(5*4*3*2)+(n*(n-1)*(n-2))/(3*2)+n))
%]]
\item If 
$15n^2-30n+16-4(3n-2)(a^2+b^2)+16a^2b^2>0$ and
$5(n-2)-2a^2-2b^2\ge 0$,
then $f \leq \lfloor 
\frac{(n^2-4a^2)(n^2-4b^2)}{15n^2-30n+16-4(3n-2)(a^2+b^2)+16a^2b^2}
\rfloor$. 
\end{enumerate}
\end{theorem}

For the feasible parameters given in Table~\ref{Tab:weakIIP}, 
we list in Table~\ref{Tab:weakIIUB}
the maximum possible sizes among sets of mutually Type~II weakly unbiased 
Hadamard  matrices, which are obtained by the
two upper bounds.
We do no list the maximum possible sizes when
there exists no pair of Type~II weakly unbiased Hadamard matrices.
In the table, ``$*$'' means that the assumption of 
Theorem~\ref{thm:weakIIbound} (ii) is not satisfied.

\begin{table}[thb]
\caption{Absolute and linear programming bounds in Theorem~\ref{thm:weakIIbound}}
\label{Tab:weakIIUB}
\begin{center}
%{\small
{\footnotesize
%{\scriptsize
%\begin{tabular}{c|c|l|ccccccc}
\begin{tabular}{c|c|c|c}
\noalign{\hrule height0.8pt}
$n$ &$(a,b,n(a))$ & Absolute bound  & Linear programming bound \\
\hline
24 & $(4, 8, 20)$& $\lfloor5569/3\rfloor=1856$  & $\lfloor 256/3\rfloor=85$
\\ 
\hline
%% 28 & $(4, 8, 21)$& $3628$  & $\lfloor256/5\rfloor=51$\\ 
%% \hline
32 & $(4, 12, 28)$& $6449$ & $528$\\
\hline
36 & $(4, 8, 21)$& $\lfloor32014/3\rfloor=10671$ &  $\lfloor 5632/39\rfloor=144$
\\ 
     & $(4, 16, 33)$& $10671$  & $*$ \\
\hline
40 & $(4, 8, 20)$& $\lfloor83491/5\rfloor=16698$ &  $\lfloor 4224/25\rfloor=168$
 \\
     & $(4, 16, 36)$& $16698$  & $*$ \\
\hline
48 & $(4, 8, 16)$& $\lfloor108103/3\rfloor=36034$ &  $\lfloor 64064/285\rfloor=224$ \\
    & $(4, 12, 36)$& $36034$ &  $\lfloor 20592/53\rfloor=388$ \\
    & $(4, 20, 44)$& $36034$ & $*$\\
    & $(4, 28, 46)$& $36034$ & $*$ \\
\noalign{\hrule height0.8pt}
\end{tabular}
}
\end{center}
%* means that the assumption of  Proposition~\ref{prop:bound} is not satisfied. 
%Absolute bounds depend only on $n$.
\end{table}

% We give the number $N'_2(C_1)$ of inequivalent $(n,4n)$ codes 
% $C=C_1 \cup (u +C_1)$
% satisfying~\eqref{eq:weakIIB1} and~\eqref{eq:weakIIB2}.
% 
% \begin{proposition}
% $N'_2(C(H_{24,1}))=0$ and 
% $N'_2(C(H_{24,2}))=0$.
% \end{proposition}

\subsection{Binary codes satisfying~\eqref{eq:weakIIB1}--\eqref{eq:weakIIB3}}
\label{Subsec:weakIIcodeZ4code}
In the process of the classification of codes of length $32$
satisfying~\eqref{eq:weakB1}--\eqref{eq:weakB3}
with $C_1=RM(1,5)$,
our exhaustive computer search verified that
there exists no $[32,7]$ code 
satisfying~\eqref{eq:weakIIB1}--\eqref{eq:weakIIB3}
with $C_1=RM(1,5)$.
% Hence, we have the following:
% 
% \begin{proposition}
% $N'_2(RM(1,5))=0$.
% \end{proposition}

Suppose that $\cC$ is a linear $\ZZ_4$-code   of length $n=2^m$
satisfies the following conditions:
\begin{align}
\label{eq:Z4-weakII1}
&\{(n_0(x)-n_2(x))^2\mid x\in \cC\}=\{0,a^2,b^2,n^2\},
\\
\label{eq:Z4-weakII2}
&|\{x \in \cC \mid n_0(x)=n_2(x)\}|=4n-2,\\
%%%%%
\label{eq:Z4-weakII3}
&\text{$\cC$  contains $ZRM(1,m)$ as a subcode},
\end{align}
where $a,b$ are even integers with $0<a<b<n$.
Then $\phi(\cC)$ satisfies~\eqref{eq:weakIIB1}--\eqref{eq:weakIIB3}.

In the process of verifying that there exists no linear $\ZZ_4$-code of 
length $16$ satisfying~\eqref{eq:Z4-weak1}--\eqref{eq:Z4-weak3},
our computer calculation completed the classification of linear $\ZZ_4$-code of 
length $16$ satisfying~\eqref{eq:Z4-weakII1}--\eqref{eq:Z4-weakII3}.
We give the numbers
$N'_4(16,k)$ of inequivalent linear $\ZZ_4$-codes $\cC$ of length $16$ with 
$|\cC|=2^k$
satisfying~\eqref{eq:Z4-weakII1}--\eqref{eq:Z4-weakII3}.

\begin{proposition}
$N'_4(16,7)=1$,
$N'_4(16,8)=3$ and 
$N'_4(16,9)=0$.
\end{proposition}

The unique linear $\ZZ_4$-code $\cC=\cC'_{16,1}$ with $|\cC|=2^7$
is constructed as
$\langle ZRM(1,4), x_1 \rangle$,
where 
$x_1= (0, 0, 0, 2, 1, 1, 1, 3, 1, 3, 1, 3, 0, 0, 0, 0)$.
The three linear $\ZZ_4$-codes $\cC=\cC'_{16,2,i}$ with $|\cC|=2^8$
$(i=1,2,3)$ are constructed as
$\langle ZRM(1,4), x_1,x_{2,i} \rangle$,
where 
\begin{align*}
x_{2,1}=&( 0, 0, 1, 3, 0, 2, 3, 3, 1, 3, 0, 0, 1, 3, 0, 0),\\
x_{2,2}=&( 0, 0, 1, 3, 0, 0, 3, 3, 1, 3, 2, 0, 1, 3, 0, 0),\\
x_{2,3}=&( 0, 0, 1, 1, 0, 0, 3, 3, 1, 3, 0, 2, 3, 3, 0, 0).
\end{align*}
This gives a set of four mutually 
Type~II weakly unbiased Hadamard matrices of order $32$
with $\sigma(H, K)=\{4,12\}$
by Theorem~\ref{thm:weakIIF2}.
By a computer calculation, we
verified that
the above linear $\ZZ_4$-codes $\cC$ have $(a^2,b^2)=(4, 36)$
and have $(d_H(\cC),d_L(\cC))=( 8, 10)$.

%%%%%%%%%%%%%%%%%%%%%%%%%%%%%%%%%%%%%%%%%%%%%%%%%%%%%%
% \section*{Acknowledgments}
\bigskip
\noindent {\bf Acknowledgments.}
The authors would like to thank the anonymous referee for useful comments.
In this work, the supercomputer of ACCMS, Kyoto University was partially used.
This work is supported by JSPS KAKENHI Grant Numbers 23340021,
26610032.

%%%%%%%%%%%%%%%%% Appendix %%%%%%%%%%%%%%%%%
%%%%%%%%%%%%%%%%% Appendix %%%%%%%%%%%%%%%%%
%%%%%%%%%%%%%%%%% Appendix %%%%%%%%%%%%%%%%%
%%%%%%%%%%%%%%%%% Appendix %%%%%%%%%%%%%%%%%
\appendix
\def\thesection{Appendix \Alph{section}}

\section{} \label{appendix}

In Lemma~\ref{lem:bound},
we did not give a detail proof of the fact that
$(C,\{R_i\}_{i=0}^4)$ is a $Q$-polynomial association scheme
when $|C|=\frac{2n(n^2-4\alpha^2)}{3n-2-4\alpha^2}$.
In this appendix, we give a detailed proof.

Suppose that $C$ and $R_i$ are as given in Lemma~\ref{lem:bound}.
Assume that $|C|=\frac{2n(n^2-4\alpha^2)}{3n-2-4\alpha^2}$.
For $i\in\{0,1,\ldots,4\}$, $A_i$ denotes the adjacency matrix of 
the graph with vertex set $C$ and edge set $R_i$.
Let $\mathcal{A}$ denote the
vector space over $\mathbb{R}$ spanned by 
$A_0=I_{|C|},A_1,\ldots,A_4$, which forms an algebra.
Let $\{E_0,E_1,\ldots,E_4\}$ denote the set of the primitive idempotents
of $\mathcal{A}$.
Then the matrix $P=(p_{ij})$ is defined by 
$A_i=\sum_{j=0}^4 p_{ji}E_i$.

\def\thesection{A}
\begin{lemma}
$(C,\{R_i\}_{i=0}^4)$ is a symmetric association scheme.
\end{lemma}
\begin{proof}
It is sufficient to show that $(C,\{R_i\}_{i=0}^4)$ satisfies the
$4$-th condition in the definition of a symmetric association scheme given in
 Section~\ref{sec:AS},
namely, 
we show that $A_iA_j\in\mathcal{A}$ for $i,j\in\{0,1,\ldots,4\}$.
Since $A_0=I_{|C|}$, $A_iA_0=A_0A_i=A_{i}$ holds for $i\in\{0,1,\ldots,4\}$.
Since $C$ is self-complementary,  
$A_iA_4=A_4A_i=A_{4-i}$ holds for $i\in\{0,1,\ldots,4\}$.

% Assume that $C$ is a self-complementary code satisfying $|C|=\frac{2n(n^2-4\alpha^2)}{3n-4\alpha^2-2}$. 
Since $|C|=\frac{2n(n^2-4\alpha^2)}{3n-4\alpha^2-2}$,
the coefficients of $K_0(z)$ and $K_1(z)$ in $\alpha_C(z)$ are
$1$ and the other coefficients are positive by the calculation in the proof  of Theorem~\ref{thm:SCLP}.
By~\cite[Theorem~5.23 (iii)]{D}, 
$C$ is a $5$-design in the binary Hamming scheme, namely, 
$C$ is an orthogonal array of strength $5$.  

%Let $z^\lambda=\sum_{l=0}^\lambda f_{\lambda,l}K_l(z)$ be the Krawtchouk expansion of $z^\lambda$, and define a polynomial by $F_{\lambda,\mu}(z)=\sum_{l=0}^{\min\{\lambda,\mu\}}f_{\lambda,l}f_{\mu,l}K_l(z)$.  
We denote the Krawtchouk expansion of $z^\lambda$ by
$z^\lambda=\sum_{l=0}^\lambda f_{\lambda,l}K_l(z)$, 
and define a polynomial by 
$F_{\lambda,\mu}(z)=\sum_{l=0}^{\min\{\lambda,\mu\}}f_{\lambda,l}f_{\mu,l}K_l(z)$.  
For $\lambda,\mu\in\{0,1,2\}$,   
expand $(\sum_{k=0}^{\lambda}f_{\lambda,k}G_kG_k^T)(\sum_{l=0}^{\mu}f_{\mu,l}G_lG_l^T)$ in two ways, 
where $G_k$ denotes the $k$-th characteristic matrix of $C$. 
In the following calculation, define $0^0$ to be $1$. 
By \cite[Theorem~5.18]{D}, 
\begin{multline*}
(\sum_{k=0}^{\lambda}f_{\lambda,k}G_kG_k^T)(\sum_{l=0}^{\mu}f_{\mu,l}G_lG_l^T)=|C|\sum_{k=0}^{\min\{\lambda,\mu\}}f_{\lambda,k}f_{\mu,k}G_kG_k^T \\
=|C|\sum_{k=0}^{\min\{\lambda,\mu\}}f_{\lambda,k}f_{\mu,k}\sum_{l=0}^4K_k(\beta_l)A_l\notag
=|C|\sum_{l=0}^{4}F_{\lambda,\mu}(\beta_l)A_l.
\end{multline*}

On the other hand, by \cite[Theorem~3.13]{D}, 
\begin{align*}
(\sum_{k=0}^{\lambda}&f_{\lambda,k}G_kG_k^T)(\sum_{l=0}^{\mu}f_{\mu,l}G_lG_l^T)
=(\sum_{k=0}^{\lambda}f_{\lambda,k}\sum_{i=0}^4K_k(\beta_i)A_i)(\sum_{l=0}^{\mu}f_{\mu,l}\sum_{j=0}^4K_l(\beta_j)A_j)
\displaybreak[0]\\
&=\sum_{k=0}^{\lambda}\sum_{l=0}^{\mu}\sum_{i=0}^4\sum_{j=0}^4f_{\lambda,k}f_{\mu,l}K_k(\beta_i)K_l(\beta_j)A_iA_j
=\sum_{i=0}^4\sum_{j=0}^4\beta_i^{\lambda}\beta_j^{\mu}A_iA_j\displaybreak[0]\\
&=\sum_{i=1}^3\sum_{j=1}^3\beta_i^{\lambda}\beta_j^{\mu}A_iA_j+\sum_{i=1}^3\beta_i^\lambda\beta_0^\mu A_i+\sum_{j=1}^3\beta_0^\lambda\beta_j^\mu A_j
+\sum_{i=1}^3\beta_i^\lambda\beta_4^\mu A_{4-i}\\
&\quad
+\sum_{j=1}^3\beta_4^\lambda\beta_j^\mu A_{4-j}+
\beta_0^{\lambda+\mu}A_0+
\beta_4^{\lambda+\mu}A_0+\beta_0^\lambda \beta_4^\mu A_4+\beta_4^\lambda \beta_0^\mu A_4.
\end{align*}
Thus, $\sum_{i=1}^3\sum_{j=1}^3\beta_i^{\lambda}\beta_j^{\mu}A_iA_j
\in \mathcal{A}$ for $i,j\in\{1,2,3\}$.
%Set $W=(\beta_k^l)_{k,l=0}^2$.
For $W=\left(\begin{smallmatrix}1&1&1\\
\beta_1&\beta_2&\beta_3\\\beta_1^2&\beta_2^2&\beta_3^2\end{smallmatrix}\right)$,
$W\otimes W$ is invertible.
Hence, $A_iA_j \in \mathcal{A}$ for $i,j\in\{1,2,3\}$.  
Therefore,  
$(C,\{R_i\}_{i=0}^4)$ is a symmetric association scheme. 
\end{proof}

\begin{lemma}
$(C,\{R_i\}_{i=0}^4)$ is $Q$-polynomial.
\end{lemma}
\begin{proof}
Set $F_i=\frac{1}{|C|}G_iG_i^T$ for $i=0,1,2,3$ and
$F_4=I-\sum_{i=0}^2F_i$.
We claim that $\{F_0,F_1,F_2\}$ is a subset of the set of primitive idempotents of $\mathcal{A}$. 
Let $E_i$ ($i = 0,1,\ldots,4$) be primitive idempotents. 
Assume that $F_i\not\in\{E_0,E_1,\ldots,E_4\}$ for some $i\in \{0,1,2\}$. 
Since $F_i$ is an idempotent, we have decomposition $F_i=E+E'$ satisfying $E,E'\neq O$, $E^2=E$, $E'^2=E'$ and $EE'=O$, where $O$ denotes the the zero matrix. 
Then $\{F_0,F_1,F_2,E,E'\}\setminus\{F_i\}$ is a set of elements which are linear independent. 
Thus, $\langle F_3,F_4 \rangle$ has dimension $1$. 
Hence, there exists a nonzero real number $c$ such that $F_4=cF_3$. 
Then 
$
A_0-\frac{1}{|C|}\sum_{i=0}^{2}\sum_{j=0}^4 K_{i}(\beta_j)A_j=\frac{c}{|C|}\sum_{j=0}^4K_{3}(\beta_j)A_j,  
$
and thus we obtain 
$
cK_{3}(\beta_j)+\sum_{i=0}^2K_i(\beta_j)=0
$
for any $j\in\{1,2,3,4\}$.
Since all $\beta_j$ are distinct and the degree of $cK_3(z)+\sum_{i=0}^2K_i(z)$ is at most three, this is a contradiction. 
Therefore, we may assume $E_i=F_i$ for $i=0,1,2$. 

For $i=0,1,2$,
\begin{align*}
A_4E_i&=\frac{1}{|C|}\sum_{j=0}^4 K_i(\beta_j)A_4A_j
=\frac{1}{|C|}\sum_{j=0}^4 K_i(\beta_j)A_{4-j}=\frac{1}{|C|}\sum_{j=0}^4 K_i(\beta_{4-j})A_{j}\\
&=\frac{1}{|C|}\sum_{j=0}^4 K_i(n-\beta_{j})A_{j}=\frac{1}{|C|}\sum_{j=0}^4(-1)^i K_i(\beta_{j})A_{j}=(-1)^iE_i.
\end{align*}
Thus, $p_{i4}=(-1)^i$ for $i=0,1,2$.
By \cite[Chap.~II, Theorem~4.1~(ii)]{BI}, 
$\{p_{04},p_{14},\ldots,p_{44}\}=
\{\gamma_0,\gamma_1,\ldots,\gamma_4\}$
as a multiset, where
$\gamma_i$ $(i=0,1,\ldots,4)$ are the eigenvalues of
the matrix
$\left(\begin{smallmatrix}0&0&0&0&1\\0&0&0&1&0\\0&0&1&0&
0\\0&1&0&0&0\\1&0&0&0&0\end{smallmatrix}\right)$.
%% $1,-1$ with multiplicities $3,2$, respectively.
Thus, we may assume that $p_{34}=-1$ and $p_{44}=1$.

By \cite[Lemma~2.3.1~(vii)]{BCN}, 
\begin{align}\label{eq:as1}
q_{0i}q_{0j}=\sum_{k=0}^4 q_{i,j}^k q_{0k},\quad
q_{4i}q_{4j}=\sum_{k=0}^4 q_{i,j}^k q_{4k}.
\end{align}
By \cite[Chap.~II, Theorem~3.5~(i)]{BI} and $p_{i4}=(-1)^i$,
 $q_{4i}=(-1)^i q_{0i}$ for $i=0,1,\ldots,4$. 
Substituting these into \eqref{eq:as1}, we obtain 
\begin{align}\label{eq:as3}
(-1)^{i+j}q_{0i}q_{0j}&=\sum_{k=0}^4 q_{i,j}^k (-1)^k q_{0k}.
\end{align}
By \eqref{eq:as1} and \eqref{eq:as3}, we have $\sum_{k=0}^4(1-(-1)^{i+j+k})q_{i,j}^k q_{0k}=0$. 
Since $q_{0k}>0$ and $q_{i,j}^k\geq0$, we obtain 
\begin{align}\label{eq:as4}
q_{i,j}^k=0 \text{ if $i+j+k$ is odd}. 
\end{align}

For $i=0,1$, 
\begin{align*}
|C|E_1\circ |C|E_i&=\sum_{l=0}^4 K_1(\beta_l)K_i(\beta_l)A_l\\
&=\sum_{l=0}^4 (n-i+1)K_{i-1}(\beta_l)A_l+\sum_{l=0}^4 (i+1)K_{i+1}(\beta_l)A_l\\
&=(n-i+1)|C|E_{i-1}+(i+1)|C|E_{i+1}.
\end{align*}
Thus, $q_{1,i}^{i-1}=n-i+1$ and $q_{1,i}^{i+1}=i+1$ for $i=0,1$, and $q_{1,i}^j=0$ for $i=0,1,j\neq i-1,i+1$.   
By \cite[Chap.~II, Proposition~3.7~(v)]{BI}, $q_{1,2}^1=n-1$, and 
by \cite[Chap.~II, Proposition~3.7~(vi)]{BI}, $q_{1,i}^j=0$ for $(i,j)\in\{(2,0),(3,0),(4,0),(3,1),(4,1)\}$. 
By \eqref{eq:as4},  $q_{1,i}^j=0$ for $(i,j)\in\{(2,2),(3,3),$ $(4,4),(2,4),(4,2)\}$. 
Again by \cite[Chap.~II, Proposition~3.7~(vi)]{BI}, $q_{1,i}^j>0$ for  $(i,j)\in\{(2,3),(3,2),(3,4),(4,3)\}$. 
%Therefore, $B_1^*$ is a tridiagonal matrix with nonzero superdiagonal and subdiagonal entries. 
The Q-polynomiality is equivalent to the condition that the
Krein matrix $B_1^*=(q_{1,j}^k)$ is a tridiagonal matrix with nonzero
entries on the superdiagonal and the subdiagonal (see \cite[p.~193]{BI}).
This completes the proof of the fact that the association scheme $(C,\{R_i\}_{i=0}^4)$ is $Q$-polynomial.
\end{proof}

\end{document}